\documentclass[a4paper,12pt]{article}

\usepackage{amsmath,amssymb,amsthm,amscd}

\newcommand{\Fg}{\mathfrak{g}}
\newcommand{\Fh}{\mathfrak{h}}
\newcommand{\Fgl}{\mathfrak{gl}}
\newcommand{\Fsl}{\mathfrak{sl}}

\newcommand{\CL}{\mathcal{L}}
\newcommand{\CB}{\mathcal{B}}
\newcommand{\CT}{\mathcal{T}}
\newcommand{\CQ}{\mathcal{Q}}
\newcommand{\CR}{\mathcal{R}}

\newcommand{\BC}{\mathbb{C}}
\newcommand{\BR}{\mathbb{R}}

\newcommand{\BZ}{\mathbb{Z}}
\newcommand{\BB}{\mathbb{B}}

\newcommand{\eps}{\mathbf{e}}

\newcommand{\Par}{\mathcal{P}}
\newcommand{\bb}{\mathbf{b}}
\newcommand{\LR}{\mathsf{LR}}
\newcommand{\XLRm}[4]{X_{#1,\,#2}^{#3;\,#4}}
\newcommand{\BLRm}[4]{B_{#1,\,#2}^{#3;\,#4}}
\newcommand{\CLRm}[4]{C_{#1,\,#2}^{#3;\,#4}}
\newcommand{\DLRm}[4]{D_{#1,\,#2}^{#3;\,#4}}

\newcommand{\Hom}{\mathop{\rm Hom}\nolimits}
\newcommand{\wt}{\mathop{\rm wt}\nolimits}
\newcommand{\Bij}{\mathop{\rm Bij}\nolimits}

\newcommand{\dist}{\mathop{\rm dist}\nolimits}
\newcommand{\Dep}{\mathop{\rm dep}\nolimits}
\newcommand{\Supp}{\mathop{\rm Supp}\nolimits}
\newcommand{\sgn}{\mathop{\rm sgn}\nolimits}
\newcommand{\ch}{\mathop{\rm ch}\nolimits}

\newcommand{\EW}{P^{[1,\,\infty)}_{-}}
\newcommand{\Ldom}[1]{P^{[#1]}_{+}}

\newcommand{\Pn}{P^{[n]}_{+}(\lambda,\,\mu)}
\newcommand{\Po}{P^{[n+1]}_{+}(\lambda,\,\mu)}
\newcommand{\Pm}[1]{P^{[#1]}_{+}(\lambda,\,\mu)}

\newcommand{\Fgn}{\Fg_{[n]}}

\newcommand{\Fgm}[1]{\Fg_{[#1]}}

\newcommand{\Vn}{V_{[n]}}
\newcommand{\Vo}{V_{[n+1]}}
\newcommand{\Vm}[1]{V_{[#1]}}

\newcommand{\Wn}{W_{[n]}}
\newcommand{\Wo}{W_{[n+1]}}
\newcommand{\Wm}[1]{W_{[#1]}}

\newcommand{\CBn}{\CB_{[n]}}
\newcommand{\CBo}{\CB_{[n+1]}}

\newcommand{\BBn}{\BB_{[n]}}
\newcommand{\BBo}{\BB_{[n+1]}}
\newcommand{\BBm}[1]{\BB_{[#1]}}

\newcommand{\lamn}{\lambda_{[n]}}

\newcommand{\lamm}[1]{\lambda_{[#1]}}

\newcommand{\mun}{\mu_{[n]}}
\newcommand{\muo}{\mu_{[n+1]}}
\newcommand{\mum}[1]{\mu_{[#1]}}

\newcommand{\nun}{\nu_{[n]}}
\newcommand{\nuo}{\nu_{[n+1]}}
\newcommand{\num}[1]{\nu_{[#1]}}

\newcommand{\xin}{\xi_{[n]}}
\newcommand{\xio}{\xi_{[n+1]}}

\newcommand{\Jn}[1]{#1^{J,\,n}}
\newcommand{\Jo}[1]{#1^{J,\,n+1}}
\newcommand{\Jl}[1]{#1^{J,\,\ell}}

\newcommand{\BBhi}[1]{\BB^{[#1]}_{\mathrm{max}}}
\newcommand{\BBhiw}[2]{\BB^{[#1]}_{\mathrm{max},\,#2}}

\newcommand{\col}[1]{\lambda^{(#1)}}
\newcommand{\com}[1]{\mu^{(#1)}}
\newcommand{\con}[1]{\nu^{(#1)}}
\newcommand{\cox}[1]{\xi^{(#1)}}
\newcommand{\co}[2]{#1^{(#2)}}

\newcommand{\abs}[1]{\langle #1 \rangle}
\newcommand{\pair}[2]{\langle #1,\,#2 \rangle}

\newcommand{\ti}[1]{\widetilde{#1}}
\newcommand{\ud}[1]{\underline{#1}}
\newcommand{\ol}[1]{\overline{#1}}

\newcommand{\ve}{\varepsilon}
\newcommand{\vp}{\varphi}

\newcommand{\bzero}{{\bf 0}}

\makeatletter
\@addtoreset{equation}{section}

\renewcommand\section{\@startsection{section}{1}{0pt}
{-3.5ex plus -1ex minus -.2ex}{1.0ex plus .2ex}{\large\bf}}
\renewcommand\subsection{\@startsection{subsection}{1}{0pt}
{2.5ex plus 1ex minus .2ex}{-1em}{\bf}}
\makeatother

\newcommand{\vsp}{\vspace{3mm}}

\theoremstyle{plain}
\newtheorem{thm}{Theorem}[section]
\newtheorem{lem}[thm]{Lemma}
\newtheorem{prop}[thm]{Proposition}
\newtheorem{cor}[thm]{Corollary}

\newtheorem{claim}{Claim}[thm]

\newtheorem*{claim*}{Claim}

\newcommand{\bqed}{\quad \hbox{\rule[-0.5pt]{3pt}{8pt}}}

\theoremstyle{definition}
\newtheorem{dfn}[thm]{Definition}

\theoremstyle{remark}
\newtheorem{rem}[thm]{Remark}

\begin{document}
%%%%%%%%%%%%%%%%

\setlength{\baselineskip}{18pt}

%====================%
%     START TITLE    %
%====================%
%
\title{\Large\bf Tensor product multiplicities \\[1.5mm] for crystal bases 
 of extremal weight modules \\[1.5mm] 
 over quantum infinite rank affine algebras \\[1.5mm] 
 of types $B_{\infty}$, $C_{\infty}$, and $D_{\infty}$}
\author{
 Satoshi Naito \\ 
 \small Institute of Mathematics, University of Tsukuba, \\
 \small Tsukuba, Ibaraki 305-8571, Japan \ 
 (e-mail: {\tt naito@math.tsukuba.ac.jp})
 \\[2mm] and \\[2mm]
 Daisuke Sagaki \\ 
 \small Institute of Mathematics, University of Tsukuba, \\
 \small Tsukuba, Ibaraki 305-8571, Japan \ 
 (e-mail: {\tt sagaki@math.tsukuba.ac.jp})
}
\date{}
\maketitle

%=======================%
%     START ABSTRACT    %
%=======================%
%
\begin{abstract} \setlength{\baselineskip}{16pt}
Using Lakshmibai-Seshadri paths, we give a combinatorial realization of 
the crystal basis of an extremal weight module of integral extremal weight 
over the quantized universal enveloping algebra associated to the infinite 
rank affine Lie algebra of type $B_{\infty}$, $C_{\infty}$, or $D_{\infty}$. 
Moreover, via this realization, we obtain an explicit description 
(in terms of Littlewood-Richardson coefficients) of how tensor products of 
these crystal bases decompose into connected components when their extremal 
weights are of nonnegative levels. These results, in types $B_{\infty}$, 
$C_{\infty}$, and $D_{\infty}$, extend the corresponding results due to Kwon, 
in types $A_{+\infty}$ and $A_{\infty}$; our results above also include, as 
a special case, the corresponding results (concerning crystal bases) due to
Lecouvey, in types $B_{\infty}$, $C_{\infty}$, and $D_{\infty}$, where 
the extremal weights are of level zero. 
\end{abstract}
%
%=========================%
%     START SECTION 01    %
%=========================%
%
\section{Introduction.}
\label{sec:intro}
Let $U_{q}(\Fg)$ be the quantized universal enveloping algebra over $\BC(q)$ 
associated to the infinite rank affine Lie algebra $\Fg$ of type $A_{+\infty}$, 
$A_{\infty}$, $B_{\infty}$, $C_{\infty}$, or $D_{\infty}$ with Cartan 
subalgebra $\Fh=\bigoplus_{i \in I}\BC h_{i}$ and integral weight lattice 
$P=\bigoplus_{i \in I}\BZ\Lambda_{i} \subset \Fh^{\ast}$, where $I$ is 
the (infinite) index set for the simple roots. In \cite{Kw-Adv}, \cite{Kw-Ep}, 
Kwon studied the crystal basis $\CB(\lambda)$ of the extremal weight 
$U_{q}(\Fg)$-module $V(\lambda)$ of extremal weight $\lambda \in P$, 
in the cases that $\Fg$ is of type $A_{+\infty}$ and type $A_{\infty}$; 
in these papers, he gave a combinatorial realization of the crystal basis 
$\CB(\lambda)$ for $\lambda \in P$ of level zero (see also 
\cite[\S4.1]{Kw-JA} for the case of dominant $\lambda \in P$), 
by using semistandard Young tableaux 
with entries in the crystal basis of the vector representation 
of $U_{q}(\Fg)$, and furthermore described explicitly 
(in terms of Littlewood-Richardson coefficients) how the tensor product 
$\CB(\lambda) \otimes \CB(\mu)$ decomposes into connected components 
when $\lambda,\,\mu \in P$ are of nonnegative levels. Also, in \cite{Lec-TAMS}, 
Lecouvey studied the crystal bases (and canonical bases) of certain irreducible 
$U_{q}(\Fg)$-modules, which are neither highest weight modules
nor lowest weight modules, in the cases that $\Fg$ is of type $A_{+\infty}$, 
$B_{\infty}$, $C_{\infty}$, and $D_{\infty}$, and furthermore showed that 
tensor product multiplicities for these modules are precisely the corresponding 
Littlewood-Richardson coefficients (not depending on the types $A_{+\infty}$, 
$B_{\infty}$, $C_{\infty}$, and $D_{\infty}$); these modules will surely 
turn out to be extremal weight $U_{q}(\Fg)$-modules of extremal weight of 
level zero.

In this paper, we extend the results above of \cite{Kw-Adv}, \cite{Kw-Ep} 
(in types $A_{+\infty}$ and $A_{\infty}$) to the cases of type $B_{\infty}$, 
$C_{\infty}$, and $D_{\infty}$ in such a way that the results above 
(concerning crystal bases) of \cite{Lec-TAMS} are also included. 
We emphasize that our approach is quite different from those in 
\cite{Kw-Adv}, \cite{Kw-Ep}, and \cite{Lec-TAMS}, where some double 
crystals (or bicrystals) play essential roles; we obtain 
our results entirely within the framework of Littelmann's path model, 
which enable us to give a unified proof in all types $B_{\infty}$, 
$C_{\infty}$, and $D_{\infty}$. In fact, we could also have included 
the cases of type $A_{+\infty}$ and $A_{\infty}$, but we do not dare do so, 
since this would make our notation more complicated and 
since these cases are already treated in \cite{Kw-Adv}, \cite{Kw-Ep}. 

Let us explain our results more precisely. We set $I:=\BZ_{\ge 0}$, 
and $[m]:=\bigl\{0,\,1,\,\dots,\,m\bigr\}$ for $m \in \BZ_{\ge 0}$. 
In addition, we denote by $P_{+} \subset P$ the set of dominant 
integral weights, and by $E$ the subset of $P$ consisting of 
all elements of level zero. For $\lambda \in P$, let $\BB(\lambda)$ 
denote the $U_{q}(\Fg)$-crystal consisting of all Lakshmibai-Seshadri 
(LS for short) paths of shape $\lambda$. 
Our first main result (Theorem~\ref{thm:isom})
states that for each $\lambda \in P$, the $U_{q}(\Fg)$-crystal $\BB(\lambda)$ 
gives a combinatorial realization of the crystal basis $\CB(\lambda)$; 
in addition, the crystal graph of $\BB(\lambda) \cong \CB(\lambda)$ 
is connected. Then, by means of this realization of $\CB(\lambda)$, 
we study how the tensor product $\CB(\lambda) \otimes \CB(\mu) \cong 
\BB(\lambda) \otimes \BB(\mu)$ decomposes into connected components 
when $\lambda,\,\mu \in P$ are of nonnegative levels, 
dividing the problem into four cases: 
the case $\lambda,\,\mu \in E$ (see \S\ref{subsec:EE}); 
the case $\lambda \in E$, $\mu \in P_{+}$ (see \S\ref{subsec:ED});
the case $\lambda \in P_{+}$, $\mu \in E$ (see \S\ref{subsec:DE});
the case $\lambda,\,\mu \in P_{+}$ (see \S\ref{subsec:DD}).
It turns out that in all the cases above, each connected component
of the tensor product $\BB(\lambda) \otimes \BB(\mu)$ is isomorphic 
to $\BB(\nu) \ (\cong \CB(\nu))$ for some $\nu \in P$ of nonnegative level. 
In particular, in the (most difficult) case 
$\lambda \in P_{+}$, $\mu \in E$, our result (Theorem~\ref{thm:DE})
states that for some $m \in \BZ_{\ge 3}$, the multiplicity of each connected 
component $\BB(\nu)$ in $\BB(\lambda) \otimes \BB(\mu)$ is equal to the 
(corresponding) tensor product multiplicity of finite-dimensional irreducible 
highest weight $U_{q}(\Fgm{m})$-modules of highest weight 
$\lambda$, $\mu$, and $\nu$, respectively. Here, $\Fgm{m}$ denotes 
the ``reductive'' Lie subalgebra of $\Fg$ 
(of type $B_{m+1}$, $C_{m+1}$, or $D_{m+1}$) 
corresponding to the subset $[m] \subset I=\BZ_{\ge 0}$; 
note that this $m \in \BZ_{\ge 3}$ does not depend on 
the connected components $\BB(\nu)$. Moreover, by virtue of tensor 
product multiplicity formulas in \cite{Ko97}, \cite{Ko98}, we obtain 
an explicit description of this multiplicity in terms of 
Littlewood-Richardson coefficients (see \S\ref{subsec:KT}). 

Now, from the argument in \S\ref{subsec:ED}, we also observe 
that for each $\lambda \in P$ of nonnegative level, there exist 
$\lambda^{0} \in E$ and $\lambda^{+} \in P_{+}$ such that 
$\BB(\lambda) \cong \BB(\lambda^{0}) \otimes \BB(\lambda^{+})$ 
as $U_{q}(\Fg)$-crystals. Therefore, by combining the results in 
the four cases above, we finally obtain 
our second main result (Theorem~\ref{thm:general}), 
which yields an explicit description (in terms of Littlewood-Richardson 
coefficients) of the multiplicity of each connected component 
$\BB(\nu)$ in $\BB(\lambda) \otimes \BB(\mu)$ for general 
$\lambda,\,\mu \in P$ of nonnegative levels. 

This paper is organized as follows. 
In \S\ref{sec:notation}, we introduce basic notation for 
infinite rank affine Lie algebras and their quantized universal enveloping 
algebras. In \S\ref{sec:crystal}, we first recall standard facts about 
crystal bases of extremal weight modules, and show the connectedness of 
(the crystal graph of) the crystal basis $\CB(\lambda)$ for $\lambda \in P$.
Then, we give a combinatorial realization of $\CB(\lambda)$ as the crystal 
$\BB(\lambda)$ of all LS paths of shape $\lambda$. In \S\ref{sec:decomp}, 
we describe explicitly how the tensor product $\BB(\lambda) \otimes \BB(\mu) 
\cong \CB(\lambda) \otimes \CB(\mu)$ decomposes into connected components 
when $\lambda,\,\mu \in P$ are of nonnegative levels, deferring the proof of 
Proposition~\ref{prop:DE-bij} (used to prove Theorem~\ref{thm:DE}) to 
\S\ref{sec:DE-bij}. Finally, in \S\ref{sec:DE-bij}, 
after reviewing tensor product multiplicity formulas 
in \cite{Ko97}, \cite{Ko98}, we show Proposition~\ref{prop:DE-bij}, 
thereby completing the proof of Theorem~\ref{thm:DE} 
(and hence Theorem~\ref{thm:general}). 

After having finished writing up this article, 
we were informed by Jae-Hoon Kwon that in \cite{Kw-Ep10}, 
he further obtained a description of how the tensor product 
$\CB(\lambda) \otimes \CB(\mu)$ decomposes into connected 
components for $\lambda \in P_{+}$ and $\mu \in -P_{+}$, 
in type $A_{\infty}$. 

%=========================%
%     START SECTION 02    %
%=========================%
%
\section{Basic notation for infinite rank affine Lie algebras.}
\label{sec:notation}
%
%==============================%
%     START SUBSECTION 0201    %
%==============================%
%
\subsection{Infinite rank affine Lie algebras.}
\label{subsec:infrank}

Let $\Fg$ be the infinite rank affine Lie algebra of 
type $B_{\infty}$, $C_{\infty}$, or $D_{\infty}$ 
(see \cite[\S7.11]{Kac}), that is, the (symmetrizable) Kac-Moody algebra 
of infinite rank associated to one of the following Dynkin diagrams: 

\vsp

\begin{center}
\hspace*{-10mm}
%WinTpicVersion3.08
\unitlength 0.1in
\begin{picture}( 44.7500, 30.1500)(  9.2500,-39.1500)
% CIRCLE 2 0 0 0
% 4 4600 1000 4600 950 4600 950 4600 950
% 
\special{pn 8}%
\special{sh 0.600}%
\special{ar 4600 1000 50 50  0.0000000 6.2831853}%
% CIRCLE 2 0 0 0
% 4 2200 1000 2200 950 2200 950 2200 950
% 
\special{pn 8}%
\special{sh 0.600}%
\special{ar 2200 1000 50 50  0.0000000 6.2831853}%
% CIRCLE 2 0 0 0
% 4 2800 1000 2800 950 2800 950 2800 950
% 
\special{pn 8}%
\special{sh 0.600}%
\special{ar 2800 1000 50 50  0.0000000 6.2831853}%
% CIRCLE 2 0 0 0
% 4 3400 1000 3400 950 3400 950 3400 950
% 
\special{pn 8}%
\special{sh 0.600}%
\special{ar 3400 1000 50 50  0.0000000 6.2831853}%
% CIRCLE 2 0 0 0
% 4 4000 1000 4000 950 4000 950 4000 950
% 
\special{pn 8}%
\special{sh 0.600}%
\special{ar 4000 1000 50 50  0.0000000 6.2831853}%
% STR 2 0 3 0
% 3 2800 1100 2800 1200 5 0
% $1$
\put(28.0000,-12.0000){\makebox(0,0){$1$}}%
% STR 2 0 3 0
% 3 3400 1100 3400 1200 5 0
% $2$
\put(34.0000,-12.0000){\makebox(0,0){$2$}}%
% STR 2 0 3 0
% 3 4000 1100 4000 1200 5 0
% $3$
\put(40.0000,-12.0000){\makebox(0,0){$3$}}%
% STR 2 0 3 0
% 3 2200 1100 2200 1200 5 0
% $0$
\put(22.0000,-12.0000){\makebox(0,0){$0$}}%
% STR 2 0 3 0
% 3 4600 1100 4600 1200 5 0
% $4$
\put(46.0000,-12.0000){\makebox(0,0){$4$}}%
% LINE 2 2 3 0
% 2 5000 1000 5400 1000
% 
\special{pn 8}%
\special{pa 5000 1000}%
\special{pa 5400 1000}%
\special{dt 0.045}%
% LINE 2 0 3 0
% 2 2800 970 2200 970
% 
\special{pn 8}%
\special{pa 2800 970}%
\special{pa 2200 970}%
\special{fp}%
% LINE 2 0 3 0
% 2 2800 1030 2200 1030
% 
\special{pn 8}%
\special{pa 2800 1030}%
\special{pa 2200 1030}%
\special{fp}%
% LINE 2 0 3 0
% 2 2800 1000 5000 1000
% 
\special{pn 8}%
\special{pa 2800 1000}%
\special{pa 5000 1000}%
\special{fp}%
% LINE 2 0 3 0
% 2 2400 1000 2500 900
% 
\special{pn 8}%
\special{pa 2400 1000}%
\special{pa 2500 900}%
\special{fp}%
% LINE 2 0 3 0
% 2 2400 1000 2500 1100
% 
\special{pn 8}%
\special{pa 2400 1000}%
\special{pa 2500 1100}%
\special{fp}%
% CIRCLE 2 0 0 0
% 4 4595 2000 4595 1950 4595 1950 4595 1950
% 
\special{pn 8}%
\special{sh 0.600}%
\special{ar 4596 2000 50 50  0.0000000 6.2831853}%
% CIRCLE 2 0 0 0
% 4 2195 2000 2195 1950 2195 1950 2195 1950
% 
\special{pn 8}%
\special{sh 0.600}%
\special{ar 2196 2000 50 50  0.0000000 6.2831853}%
% CIRCLE 2 0 0 0
% 4 2795 2000 2795 1950 2795 1950 2795 1950
% 
\special{pn 8}%
\special{sh 0.600}%
\special{ar 2796 2000 50 50  0.0000000 6.2831853}%
% CIRCLE 2 0 0 0
% 4 3395 2000 3395 1950 3395 1950 3395 1950
% 
\special{pn 8}%
\special{sh 0.600}%
\special{ar 3396 2000 50 50  0.0000000 6.2831853}%
% CIRCLE 2 0 0 0
% 4 3995 2000 3995 1950 3995 1950 3995 1950
% 
\special{pn 8}%
\special{sh 0.600}%
\special{ar 3996 2000 50 50  0.0000000 6.2831853}%
% STR 2 0 3 0
% 3 2795 2100 2795 2200 5 0
% $1$
\put(27.9500,-22.0000){\makebox(0,0){$1$}}%
% STR 2 0 3 0
% 3 3395 2100 3395 2200 5 0
% $2$
\put(33.9500,-22.0000){\makebox(0,0){$2$}}%
% STR 2 0 3 0
% 3 3995 2100 3995 2200 5 0
% $3$
\put(39.9500,-22.0000){\makebox(0,0){$3$}}%
% STR 2 0 3 0
% 3 2195 2100 2195 2200 5 0
% $0$
\put(21.9500,-22.0000){\makebox(0,0){$0$}}%
% STR 2 0 3 0
% 3 4595 2100 4595 2200 5 0
% $4$
\put(45.9500,-22.0000){\makebox(0,0){$4$}}%
% LINE 2 2 3 0
% 2 4995 2000 5395 2000
% 
\special{pn 8}%
\special{pa 4996 2000}%
\special{pa 5396 2000}%
\special{dt 0.045}%
% LINE 2 0 3 0
% 2 2795 1970 2195 1970
% 
\special{pn 8}%
\special{pa 2796 1970}%
\special{pa 2196 1970}%
\special{fp}%
% LINE 2 0 3 0
% 2 2795 2030 2195 2030
% 
\special{pn 8}%
\special{pa 2796 2030}%
\special{pa 2196 2030}%
\special{fp}%
% LINE 2 0 3 0
% 2 2795 2000 4995 2000
% 
\special{pn 8}%
\special{pa 2796 2000}%
\special{pa 4996 2000}%
\special{fp}%
% LINE 2 0 3 0
% 2 2600 2000 2500 1900
% 
\special{pn 8}%
\special{pa 2600 2000}%
\special{pa 2500 1900}%
\special{fp}%
% LINE 2 0 3 0
% 4 2600 2000 2510 2100 2510 2100 2510 2100
% 
\special{pn 8}%
\special{pa 2600 2000}%
\special{pa 2510 2100}%
\special{fp}%
\special{pa 2510 2100}%
\special{pa 2510 2100}%
\special{fp}%
% CIRCLE 2 0 0 0
% 4 2200 3000 2200 2950 2200 2950 2200 2950
% 
\special{pn 8}%
\special{sh 0.600}%
\special{ar 2200 3000 50 50  0.0000000 6.2831853}%
% CIRCLE 2 0 0 0
% 4 2800 3400 2800 3350 2800 3350 2800 3350
% 
\special{pn 8}%
\special{sh 0.600}%
\special{ar 2800 3400 50 50  0.0000000 6.2831853}%
% CIRCLE 2 0 0 0
% 4 2200 3800 2200 3750 2200 3750 2200 3750
% 
\special{pn 8}%
\special{sh 0.600}%
\special{ar 2200 3800 50 50  0.0000000 6.2831853}%
% CIRCLE 2 0 0 0
% 4 3400 3400 3400 3350 3400 3350 3400 3350
% 
\special{pn 8}%
\special{sh 0.600}%
\special{ar 3400 3400 50 50  0.0000000 6.2831853}%
% CIRCLE 2 0 0 0
% 4 4000 3400 4000 3350 4000 3350 4000 3350
% 
\special{pn 8}%
\special{sh 0.600}%
\special{ar 4000 3400 50 50  0.0000000 6.2831853}%
% CIRCLE 2 0 0 0
% 4 4600 3400 4600 3350 4600 3350 4600 3350
% 
\special{pn 8}%
\special{sh 0.600}%
\special{ar 4600 3400 50 50  0.0000000 6.2831853}%
% LINE 2 0 3 0
% 6 2200 3000 2800 3400 2800 3400 2200 3800 2800 3400 5000 3400
% 
\special{pn 8}%
\special{pa 2200 3000}%
\special{pa 2800 3400}%
\special{fp}%
\special{pa 2800 3400}%
\special{pa 2200 3800}%
\special{fp}%
\special{pa 2800 3400}%
\special{pa 5000 3400}%
\special{fp}%
% LINE 2 2 3 0
% 2 5000 3400 5400 3400
% 
\special{pn 8}%
\special{pa 5000 3400}%
\special{pa 5400 3400}%
\special{dt 0.045}%
% STR 2 0 3 0
% 3 2200 3100 2200 3200 5 0
% $0$
\put(22.0000,-32.0000){\makebox(0,0){$0$}}%
% STR 2 0 3 0
% 3 2200 3900 2200 4000 5 0
% $1$
\put(22.0000,-40.0000){\makebox(0,0){$1$}}%
% STR 2 0 3 0
% 3 2800 3500 2800 3600 5 0
% $2$
\put(28.0000,-36.0000){\makebox(0,0){$2$}}%
% STR 2 0 3 0
% 3 3400 3500 3400 3600 5 0
% $3$
\put(34.0000,-36.0000){\makebox(0,0){$3$}}%
% STR 2 0 3 0
% 3 4000 3500 4000 3600 5 0
% $4$
\put(40.0000,-36.0000){\makebox(0,0){$4$}}%
% STR 2 0 3 0
% 3 4600 3500 4600 3600 5 0
% $5$
\put(46.0000,-36.0000){\makebox(0,0){$5$}}%
% STR 2 0 3 0
% 3 1600 900 1600 1000 5 0
% $B_{\infty}$ : 
\put(16.0000,-10.0000){\makebox(0,0){$B_{\infty}$ : }}%
% STR 2 0 3 0
% 3 1600 1900 1600 2000 5 0
% $C_{\infty}$ : 
\put(16.0000,-20.0000){\makebox(0,0){$C_{\infty}$ : }}%
% STR 2 0 3 0
% 3 1600 3300 1600 3400 5 0
% $D_{\infty}$ : 
\put(16.0000,-34.0000){\makebox(0,0){$D_{\infty}$ : }}%
\end{picture}%
\end{center}

\vsp

\noindent
Following \cite[\S7.11]{Kac}, we realize this Lie algebra $\Fg$ 
as a Lie subalgebra of the Lie algebra $\Fgl_{\infty}(\BC)$ of 
complex matrices $(a_{ij})_{i,j \in \BZ}$ 
with finitely many nonzero entries as follows. 
If $\Fg$ is of type $B_{\infty}$ (resp., $C_{\infty}$, 
$D_{\infty}$), then we define elements $x_{i},\,y_{i},\,h_{i} \in 
\Fgl_{\infty}(\BC)$ for $i \in I:=\BZ_{\ge 0}$ 
by \eqref{eq:binf} (resp., \eqref{eq:cinf}, \eqref{eq:dinf}): 
%
%%%%%%%%%%%%%%%
%%% eq:binf %%%
%%%%%%%%%%%%%%%
%
\begin{equation} \label{eq:binf}
\begin{cases}
x_{0}:=E_{0,1}+E_{-1,0}, \quad 
x_{i}:=E_{i,i+1}+E_{-i-1,-i} \quad \text{for $i \in \BZ_{\ge 1}$}, \\[1.5mm]
y_{0}:=2(E_{1,0}+E_{0,-1}), \quad
y_{i}:=E_{i+1,i}+E_{-i,-i-1} \quad \text{for $i \in \BZ_{\ge 1}$}, \\[1.5mm]
h_{0}:=2\eps'_{0}, \quad
h_{i}:=-\eps'_{i-1}+\eps'_{i} \quad \text{for $i \in \BZ_{\ge 1}$}. 
\end{cases}
\end{equation}
%
%%%%%%%%%%%%%%%
%%% eq:cinf %%%
%%%%%%%%%%%%%%%
%
\begin{equation} \label{eq:cinf}
\begin{cases}
x_{0}:=E_{0,1}, \quad 
x_{i}:=E_{i,i+1}+E_{-i,-i+1} \quad \text{for $i \in \BZ_{\ge 1}$}, \\[1.5mm]
y_{0}:=E_{1,0}, \quad
y_{i}:=E_{i+1,i}+E_{-i+1,-i} \quad \text{for $i \in \BZ_{\ge 1}$}, \\[1.5mm]
h_{0}:=\eps'_{0}, \quad
h_{i}:=-\eps'_{i-1}+\eps'_{i} \quad \text{for $i \in \BZ_{\ge 1}$}. 
\end{cases}
\end{equation}
%
%%%%%%%%%%%%%%%
%%% eq:dinf %%%
%%%%%%%%%%%%%%%
%
\begin{equation} \label{eq:dinf}
\begin{cases}
x_{0}:=E_{0,2}-E_{-1,1}, \quad 
x_{i}:=E_{i,i+1}-E_{-i,-i+1} \quad \text{for $i \in \BZ_{\ge 1}$}, \\[1.5mm]
y_{0}:=E_{2,0}-E_{1,-1}, \quad
y_{i}:=E_{i,i+1}-E_{-i+1,-i} \quad \text{for $i \in \BZ_{\ge 1}$}, \\[1.5mm]
h_{0}:=\eps'_{0}+\eps'_{1}, \quad
h_{i}:=-\eps'_{i-1}+\eps'_{i} \quad \text{for $i \in \BZ_{\ge 1}$}.
\end{cases}
\end{equation}
Here, for each $i,\,j \in \BZ$, $E_{i,j} \in \Fgl_{\infty}(\BC)$ is the matrix 
with a $1$ in the $(i,\,j)$ position, and $0$ elsewhere, and 
\begin{equation*}
\eps'_{j}:=
\begin{cases}
E_{-j-1,-j-1}-E_{j+1,j+1} & 
 \text{if $\Fg$ is of type $B_{\infty}$}, \\[1.5mm]
E_{-j,-j}-E_{j+1,j+1} & 
 \text{if $\Fg$ is of type $C_{\infty}$ or $D_{\infty}$},
\end{cases}
\qquad \text{for $j \in \BZ_{\ge 0}$}.
\end{equation*}
The infinite rank affine Lie algebra $\Fg$ 
is isomorphic to the Lie subalgebra of $\Fgl_{\infty}(\BC)$ 
generated by $x_{i},\,y_{i},\,h_{i}$ for $i \in I$; 
the elements $x_{i},\,y_{i}$ for $i \in I$ are the Chevalley generators, 
$\Fh:=\bigoplus_{i \in I} \BC h_{i}$ is the Cartan subalgebra, 
and $\Pi^{\vee}:=\{h_{i}\}_{i \in I}$ is the set of simple coroots. 
Let $\Pi:=\{\alpha_{i}\}_{i \in I} \subset 
\Fh^{\ast}:=\Hom_{\BC}(\Fh,\,\BC)$ be the set of simple roots for $\Fg$.
Then we have 
%
%%%%%%%%%%%%%%%%%
%%% eq:simple %%%
%%%%%%%%%%%%%%%%%
%
\begin{equation} \label{eq:simple}
\begin{cases}
\alpha_{0}=\eps_{0}, \quad
\alpha_{i}=-\eps_{i-1}+\eps_{i} \quad
\text{for $i \in \BZ_{\ge 1}$}, 
& \text{if $\Fg$ is of type $B_{\infty}$}, \\[1.5mm]
\alpha_{0}=2\eps_{0}, \quad
\alpha_{i}=-\eps_{i-1}+\eps_{i} \quad
\text{for $i \in \BZ_{\ge 1}$}, 
& \text{if $\Fg$ is of type $C_{\infty}$}, \\[1.5mm]
\alpha_{0}=\eps_{0}+\eps_{1}, \quad
\alpha_{i}=-\eps_{i-1}+\eps_{i} \quad
\text{for $i \in \BZ_{\ge 1}$}, 
& \text{if $\Fg$ is of type $D_{\infty}$}.
\end{cases}
\end{equation}
Here, for each $j \in \BZ_{\ge 0}$, 
we define $\eps_{j} \in \Fh^{\ast}=\Hom_{\BC}(\Fh,\BC)$ 
by: $\pair{\eps_{j}}{\eps'_{k}}=\delta_{jk}$ 
for $k \in \BZ_{\ge 0}$, 
where $\pair{\cdot\,}{\cdot}$ denotes the natural pairing of 
$\Fh^{\ast}$ and $\Fh$; note that 
$\Fh=\bigoplus_{j \in \BZ_{\ge 0}} \BC \eps'_{j}$. 

Let $W=\langle r_{i} \mid i \in I \rangle \subset GL(\Fh^{\ast})$ 
denote the Weyl group of $\Fg$, where $r_{i}$ denotes 
the simple reflection corresponding to the simple root 
$\alpha_{i}$ for $i \in I$. 
If $\Fg$ is of type $B_{\infty}$ (resp., $C_{\infty}$, $D_{\infty}$), 
then the following equation \eqref{eq:binf-r} 
(resp., \eqref{eq:cinf-r}, \eqref{eq:dinf-r}) holds 
for $j \in \BZ_{\ge 0}$: 
%
%%%%%%%%%%%%%%%%%
%%% eq:binf-r %%%
%%%%%%%%%%%%%%%%%
%
\begin{equation} \label{eq:binf-r}
r_{0}(\eps_{j})=
  \begin{cases}
  -\eps_{0} & \text{if $j=0$}, \\[1mm]
  \eps_{j} & \text{otherwise},
  \end{cases}
\qquad
r_{i}(\eps_{j})=
  \begin{cases}
  \eps_{i} & \text{if $j=i-1$}, \\[1mm]
  \eps_{i-1} & \text{if $j=i$}, \\[1mm]
  \eps_{j} & \text{otherwise},
  \end{cases}
\quad \text{for $i \in \BZ_{\ge 1}$};
\end{equation}
%
%%%%%%%%%%%%%%%%%
%%% eq:cinf-r %%%
%%%%%%%%%%%%%%%%%
%
\begin{equation} \label{eq:cinf-r}
r_{0}(\eps_{j})=
  \begin{cases}
  -\eps_{0} & \text{if $j=0$}, \\[1mm]
  \eps_{j} & \text{otherwise},
  \end{cases}
\qquad
r_{i}(\eps_{j})=
  \begin{cases}
  \eps_{i} & \text{if $j=i-1$}, \\[1mm]
  \eps_{i-1} & \text{if $j=i$}, \\[1mm]
  \eps_{j} & \text{otherwise},
  \end{cases}
\quad \text{for $i \in \BZ_{\ge 1}$};
\end{equation}
%
%%%%%%%%%%%%%%%%%
%%% eq:dinf-r %%%
%%%%%%%%%%%%%%%%%
%
\begin{equation} \label{eq:dinf-r}
r_{0}(\eps_{j})=
  \begin{cases}
  -\eps_{1} & \text{if $j=0$}, \\[1mm]
  -\eps_{0} & \text{if $j=1$}, \\[1mm]
  \eps_{j} & \text{otherwise},
  \end{cases}
\qquad
r_{i}(\eps_{j})=
  \begin{cases}
  \eps_{i} & \text{if $j=i-1$}, \\[1mm]
  \eps_{i-1} & \text{if $j=i$}, \\[1mm]
  \eps_{j} & \text{otherwise},
  \end{cases}
\quad \text{for $i \in \BZ_{\ge 1}$}.
\end{equation}
Also, for the action of $W$ on $\Fh$, entirely similar formulas hold 
in all the cases above, with $\eps$ replaced by $\eps'$. 

Let $\Delta:=W\Pi$ be the set of roots for $\Fg$, and 
$\Delta^{+}:=\Delta \cap \sum_{i \in I} \BZ_{\ge 0}\alpha_{i}$ 
the set of positive roots for $\Fg$. 
If $\Fg$ is of type $B_{\infty}$ 
(resp., $C_{\infty}$, $D_{\infty}$), then 
the sets $\Delta$ and $\Delta^{+}$ 
are given by \eqref{eq:binf-rs} 
(resp., \eqref{eq:cinf-rs}, \eqref{eq:dinf-rs}): 
%
%%%%%%%%%%%%%%%%%%
%%% eq:binf-rs %%%
%%%%%%%%%%%%%%%%%%
%
\begin{equation} \label{eq:binf-rs}
\begin{cases}
\Delta=\bigl\{\pm \eps_{j} \mid 1 \le j\bigr\} \cup 
\bigl\{\pm \eps_{j} \pm \eps_{i} \mid 0 \le j < i\bigr\}, \\[1mm]
\Delta^{+}=\bigl\{\eps_{j} \mid 1 \le j\bigr\} \cup 
\bigl\{-\eps_{j}+\eps_{i},\ \eps_{j}+\eps_{i} \mid 0 \le j < i\bigr\};
\end{cases}
\end{equation}
%
%%%%%%%%%%%%%%%%%%
%%% eq:cinf-rs %%%
%%%%%%%%%%%%%%%%%%
%
\begin{equation} \label{eq:cinf-rs}
\begin{cases}
\Delta=\bigl\{\pm 2\eps_{j} \mid 1 \le j\bigr\} \cup 
\bigl\{\pm \eps_{j} \pm \eps_{i} \mid 0 \le j < i\bigr\}, \\[1mm]
\Delta^{+}=\bigl\{2\eps_{j} \mid 1 \le j\bigr\} \cup 
\bigl\{-\eps_{j}+\eps_{i},\ \eps_{j}+\eps_{i} \mid 0 \le j < i\bigr\};
\end{cases}
\end{equation}
%
%%%%%%%%%%%%%%%%%%
%%% eq:dinf-rs %%%
%%%%%%%%%%%%%%%%%%
%
\begin{equation} \label{eq:dinf-rs}
\begin{cases}
\Delta=
\bigl\{\pm \eps_{j} \pm \eps_{i} \mid 0 \le j < i\bigr\}, \\[1mm]
\Delta^{+}=
\bigl\{-\eps_{j}+\eps_{i},\ \eps_{j}+\eps_{i} \mid 0 \le j < i\bigr\}. 
\end{cases}
\end{equation}
We set $E:=\bigoplus_{j \in \BZ_{\ge 0}} \BZ \eps_{j} 
\subset \Fh^{\ast}$, and then 
$E_{\BC}:=\BC \otimes_{\BZ} E=
\bigoplus_{j \in \BZ_{\ge 0}} \BC \eps_{j} 
\subset \Fh^{\ast}$; note that $\Delta$ is contained in $E$, and 
$E$ is stable under the action of the Weyl group $W$.
Define a $W$-invariant, symmetric $\BC$-bilinear form 
$(\cdot\,,\,\cdot)$ on $E_{\BC}$ by: 
$(\eps_{i},\,\eps_{j})=\delta_{ij}$ for $i,\,j \in \BZ_{\ge 0}$, 
and a $\BC$-linear isomorphism 
$\psi:
 E_{\BC}=\bigoplus_{j \in \BZ_{\ge 0}} \BC \eps_{j} 
 \rightarrow 
 \Fh=\bigoplus_{j \in \BZ_{\ge 0}} \BC \eps'_{j}$ by: 
$\psi(\eps_{j})=\eps_{j}'$ for $j \in \BZ_{\ge 0}$.
Also, we define the dual root $\beta^{\vee} \in \Fh$ of 
a root $\beta \in \Delta$ by: 
$\beta^{\vee}=2\psi(\beta)/(\beta,\beta)$; 
observe that $\alpha_{i}^{\vee}=h_{i}$ for all $i \in I$. 

%==============================%
%     START SUBSECTION 0202    %
%==============================%
%
\subsection{Integral weights and their levels.}
\label{subsec:intwt}
For each $i \in I$, let $\Lambda_{i} \in \Fh^{\ast}$ denote
the $i$-th fundamental weight for $\Fg$. 
We should warn the reader that 
these elements $\Lambda_{i}$, $i \in I$, are not contained 
in the subspace $E_{\BC}=\bigoplus_{j \in \BZ_{\ge 0}}\BC \eps_{j}$ 
of $\Fh^{\ast}=\prod_{j \in \BZ_{\ge 0}} \BC \eps_{j}$, 
where the vector space $\prod_{j \in \BZ_{\ge 0}} \BC \eps_{j}$ 
is thought of as a certain completion of 
$\bigoplus_{j \in \BZ_{\ge 0}}\BC \eps_{j}$. 
In fact, we have 
%
%%%%%%%%%%%%%%%%%
%%% eq:binf-f %%%
%%%%%%%%%%%%%%%%%
%
\begin{equation} \label{eq:binf-f}
\begin{cases}
\Lambda_{0}=\dfrac{1}{2}(\eps_{0}+\eps_{1}+\eps_{2}+\cdots), \\[5mm]
\Lambda_{i}=\eps_{i}+\eps_{i+1}+\eps_{i+2}+\cdots \quad 
  \text{for $i \in \BZ_{\ge 1}$}, 
\end{cases}
\quad \text{if $\Fg$ is of type $B_{\infty}$}, 
\end{equation}
%
%%%%%%%%%%%%%%%%%
%%% eq:cinf-f %%%
%%%%%%%%%%%%%%%%%
%
\begin{equation} \label{eq:cinf-f}
\Lambda_{i}=\eps_{i}+\eps_{i+1}+\eps_{i+2}+\cdots \quad 
  \text{for $i \in I$}, \quad
  \text{if $\Fg$ is of type $C_{\infty}$}, 
\end{equation}
%
%%%%%%%%%%%%%%%%%
%%% eq:dinf-f %%%
%%%%%%%%%%%%%%%%%
%
\begin{equation} \label{eq:dinf-f}
\begin{cases}
\Lambda_{0}=\dfrac{1}{2}(\eps_{0}+\eps_{1}+\eps_{2}+\cdots), \\[5mm]
\Lambda_{1}=\dfrac{1}{2}(-\eps_{0}+\eps_{1}+\eps_{2}+\cdots), \\[5mm]
\Lambda_{i}=\eps_{i}+\eps_{i+1}+\eps_{i+2}+\cdots \quad 
  \text{for $i \in \BZ_{\ge 2}$},
\end{cases}
\quad \text{if $\Fg$ is of type $D_{\infty}$}.
\end{equation}
Let $P:=\bigoplus_{i \in I}\BZ\Lambda_{i} \subset \Fh^{\ast}$ 
denote the integral weight lattice for $\Fg$, and 
let $P_{+}:=\sum_{i \in I}\BZ_{\ge 0}\Lambda_{i} \subset P$ 
be the set of dominant integral weights for $\Fg$. 

For an integral weight $\lambda \in P$, 
we define $\col{j} \in (1/2)\BZ$ for $j \in \BZ_{\ge 0}$ by: 
\begin{equation*}
\col{j}:=\pair{\lambda}{\eps'_{j}}, \quad 
\text{or equivalently}, \quad
\lambda=\col{0}\eps_{0}+\col{1}\eps_{1}+
\col{2}\eps_{2}+\cdots;
\end{equation*} 
note that either 
$\col{j} \in \BZ$ for all $j \in \BZ_{\ge 0}$, or 
$\col{j} \in 1/2+\BZ$ for all $j \in \BZ_{\ge 0}$.
If $\lambda \in P$, then for $n$ sufficiently large, 
we have
\begin{equation*}
\col{n}=\col{n+1}=\col{n+2}= \cdots, 
\end{equation*}
since $\lambda \in P$ is a finite sum of 
integer multiples of fundamental weights; 
in this case, we set $L_{\lambda}:=\col{n}$, 
and call it the level of the integral weight $\lambda$. 
Note that $E \subset P$ is identical to the set 
of integral weights of level zero, 
and that if $\lambda \in P$ is a dominant integral weight 
not equal to $0$, then the level $L_{\lambda}$ of 
$\lambda$ is positive. 
%
%%%%%%%%%%%%%%%
%%% rem:dom %%%
%%%%%%%%%%%%%%%
%
\begin{rem}[{cf. \cite[statements (2)--(4) on p.\,82]{Ko98}}] \label{rem:dom}
Let $\lambda \in P_{+}$. 
Then, it follows that 
\begin{equation*}
0 \le \underbrace{\abs{\col{0}} \le \col{1} \le \col{2} \le 
\cdots \le \col{n-1} <}_{\text{If $n=0$, then this part is omitted.}} 
L_{\lambda}= \col{n}=\col{n+1}=\cdots
\end{equation*}
for some $n \in \BZ_{\ge 0}$, 
where for $x \in (1/2)\BZ$, we set
\begin{equation*}
\abs{x}:=
 \begin{cases}
 x & \text{if $\Fg$ is of type $B_{\infty}$ or $C_{\infty}$}, \\[1mm]
 |x| & \text{if $\Fg$ is of type $D_{\infty}$}.
 \end{cases}
\end{equation*}
\end{rem}

%==============================%
%     START SUBSECTION 0203    %
%==============================%
%
\subsection{
  Quantized universal enveloping algebras and 
  their finite rank subalgebras.}
\label{subsec:levi}

We set $P^{\vee}:=\bigoplus_{i \in I}\BZ h_{i} \subset \Fh$, 
and let $U_{q}(\Fg)=\langle x_{i},\,y_{i},\,q^{h} \mid 
i \in I,\,h \in P^{\vee}\rangle$ denote 
the quantized universal enveloping algebra of $\Fg$ over $\BC(q)$
with integral weight lattice $P$, and 
Chevalley generators $x_{i},\,y_{i}$, $i \in I$.
Also, let $U_{q}^{+}(\Fg)$ (resp., $U^{-}_{q}(\Fg)$) denote 
the positive (resp., negative) part of $U_{q}(\Fg)$, 
that is, the $\BC(q)$-subalgebra of $U_{q}(\Fg)$ 
generated by $x_{i}$, $i \in I$ (resp., $y_{i}$, $i \in I$). 
We have a $\BC(q)$-algebra anti-automorphism 
$\ast:U_{q}(\Fg) \rightarrow U_{q}(\Fg)$ defined by:
%
%%%%%%%%%%%%%%
%%% eq:ast %%%
%%%%%%%%%%%%%%
%
\begin{equation} \label{eq:ast}
\begin{cases}
(q^{h})^{\ast}=q^{-h} 
  & \text{for $h \in P^{\vee}$}, \\[1.5mm]
x_{i}^{\ast}=x_{i},\ 
y_{i}^{\ast}=y_{i}
  & \text{for $i \in I$}.
\end{cases}
\end{equation}
Note that $U_{q}^{\pm}(\Fg)$ is stable 
under the $\BC(q)$-algebra anti-automorphism 
$\ast:U_{q}(\Fg) \rightarrow U_{q}(\Fg)$. 

For $m,\,n \in \BZ_{\ge 0}$ with $n \ge m$, 
we denote the finite interval 
$\bigl\{m,\,m+1,\,\dots,\,n\bigr\}$ in $I=\BZ_{\ge 0}$ 
by $[m,\,n]$. Also, for $n \in \BZ_{\ge 0}$, 
we write simply $[n]$ for the finite interval 
$[0,\,n]=\bigl\{0,\,1,\,\dots,\,n\bigr\}$ in $I=\BZ_{\ge 0}$.
Let $J$ be a finite interval in $I=\BZ_{\ge 0}$, 
which is of the form $[m,\,n]$ for some 
$m,\,n \in \BZ_{\ge 0}$ with $n \ge m$. 
We denote by $\Fg_{J}$ the (Lie) subalgebra of $\Fg$ 
generated by $x_{i},\,y_{i}$, $i \in J$, and $\Fh$. 
Then, the set $\Delta_{J}:=\Delta \cap 
\bigoplus_{i \in J}\BZ \alpha_{i}$ and 
$\Delta_{J}^{+}:=\Delta^{+} \cap 
\bigoplus_{i \in J}\BZ \alpha_{i}$ are 
the sets of roots and positive roots for $\Fg_{J}$, 
respectively. 
%
%%%%%%%%%%%%%%%%
%%% rem:Levi %%%
%%%%%%%%%%%%%%%%
%
\begin{rem} \label{rem:Levi}
If $\Fg$ is of type $B_{\infty}$ 
(resp., $C_{\infty}$, $D_{\infty})$, and $n \in \BZ_{\ge 3}$, 
then the Lie subalgebra $\Fgn$ of $\Fg$ is a ``reductive'' Lie algebra 
of type $B_{n+1}$ (resp., $C_{n+1}$, $D_{n+1}$). 
Also, if $n \in \BZ_{\ge 1}$, then the Lie subalgebra 
$\Fgm{1,\,n}$ of $\Fg$ is a ``reductive'' Lie algebra 
of type $A_{n}$.
\end{rem}

Denote by $U_{q}(\Fg_{J})$ the $\BC(q)$-subalgebra of 
$U_{q}(\Fg)$ generated by $x_{i},\,y_{i}$, $i \in J$, 
and $q^{h}$, $h \in P^{\vee}$, which can be thought of as 
the quantized universal enveloping algebra of 
$\Fg_{J}$ over $\BC(q)$, and 
denote by $U_{q}^{+}(\Fg_{J})$ (resp., $U^{-}_{q}(\Fg_{J})$)
the positive (resp., negative) part of $U_{q}(\Fg_{J})$, 
that is, the $\BC(q)$-subalgebra of $U_{q}(\Fg_{J})$ 
generated by $x_{i}$, $i \in J$ (resp., $y_{i}$, $i \in J$). 
Then it is easily seen that $U_{q}(\Fg_{J})$ and 
$U_{q}^{\pm}(\Fg_{J})$ are stable 
under the $\BC(q)$-algebra anti-automorphism 
$\ast:U_{q}(\Fg) \rightarrow U_{q}(\Fg)$ 
given by \eqref{eq:ast}. 

Let $W_{J}$ denote the (finite) subgroup of $W$ generated 
by the $r_{i}$ for $i \in J$, which is the Weyl group of $\Fg_{J}$. 
An integral weight $\lambda \in P$ is said to 
be $J$-dominant (resp., $J$-antidominant) if 
$\pair{\lambda}{h_{i}} \ge 0$ (resp., $\pair{\lambda}{h_{i}} \le 0$)
for all $i \in J$. 
For each integral weight $\lambda \in P$, 
we denote by $\lambda_{J}$ the unique element of 
$W_{J}\lambda$ that is $J$-dominant.
%
%%%%%%%%%%%%%%%%%
%%% rem:n-dom %%%
%%%%%%%%%%%%%%%%%
%
\begin{rem}[{cf. \cite[statements (2)--(4) on p.\,82]{Ko98}}] \label{rem:n-dom}
Let $n \in \BZ_{\ge 3}$, and let $\lambda \in P$ 
be an $[n]$-dominant integral weight. 
Then it follows from \eqref{eq:simple} that 
\begin{equation*}
0 \le \abs{\col{0}} \le \col{1} \le \cdots \le \col{n-1} \le \col{n}. 
\end{equation*}
\end{rem}

Now, for $\lambda \in E=\bigoplus_{j \in \BZ_{\ge 0}}\BZ\eps_{j}$, 
we set
\begin{equation*}
\Supp(\lambda):=\bigl\{j \in \BZ_{\ge 0} \mid \col{j} \ne 0\bigr\}.
\end{equation*}
%
%%%%%%%%%%%%%%%%
%%% lem:lamn %%%
%%%%%%%%%%%%%%%%
%
\begin{lem} \label{lem:lamn}
Let $\lambda \in E$, and set $p:=\#\Supp (\lambda)$. 
Let $n \in \BZ_{\ge 3}$ be such that 
$\Supp(\lambda) \subsetneqq [n]$. 
Then, the unique $[n]$-dominant element $\lamn$ 
of $\Wn\lambda$ satisfies the following property\,{\rm:}
\begin{equation*}
\Supp (\lamn)=[n-p+1,\,n], \quad 
\text{\rm and} \quad
0 < \co{\lamn}{n-p+1} \le \cdots \le 
\co{\lamn}{n-1} \le \co{\lamn}{n}.
\end{equation*}
\end{lem}

\begin{proof}
Since $\Supp(\lambda) \subsetneqq [n]$ by assumption, and 
$\lamn \in \Wn\lambda$ by definition, it follows by using 
\eqref{eq:binf-r}--\eqref{eq:dinf-r} that 
$\Supp(\lamn) \subsetneqq [n]$. 
Also, we see from Remark~\ref{rem:n-dom} that
\begin{equation*}
0 \le \abs{\co{\lamn}{0}} \le \co{\lamn}{1} \le \cdots \le 
\co{\lamn}{n-1} \le \co{\lamn}{n}. 
\end{equation*}
If $\co{\lamn}{0} \ne 0$, then we have $\Supp (\lamn)=[n]$, 
which is contradiction. Therefore, we get $\co{\lamn}{0}=0$, and hence 
$0 \le \co{\lamn}{1} \le \cdots \le 
\co{\lamn}{n-1} \le \co{\lamn}{n}$. 
Furthermore, since $\#\Supp(\lamn)=\#\Supp(\lambda)=p$, 
we deduce that $\Supp (\lamn)=[n-p+1,\,n]$. 
This proves the lemma. 
\end{proof}

%=========================%
%     START SECTION 03    %
%=========================%
%
\section{Path model for the crystal basis of an extremal weight module.}
\label{sec:crystal}
%
%==============================%
%     START SUBSECTION 0301    %
%==============================%
%
\subsection{Extremal elements.}
\label{subsec:extremal}

Let $\CB$ be a $U_{q}(\Fg)$-crystal 
(resp., $U_{q}(\Fg_{J})$-crystal for a 
finite interval $J$ in $I=\BZ_{\ge 0}$), 
equipped with the maps 
$e_{i},\,f_{i}:\CB \rightarrow \CB \cup \{\bzero\}$ 
for $i \in I$ (resp., $i \in J$), which we call 
the Kashiwara operators, and the maps $\wt:\CB \rightarrow P$, 
$\ve_{i},\,\vp_{i}:\CB \rightarrow \BZ \cup \{-\infty\}$ 
for $i \in I$ (resp., $i \in J$). 
Here, $\bzero$ is a formal element 
not contained in $\CB$. 
For $\nu \in P$, we denote by $\CB_{\nu}$ the subset of 
$\CB$ consisting of all elements of weight $\nu$. 

\begin{dfn}[{cf. \cite[p.\,389]{Kas94}}]
(1) A $U_{q}(\Fg)$-crystal $\CB$ is said to be normal 
if $\CB$, regarded as a $U_{q}(\Fg_{K})$-crystal 
by restriction, is isomorphic to a direct sum of 
the crystal bases of finite-dimensional 
irreducible $U_{q}(\Fg_{K})$-modules for every
finite interval $K$ in $I=\BZ_{\ge 0}$. 

(2) Let $J$ be a finite interval in $I=\BZ_{\ge 0}$. 
A $U_{q}(\Fg_{J})$-crystal $\CB$ is said to be normal 
if $\CB$, regarded as a $U_{q}(\Fg_{K})$-crystal 
by restriction, is isomorphic to a direct sum of 
the crystal bases of finite-dimensional 
irreducible $U_{q}(\Fg_{K})$-modules for every
finite interval $K$ in $I=\BZ_{\ge 0}$ contained in $J$. 
\end{dfn}

If $\CB$ is a normal $U_{q}(\Fg)$-crystal, 
then $\CB$, regarded as a $U_{q}(\Fg_{J})$-crystal 
by restriction, is a normal $U_{q}(\Fg_{J})$-crystal 
for every finite interval $J$ in $I=\BZ_{\ge 0}$. 
Also, if $\CB$ is a normal $U_{q}(\Fg)$-crystal 
(resp., normal $U_{q}(\Fg_{J})$-crystal 
for a finite interval $J$ in $I=\BZ_{\ge 0}$), 
then we have 
\begin{equation*}
\ve_{i}(b)=
 \max \bigl\{k \in \BZ_{\ge 0} \mid e_{i}^{k}b \ne \bzero\bigr\}
\quad \text{and} \quad
\vp_{i}(b)=
 \max \bigl\{k \in \BZ_{\ge 0} \mid f_{i}^{k}b \ne \bzero\bigr\}
\end{equation*}
for $b \in \CB$ and $i \in I$ (resp., $i \in J$); 
in this case, we set 
\begin{equation*}
e_{i}^{\max}b:=e_{i}^{\ve_{i}(b)}b
\quad \text{and} \quad
f_{i}^{\max}b:=f_{i}^{\vp_{i}(b)}b
\end{equation*}
for $b \in \CB$ and $i \in I$ (resp., $i \in J$). 
Also, we can define an action of the Weyl group $W$ 
(resp., $W_{J}$) on $\CB$ as follows. 
For each $i \in I$ (resp., $i \in J$), define 
$S_{i}:\CB \rightarrow \CB$ by:
%
%%%%%%%%%%%%%
%%% eq:si %%%
%%%%%%%%%%%%%
%
\begin{equation} \label{eq:si}
S_{i}b=\begin{cases}
f_{i}^{k}b 
 & \text{if \ } k:=\pair{\wt b}{h_{i}} \ge 0, \\[1.5mm]
e_{i}^{-k}b 
 & \text{if \ } k:=\pair{\wt b}{h_{i}} < 0,
\end{cases} \qquad \text{for $b \in \CB$}. 
\end{equation}
Then, these operators $S_{i}$, $i \in I$, give rise to 
a unique (well-defined) action 
$S:W \rightarrow \Bij(\CB)$ 
(resp., $S:\Wn \rightarrow \Bij(\CB)$), 
$w \mapsto S_{w}$, of the Weyl group $W$ (resp., $W_{J}$)
on the set $\CB$ such that 
$S_{r_{i}}=S_{i}$ for all $i \in I$ (resp., $i \in J$). 
Here, for a set $X$, $\Bij(X)$ denotes the group of 
all bijections from the set $X$ to itself. 
%
%%%%%%%%%%%%%%%%%%%%
%%% dfn:extremal %%%
%%%%%%%%%%%%%%%%%%%%
%
\begin{dfn} \label{dfn:extremal}
Suppose that $\CB$ is a normal $U_{q}(\Fg)$-crystal 
(resp., $U_{q}(\Fg_{J})$-crystal for a finite interval 
$J$ in $I=\BZ_{\ge 0}$).

(1) An element $b \in \CB$ is said to be extremal (resp., $J$-extremal) 
if for every $w \in W$ and $i \in I$ (resp., $w \in W_{J}$ and $i \in J$),
%
%%%%%%%%%%%%%%%%%%%
%%% eq:extremal %%%
%%%%%%%%%%%%%%%%%%%
%
\begin{equation} \label{eq:extremal}
\begin{cases}
e_{i}S_{w}b=\bzero 
 & \text{if $\pair{w(\wt b)}{h_{i}} \ge 0$}, \\[1.5mm]
f_{i}S_{w}b=\bzero 
 & \text{if $\pair{w(\wt b)}{h_{i}} \le 0$}.
\end{cases}
\end{equation}

(2) An element $b \in \CB$ is said to be maximal
(resp., $J$-maximal) if $e_{i}b=\bzero$ for all $i \in I$ 
(resp., $i \in J$). 

(3) An element $b \in \CB$ is said to be minimal (resp., $J$-minimal)
if $f_{i}b=\bzero$ for all $i \in I$ (resp., $i \in J$). 
\end{dfn}
%
%%%%%%%%%%%%%%%
%%% rem:ext %%%
%%%%%%%%%%%%%%%
%
\begin{rem} \label{rem:ext}
Suppose that $\CB$ is a normal $U_{q}(\Fg)$-crystal.

(1) An element $b \in \CB$ is extremal 
if and only if there exists $m \in \BZ_{\ge 0}$ such that 
$b$ is $[n]$-extremal for all $n \in \BZ_{\ge m}$. 

(2) Let $n \in \BZ_{\ge 0}$. 
Since the (unique) $[n]$-maximal element of the crystal basis of 
the finite-dimensional irreducible $U_{q}(\Fgn)$-module is 
$[n]$-extremal, it follows that an $[n]$-maximal element of 
$\CB$ is $[n]$-extremal. Consequently, by part (1), 
a maximal element of $\CB$ is extremal.

(3) Let $J$ be a finite interval of $I=\BZ_{\ge 0}$. 
By the same reasoning as in part (2), we see that 
a $J$-minimal element of $\CB$ is $J$-extremal. 
\end{rem}

%==============================%
%     START SUBSECTION 0302    %
%==============================%
%
\subsection{Crystal bases of extremal weight modules.}
\label{subsec:cb-ext}

In this subsection, we study some basic properties of 
crystal bases of extremal weight modules; 
note that all results in \cite{Kas91}--\cite{Kas96} about 
extremal weight modules and their crystal bases that 
we use in this paper remain valid in the case 
of infinite rank affine Lie algebras. 
%
%%%%%%%%%%%%%%%%%%
%%% dfn:extvec %%%
%%%%%%%%%%%%%%%%%%
%
\begin{dfn}[{see \cite[Definition~8.1]{Kas94}}] 
\label{dfn:extvec}
Let $\lambda \in P$ be an integral weight. 

(1) Let $M$ be an integrable $U_{q}(\Fg)$-module. 
A weight vector $v \in M$ of weight $\lambda$
is said to be extremal if there exists a family 
$\{v_{w}\}_{w \in W}$ of weight vectors in $M$
satisfying the following conditions: 

(i) If $w$ is the identity element $e$ of $W$, 
then $v_{w}=v_{e}=v$; 

(ii) for $w \in W$ and $i \in I$ such that 
$k:=\pair{w\lambda}{h_{i}} \ge 0$, we have
$x_{i}v_{w}=0$ and $y_{i}^{(k)}v_{w}=v_{r_{i}w}$; 

(iii) for $w \in W$ and $i \in I$ such that 
$k:=\pair{w\lambda}{h_{i}} \le 0$, we have 
$y_{i}v_{w}=0$ and $x_{i}^{(-k)}v_{w}=v_{r_{i}w}$. 

\noindent
Here, for $i \in I$ and $k \in \BZ_{\ge 0}$, 
$x_{i}^{(k)}$ and $y_{i}^{(k)}$ denote 
the $k$-th $q$-divided powers of $x_{i}$ and $y_{i}$, 
respectively. 

(2) Let $J$ be a finite interval in $I=\BZ_{\ge 0}$, and 
let $M$ be an integrable $U_{q}(\Fg_{J})$-module. 
A weight vector $v \in M$ of weight $\lambda$  
is said to be $J$-extremal if there exists a family 
$\{v_{w}\}_{w \in W_{J}}$ of weight vectors in $M$
satisfying the same conditions as (i), (ii), (iii) above, 
with $W$ replaced by $W_{J}$, and $I$ by $J$.  
\end{dfn}

The extremal weight $U_{q}(\Fg)$-module $V(\lambda)$ of 
extremal weight $\lambda$ is, by definition, 
the integrable $U_{q}(\Fg)$-module generated 
by a single element $v_{\lambda}$ subject to 
the defining relation that the $v_{\lambda}$ is 
an extremal vector of weight $\lambda$ 
(see \cite[Proposition~8.2.2]{Kas94} and also \cite[\S3.1]{Kas02}). 
We know from \cite[Proposition~8.2.2]{Kas94} that 
$V(\lambda)$ admits the crystal basis 
$(\CL(\lambda),\CB(\lambda))$, and that 
$\CB(\lambda)$ is a normal $U_{q}(\Fg)$-crystal. 
If we denote by $u_{\lambda} \in \CB(\lambda)$ the element 
corresponding to the extremal vector 
$v_{\lambda} \in V(\lambda)$ of weight $\lambda$, then 
the element $u_{\lambda} \in \CB(\lambda)$ is extremal. 
%
%%%%%%%%%%%%%%%%%%
%%% rem:extmod %%%
%%%%%%%%%%%%%%%%%%
%
\begin{rem}[{see \cite[\S\S8.2 and 8.3]{Kas94}}] \label{rem:extmod}
(1) If $\lambda \in P_{+}$, 
then the extremal weight module $V(\lambda)$ is 
isomorphic to the irreducible highest weight $U_{q}(\Fg)$-module of 
highest weight $\lambda$. 
Therefore, the crystal basis $\CB(\lambda)$ is 
isomorphic, as a crystal, to 
the crystal basis of the irreducible highest 
weight $U_{q}(\Fg)$-module of highest weight $\lambda$. 

(2) For each $w \in W$, there exists an 
isomorphism $V(\lambda) \cong V(w\lambda)$ 
of $U_{q}(\Fg)$-modules 
between $V(\lambda)$ and $V(w\lambda)$. 
Also, for each $w \in W$, there exists an isomorphism 
$\CB(\lambda) \cong \CB(w\lambda)$ 
of $U_{q}(\Fg)$-crystals 
between $\CB(\lambda)$ and $\CB(w\lambda)$. 
\end{rem}

Let $J$ be a finite interval in $I=\BZ_{\ge 0}$. 
As in the case of $V(\lambda)$, 
the extremal weight $U_{q}(\Fg_{J})$-module 
$V_{J}(\lambda)$ of extremal weight $\lambda$ 
is, by definition, the integrable 
$U_{q}(\Fg_{J})$-module generated by a single element 
$v_{\lambda}$ subject to the defining relation that 
the $v_{\lambda}$ is a $J$-extremal vector 
of weight $\lambda$. 
We denote by $(\CL_{J}(\lambda),\CB_{J}(\lambda))$ 
the crystal basis of $V_{J}(\lambda)$. 
%
%%%%%%%%%%%%%%%%%%%%
%%% rem:extmod-n %%%
%%%%%%%%%%%%%%%%%%%%
%
\begin{rem} \label{rem:extmod-n}
Let $J$ be a finite interval in $I=\BZ_{\ge 0}$. 
Then, results entirely similar to 
those in Remark~\ref{rem:extmod} hold for 
extremal weight $U_{q}(\Fg_{J})$-modules 
and their crystal bases. 
If $\lambda$ is $J$-dominant (resp., $J$-antidominant), 
then $V_{J}(\lambda)$ is the  finite-dimensional 
irreducible $U_{q}(\Fg_{J})$-module of highest weight $\lambda$ 
(resp., lowest weight $\lambda$). 
Also, for $\lambda \in P$, we have
$\CB_{J}(\lambda) \cong \CB_{J}(\lambda_{J})$ 
as $U_{q}(\Fg_{J})$-crystals since $\lambda_{J} \in W_{J}\lambda$. 
Because $\CB_{J}(\lambda_{J})$ is isomorphic, 
as a $U_{q}(\Fg_{J})$-crystal, to the crystal basis of 
the finite-dimensional irreducible $U_{q}(\Fg_{J})$-module of 
highest weight $\lambda_{J}$, it follows that 
the crystal graph of $\CB_{J}(\lambda)$ is connected, 
and $\#(\CB_{J}(\lambda))_{\nu}=1$ 
for all $\nu \in W_{J}\lambda$. 
\end{rem}

To prove Proposition~\ref{prop:cb-conn} below, 
we need to recall the description of the crystal basis $\CB(\lambda)$ 
for $\lambda \in P$ from \cite{Kas94}. 
Let $(\CL(\pm\infty),\CB(\pm\infty))$ denote 
the crystal basis of $U_{q}^{\mp}(\Fg)$. 
Recall from \cite[Proposition~5.2.4]{Kas91} that 
the crystal lattice $\CL(\pm\infty)$ of $U_{q}^{\mp}(\Fg)$ 
is stable under the $\BC(q)$-algebra anti-automorphism 
$\ast:U_{q}(\Fg) \rightarrow U_{q}(\Fg)$ given by \eqref{eq:ast}. 
Furthermore, we know from \cite[Theorem~2.1.1]{Kas93} that 
the $\BC$-linear automorphism (also denoted by $\ast$)
on $\CL(\pm\infty)/q\CL(\pm\infty)$ induced by 
$\ast:\CL(\pm\infty) \rightarrow \CL(\pm\infty)$ stabilizes 
the crystal basis $\CB(\pm \infty)$, and gives rise to 
a map $\ast:\CB(\pm \infty) \rightarrow \CB(\pm \infty)$. 
Now, we set
\begin{equation*}
\CB^{\lambda}:=\CB(\infty) \otimes \CT_{\lambda} \otimes \CB(-\infty)
\quad \text{for $\lambda \in P$},
\end{equation*}
where $\CT_{\lambda}:=\{t_{\lambda}\}$ is 
a $U_{q}(\Fg)$-crystal consisting of a single element 
$t_{\lambda}$ of weight $\lambda \in P$
(see \cite[Example~1.5.3, part~2]{Kas94}), and then set 
$\ti{\CB}:=\bigoplus_{\lambda \in P} \CB^{\lambda}$. 
Also, we define a map $\ast:\ti{\CB} \rightarrow \ti{\CB}$ by:
\begin{equation*}
(b_{1} \otimes t_{\lambda} \otimes b_{2})^{\ast}=
b_{1}^{\ast} \otimes t_{-(\wt b_{1}+\lambda+\wt b_{2})} \otimes b_{2}^{\ast}
\end{equation*}
for $b_{1} \in \CB(\infty)$, $\lambda \in P$, and 
$b_{2} \in \CB(-\infty)$ (cf. \cite[Corollary~4.3.3]{Kas94}). 
It follows from \cite[Theorem~3.1.1 and 
Theorem~2.1.2\,(v)]{Kas94} that 
$\CB^{\lambda}$ is normal for all $\lambda \in P$, 
and hence so is $\ti{\CB}=\bigoplus_{\lambda \in P} \CB^{\lambda}$. 
Moreover, by \cite[Proposition~8.2.2, 
Theorem~3.1.1, and Corollary~4.3.3]{Kas94}, 
the subset 
%
%%%%%%%%%%%%%%%%%
%%% eq:sc-ext %%%
%%%%%%%%%%%%%%%%%
%
\begin{equation} \label{eq:sc-ext}
\bigl\{b \in \CB^{\lambda} \mid 
\text{$b^{\ast}$ is extremal}
\bigr\} \subset \CB^{\lambda}
\end{equation}
is a $U_{q}(\Fg)$-subcrystal of $\CB^{\lambda}$, 
and it is isomorphic, as a $U_{q}(\Fg)$-crystal, 
to the crystal basis $\CB(\lambda)$; under this isomorphism, 
the extremal element $u_{\lambda} \in \CB(\lambda)$ 
corresponds to the element $u_{\infty} \otimes t_{\lambda} 
\otimes u_{-\infty} \in \CB^{\lambda}$, where 
$u_{\pm\infty}$ is the element of $\CB(\pm\infty)$ 
corresponding to the identity element 
$1 \in U_{q}^{\mp}(\Fg)$. Thus, we can identify $\CB(\lambda)$ with 
the $U_{q}(\Fg)$-subcrystal given by \eqref{eq:sc-ext}, 
and identify $u_{\lambda} \in \CB(\lambda)$ with 
$u_{\infty} \otimes t_{\lambda} \otimes u_{-\infty} 
\in \CB^{\lambda}$. 

We have the following proposition 
(see also \cite[Proposition~3.1]{Kw-Adv} in type $A_{+\infty}$ and 
\cite[Proposition~4.1]{Kw-Ep} in type $A_{\infty}$). 
%
%%%%%%%%%%%%%%%%%%%%
%%% prop:cb-conn %%%
%%%%%%%%%%%%%%%%%%%%
%
\begin{prop} \label{prop:cb-conn}
Let $\lambda \in P$ be an integral weight. Then, 
the crystal graph of the crystal basis $\CB(\lambda)$ is connected. 
\end{prop}

\begin{proof}
While the argument in the proof of 
\cite[Proposition~3.1]{Kw-Adv} in type $A_{+\infty}$ 
(or, of \cite[Proposition~4.1]{Kw-Ep} in type $A_{\infty}$) 
still works in the case of type $B_{\infty}$, $C_{\infty}$, or $D_{\infty}$, 
we prefer to give a different proof.

For each $n \in \BZ_{\ge 0}$, let 
$(\CL_{[n]}(\pm\infty),\CBn(\pm\infty))$ denote
the crystal basis of $U_{q}^{\mp}(\Fgn) \subset U_{q}^{\mp}(\Fg)$. 
We deduce from the definitions
that $U_{q}^{\mp}(\Fgn)$ is stable under 
the Kashiwara operators $e_{i}$ and $f_{i}$, $i \in [n]$, 
on $U_{q}^{\mp}(\Fg)$, and that their restrictions to 
$U_{q}^{\mp}(\Fgn)$ are exactly the Kashiwara operators 
$e_{i}$ and $f_{i}$, $i \in [n]$, on $U_{q}^{\mp}(\Fgn)$, 
respectively. Therefore, the crystal lattice 
$\CL_{[n]}(\pm\infty)$ of $U_{q}^{\mp}(\Fgn)$ 
is identical to the $A$-submodule of the crystal lattice $\CL(\pm\infty)$ of 
$U_{q}^{\mp}(\Fg)$ generated by those elements of the form: 
$X \cdot 1 \in U_{q}^{\mp}(\Fgn) \subset U_{q}(\Fgn)$ for some 
monomial $X$ in the Kashiwara operators $e_{i}$ and $f_{i}$ for
$i \in [n]$, where $A:=\bigl\{f(q) \in \BC(q) \mid 
\text{$f(q)$ is regular at $q=0$}\bigr\}$. 
Consequently, the crystal basis $\CBn(\pm\infty)$ of $U_{q}^{\mp}(\Fgn)$ 
is identical to the subset of $\CB(\pm\infty)$ consisting of 
all elements $b$ of the form: $b=Xu_{\pm\infty}$ for some 
monomial $X$ in the Kashiwara operators 
$e_{i}$ and $f_{i}$ for $i \in [n]$. 

Now, we define a subset $\CBn^{\lambda}$ of $\CB^{\lambda}$ by: 
\begin{equation*}
\CBn^{\lambda}=
\bigl\{b_{1} \otimes t_{\lambda} \otimes b_{2} \in \CB^{\lambda}
 \mid b_{1} \in \CBn(\infty),\ b_{2} \in \CBn(-\infty)
\bigr\}.
\end{equation*}
Then, it is obvious from the tensor product rule for crystals that 
$\CBn^{\lambda}$ is a $U_{q}(\Fgn)$-crystal isomorphic 
to the tensor product 
$\CBn(\infty) \otimes \CT_{\lambda} \otimes 
\CBn(-\infty)$ of $U_{q}(\Fgn)$-crystals; namely, 
as $U_{q}(\Fgn)$-crystals, 
%
%%%%%%%%%%%%%%
%%% eq:CBn %%%
%%%%%%%%%%%%%%
%
\begin{equation} \label{eq:CBn}
\CBn^{\lambda} \stackrel{\sim}{\rightarrow} 
\CBn(\infty) \otimes \CT_{\lambda} \otimes \CBn(-\infty),
\qquad b_{1} \otimes t_{\lambda} \otimes b_{2} \mapsto 
b_{1} \otimes t_{\lambda} \otimes b_{2}.
\end{equation}
Hence, by \cite[Theorem~3.1.1 and Theorem~2.1.2\,(v)]{Kas94}, 
$\CBn^{\lambda}$ is a normal $U_{q}(\Fgn)$-crystal. 
We also remark that 
%
%%%%%%%%%%%%%%%%%
%%% eq:limCBn %%%
%%%%%%%%%%%%%%%%%
%
\begin{equation} \label{eq:limCBn}
\CB(\lambda) \cap \CBn^{\lambda} \subset 
\CB(\lambda) \cap \CBo^{\lambda} \quad 
\text{for all $n \in \BZ_{\ge 0}$, and} \quad
\CB(\lambda)=\bigcup_{n \ge 0} 
 \bigl(\CB(\lambda) \cap \CBn^{\lambda}\bigr).
\end{equation}

\begin{claim*}
The crystal basis $\CBn(\lambda)$ of the extremal weight 
$U_{q}(\Fgn)$-module of extremal weight $\lambda$ is 
isomorphic, as a $U_{q}(\Fgn)$-crystal, 
to $\CB(\lambda) \cap \CBn^{\lambda}$. 
\end{claim*}

\noindent 
{\it Proof of Claim.}
We have a $\BC(q)$-algebra anti-automorphism 
$\star:U_{q}(\Fgn) \rightarrow U_{q}(\Fgn)$ defined by:
%
%%%%%%%%%%%%%%%
%%% eq:star %%%
%%%%%%%%%%%%%%%
%
\begin{equation} \label{eq:star}
\begin{cases}
(q^{h})^{\star}=q^{-h} 
  & \text{for $h \in P^{\vee}$}, \\[1.5mm]
x_{i}^{\star}=x_{i},\ 
y_{i}^{\star}=y_{i}
  & \text{for $i \in [n]$}.
\end{cases}
\end{equation}
As in the case of $\CB(\pm\infty)$, 
this $\BC(q)$-algebra anti-automorphism 
$\star:U_{q}(\Fgn) \rightarrow U_{q}(\Fgn)$
induces a map $\star:\CBn(\pm\infty) \rightarrow 
\CBn(\pm\infty)$. We set $\ti{\CB}_{[n]}:=
\bigoplus_{\lambda \in P} \CBn^{\lambda} \subset \ti{\CB}$, 
which is a normal $U_{q}(\Fgn)$-crystal, and then 
define $\star:\ti{\CB}_{[n]} \rightarrow \ti{\CB}_{[n]}$ by:
\begin{equation*}
(b_{1} \otimes t_{\lambda} \otimes b_{2})^{\star}=
b_{1}^{\star} \otimes t_{-(\wt b_{1}+\lambda+\wt b_{2})} \otimes b_{2}^{\star}
\end{equation*}
for $b_{1} \in \CBn(\infty)$, $\lambda \in P$, and 
$b_{2} \in \CBn(-\infty)$. Then, we see from 
\cite[Proposition~8.2.2, Theorem~3.1.1, and Corollary~4.3.3]{Kas94}, 
together with \eqref{eq:CBn}, that 
%
%%%%%%%%%%%%%%%%%%%
%%% eq:sc-ext-n %%%
%%%%%%%%%%%%%%%%%%%
%
\begin{equation} \label{eq:sc-ext-n}
\bigl\{
 b \in \CBn^{\lambda} \mid 
 \text{$b^{\star}$ is $[n]$-extremal}
\bigr\} 
\subset \CBn^{\lambda} \quad (\subset \CB^{\lambda})
\end{equation}
is a $U_{q}(\Fgn)$-subcrystal of $\CBn^{\lambda}$ 
isomorphic to the crystal basis 
$\CBn(\lambda)$. We identify $\CBn(\lambda)$ 
with the $U_{q}(\Fgn)$-subcrystal of $\CBn^{\lambda}$ 
given by \eqref{eq:sc-ext-n}. 

Also, observe that the $\BC(q)$-algebra anti-automorphism 
$\star:U_{q}(\Fgn) \rightarrow U_{q}(\Fgn)$ given by 
\eqref{eq:star} is identical to the restriction to $U_{q}(\Fgn)$ 
of the $\BC(q)$-algebra anti-automorphism 
$\ast:U_{q}(\Fg) \rightarrow U_{q}(\Fg)$ given by \eqref{eq:ast}.
Consequently, the map 
$\star:\CBn(\pm\infty) \rightarrow \CBn(\pm\infty)$ is 
identical to the restriction to $\CBn(\pm\infty)$ of the map 
$\ast:\CB(\pm\infty) \rightarrow \CB(\pm\infty)$, and hence 
the map $\star:\ti{\CB}_{[n]} \rightarrow \ti{\CB}_{[n]}$ is 
identical to the restriction to $\ti{\CB}_{[n]}$ of the map 
$\ast:\ti{\CB} \rightarrow \ti{\CB}$. Therefore, we have
%
%%%%%%%%%%%%%%%%%%
%%% eq:ext-inc %%% 
%%%%%%%%%%%%%%%%%%
%
\begin{equation} \label{eq:ext-inc}
\CBn(\lambda)
=
\bigl\{
 b \in \CBn^{\lambda} \mid 
 \text{$b^{\ast}$ is $[n]$-extremal}
\bigr\} \supset \CB(\lambda) \cap \CBn^{\lambda}.
\end{equation}
Because we know from Remark~\ref{rem:extmod-n} with $J=[n]$ that 
the crystal graph of the $U_{q}(\Fgn)$-crystal $\CBn(\lambda)$ 
is connected, we conclude from \eqref{eq:ext-inc} 
that $\CBn(\lambda)=\CB(\lambda) \cap \CBn^{\lambda}$. 
This proves the claim. \nolinebreak \bqed

\vsp

We now complete the  proof of 
Proposition~\ref{prop:cb-conn}. 
Take $b_{1},\,b_{2} \in \CB(\lambda)$ arbitrarily. 
We show that there exists a monomial $X$ in the 
Kashiwara operators $e_{i}$ and $f_{i}$ for $i \in I$ 
such that $Xb_{1}=b_{2}$. 
It follows from \eqref{eq:limCBn} that 
there exists $n \in \BZ_{\ge 0}$ such that 
$b_{1}$ and $b_{2}$ are both contained in 
$\CB(\lambda) \cap \CBn^{\lambda}$. 
Since $\CB(\lambda) \cap 
\CBn^{\lambda} \cong \CBn(\lambda)$ 
as $U_{q}(\Fgn)$-crystals by the claim above, 
we deduce from Remark~\ref{rem:extmod-n} with $J=[n]$ that 
there exists a monomial $X$ in the 
Kashiwara operators $e_{i}$ and $f_{i}$ for 
$i \in [n]$ such that $Xb_{1}=b_{2}$. 
This finishes the proof of Proposition~\ref{prop:cb-conn}.
\end{proof}
%
%%%%%%%%%%%%%%%%%%%
%%% prop:wtmult %%%
%%%%%%%%%%%%%%%%%%%
%
\begin{prop} \label{prop:wtmult}
Let $\lambda \in P$ be an integral weight. 
For each $w \in W$, we have 
$\CB(\lambda)_{w\lambda}=\bigl\{S_{w}u_{\lambda}\bigr\}$. 
\end{prop}

\begin{proof}
By using the action of the Weyl group $W$ on $\CB(\lambda)$, 
we are reduced to the case $w=e$. 
Now, we show that $\CB(\lambda)_{\lambda}=\bigl\{u_{\lambda}\bigr\}$. 
Let $b \in \CB(\lambda)_{\lambda}$. Then, 
there exists $n \in \BZ_{\ge 0}$ such that 
$b \in \CB(\lambda) \cap \CBn^{\lambda}$ by \eqref{eq:limCBn}; 
note that $u_{\lambda} \in \CB(\lambda) \cap \CBn^{\lambda}$. 
Since $\CB(\lambda) \cap \CBn^{\lambda} \cong \CBn(\lambda)$ as 
$U_{q}(\Fgn)$-crystals by the Claim in the proof of 
Proposition~\ref{prop:cb-conn}, it follows from 
Remark~\ref{rem:extmod-n} with $J=[n]$ that 
$\#(\CB(\lambda) \cap \CBn^{\lambda})_{\nu}=1$ 
for every $\nu \in \Wn\lambda$. 
In particular, we have 
$\#(\CB(\lambda) \cap \CBn^{\lambda})_{\lambda}=1$, 
which implies that $b=u_{\lambda}$, as desired. 
\end{proof}
%
%%%%%%%%%%%%%%
%%% rem:nf %%%
%%%%%%%%%%%%%%
%
\begin{rem} \label{rem:nf}
For a general $\xi \in P$, 
the set $\CB(\lambda)_{\xi}$ is not of finite cardinality. 
\end{rem}
%
%%%%%%%%%%%%%%%%%%%%%
%%% prop:non-isom %%%
%%%%%%%%%%%%%%%%%%%%%
%
\begin{prop} \label{prop:non-isom}
Let $\lambda,\,\mu \in P$ be integral weights 
(not necessarily dominant). 
Then, $\CB(\lambda) \cong \CB(\mu)$ as $U_{q}(\Fg)$-crystals 
if and only if $\lambda \in W\mu$.
\end{prop}

\begin{proof}
We know from Remark~\ref{rem:extmod}\,(2) that 
if $\lambda \in W\mu$, then $\CB(\lambda) \cong \CB(\mu)$ 
as $U_{q}(\Fg)$-crystals. 
Conversely, suppose that 
$\CB(\lambda) \cong \CB(\mu)$ as $U_{q}(\Fg)$-crystals. 
Let $b \in \CB(\lambda)$ be the element corresponding 
to $u_{\mu} \in \CB(\mu)$ under the isomorphism 
$\CB(\lambda) \cong \CB(\mu)$ of  $U_{q}(\Fg)$-crystals; 
note that $b \in \CB(\lambda)$ is an extremal element of 
weight $\mu$ since so is $u_{\mu}$. 
Take $n \in \BZ_{\ge 0}$ such that 
$b \in \CB(\lambda) \cap \CBn^{\lambda}$, and let 
$w \in \Wn$ be such that $w\mu=\mun$ 
(recall that $\mun$ is the unique element of 
$\Wn\mu$ that is $[n]$-dominant). 
Then, since $S_{w}$ is defined by using only 
the Kashiwara operators $e_{i}$ and $f_{i}$ for $i \in [n]$, 
we see that $S_{w}b$ is contained in $\CB(\lambda) \cap \CBn^{\lambda}$. 
Furthermore, since $b \in \CB(\lambda)$ is an extremal element of 
weight $\mu$, and since $\wt(S_{w}b)=w\mu=\mun$ is $[n]$-dominant, 
we deduce from Definition~\ref{dfn:extremal}\,(1) that $S_{w}b \in 
\CB(\lambda) \cap \CBn^{\lambda}$ is an $[n]$-maximal element. 
Now we recall that, by the Claim in the proof of 
Proposition~\ref{prop:cb-conn} and 
by Remark~\ref{rem:extmod-n} with $J=[n]$, 
the $U_{q}(\Fgn)$-crystal $\CB(\lambda) \cap \CBn^{\lambda}$ is 
isomorphic to the crystal basis of 
the finite-dimensional irreducible $U_{q}(\Fgn)$-module of 
highest weight $\lamn$. 
Therefore, we deduce that the element $S_{w}b$ is the (unique) 
$[n]$-maximal element of $\CB(\lambda) \cap \CBn^{\lambda}$ 
of weight $\lamn$, and hence that $\mun=w\mu=\wt (S_{w}b)=
\lamn \in \Wn\lambda$, which implies that 
$\lambda \in W\mu$. Thus we have proved the proposition. 
\end{proof}

%==============================%
%     START SUBSECTION 0303    %
%==============================%
%
\subsection{Lakshmibai-Seshadri paths and crystal structure on them.}
\label{subsec:path}
In this subsection, following \cite{Lit-Inv} and \cite{Lit-Ann}, 
we review basic facts about Lakshmibai-Seshadri paths and 
crystal structure on them; note that all results in \cite{Lit-Inv} and 
\cite{Lit-Ann} that we use in this paper remain valid in 
the case of infinite rank affine Lie algebras. 
We take and fix an arbitrary (not necessarily dominant) 
integral weight $\lambda \in P$. 
%
%%%%%%%%%%%%%%%%%%
%%% dfn:bruhat %%%
%%%%%%%%%%%%%%%%%%
%
\begin{dfn} \label{dfn:bruhat}
Let $\mu,\,\nu$ be elements of $W\lambda$. 
We write $\mu > \nu$ if there exist a sequence 
$\mu=\xi_{0},\,\xi_{1},\,\dots,\,\xi_{l}=\nu$ 
of elements of $W\lambda$ and 
a sequence $\beta_{1},\,\dots,\,\beta_{l} \in \Delta^{+}$ 
of positive roots for $\Fg$, with $l \ge 1$, such that 
$\xi_{m}=r_{\beta_{m}}(\xi_{m-1})$ and 
$\pair{\xi_{m-1}}{\beta_{m}^{\vee}} \in \BZ_{< 0}$ 
for all $1 \le m \le l$, where 
$r_{\beta} \in W$ denotes the reflection with respect 
to a root $\beta \in \Delta$. 
In this case, the sequence 
$\xi_{0},\,\xi_{1},\,\dots,\,\xi_{l}$ above is 
called a chain for $(\mu,\nu)$ in $W\lambda$. 
If $\mu > \nu$, then 
we define $\dist(\mu,\nu)$ to 
be the maximum length $l$ of 
all possible chains for $(\mu,\nu)$ in $W\lambda$ . 
\end{dfn}
%
%%%%%%%%%%%%%%
%%% rem:Q+ %%%
%%%%%%%%%%%%%%
%
\begin{rem} \label{rem:Q+}
Let $\mu,\,\nu$ be elements of $W\lambda$ such that $\mu > \nu$. 
If $\mu=\xi_{0},\,\xi_{1},\,\dots,\,\xi_{l}=\nu$ is 
a chain for $(\mu,\nu)$ in $W\lambda$, 
with corresponding positive roots 
$\beta_{1},\,\dots,\,\beta_{l} \in \Delta^{+}$, 
then we have
%
%%%%%%%%%%%%%
%%% eq:Q+ %%%
%%%%%%%%%%%%%
%
\begin{equation} \label{eq:Q+}
\nu-\mu \in \sum_{m=1}^{l}\BZ_{> 0} \beta_{m} \subset 
\sum_{i \in I} \BZ_{\ge 0}\alpha_{i} \setminus \{0\}.
\end{equation}
\end{rem}
%
%%%%%%%%%%%%%%%%%%
%%% dfn:achain %%%
%%%%%%%%%%%%%%%%%%
%
\begin{dfn} \label{dfn:achain}
Let $\mu,\,\nu$ be elements of $W\lambda$ 
such that $\mu > \nu$, 
and let $0 < a < 1$ be a rational number. 
An $a$-chain for $(\mu,\nu)$ in $W\lambda$ is, by definition, 
a sequence $\mu=\zeta_{0} > \zeta_{1} > 
\dots > \zeta_{l}=\nu$ of elements of $W\lambda$, 
with $l \ge 1$, such that 
$\dist(\zeta_{m-1},\zeta_{m})=1$ and 
$\pair{\zeta_{m-1}}{\beta_{m}^{\vee}} \in a^{-1}\BZ_{< 0}$ 
for all $1 \le m \le l$, 
where $\beta_{m} \in \Delta^{+}$ is 
the positive root for $\Fg$ corresponding to a chain for 
$(\zeta_{m-1},\zeta_{m})$ in $W\lambda$. 
\end{dfn}
%
%%%%%%%%%%%%%%
%%% dfn:LS %%%
%%%%%%%%%%%%%%
%
\begin{dfn}[{\cite[\S4]{Lit-Ann}}] \label{dfn:LS}
Let $\lambda \in P$, and let $(\ud{\nu}\,;\,\ud{a})$ be a pair of 
a sequence $\ud{\nu} : \nu_{1} > \nu_{2} > \dots > \nu_{s}$ of 
elements of $W\lambda$ and a sequence $\ud{a} : 0=a_{0}
< a_{1} < \dots < a_{s}=1$ of rational numbers, with $s \ge 1$. 
The pair $(\ud{\nu}\,;\,\ud{a})$ is called 
a Lakshmibai-Seshadri path (LS path for short) of 
shape $\lambda$ (for $\Fg$), 
if for every $u=1,\,2,\,\dots,\,s-1$, there exists 
an $a_{u}$-chain for $(\nu_{u},\nu_{u+1})$ in $W\lambda$. 
We denote by $\BB(\lambda)$ 
the set of all LS paths of shape $\lambda$. 
\end{dfn}
%
%%%%%%%%%%%%%%
%%% rem:LS %%%
%%%%%%%%%%%%%%
%
\begin{rem} \label{rem:LS}
(1) It is easily seen from the definition that 
$\pi_{\nu}:=(\nu\,;\,0,\,1) \in \BB(\lambda)$ 
for all $\nu \in W\lambda$. 

(2) It is obvious from the definitions that 
$\BB(w\lambda)=\BB(\lambda)$ for all $w \in W$. 
\end{rem}

Let $J$ be a finite interval in $I=\BZ_{\ge 0}$. 
We define an LS path of shape $\lambda$ for 
$\Fg_{J}$ as follows. 
First we define a partial order $>_{J}$ on $W_{J}\lambda$, 
which is entirely similar to the one 
in Definition~\ref{dfn:bruhat}, 
with $W$ replaced by $W_{J}$, 
and $\Delta^{+}$ by $\Delta_{J}^{+}=
\Delta^{+} \cap \bigoplus_{i \in J}\BZ\alpha_{i}$.
If $\mu,\,\nu \in W_{J}\lambda$ are such that 
$\mu >_{J} \nu$, then we denote by 
$\dist_{J}(\cdot\,,\,\cdot)$ the maximal length 
of all possible chains for $(\mu,\,\nu)$ in $W_{J}\lambda$. 
Next, for a rational number $0 < a < 1$, 
we define $a$-chains in $W_{J}\lambda$ 
as in Definition~\ref{dfn:achain}, with 
$W$ replaced by $W_{J}$, $\dist(\cdot\,,\,\cdot)$ 
by $\dist_{J}(\cdot\,,\,\cdot)$, and 
$\Delta^{+}$ by $\Delta_{J}^{+}$. 
Finally, we define an LS path of shape $\lambda$ for $\Fg_{J}$ 
in the same way as in Definition~\ref{dfn:LS}, 
with $W$ replaced by $W_{J}$. 
We denote by $\BB_{J}(\lambda)$ 
the set of all LS paths of shape $\lambda$ for $\Fg_{J}$. 
%
%%%%%%%%%%%%%%%
%%% lem:BBn %%%
%%%%%%%%%%%%%%%
%
\begin{lem} \label{lem:BBn}
Let $J$ be a finite interval of $I=\BZ_{\ge 0}$. 
Then, the set $\BB_{J}(\lambda)$ is identical to 
the subset of $\BB(\lambda)$ consisting of all 
elements $\pi=(\nu_{1},\,\nu_{2},\,\dots,\,\nu_{s}\,;\,
a_{0},\,a_{1},\,\dots,\,a_{s})$, with $s \ge 1$, satisfying 
the condition that $\nu_{u} \in W_{J}\lambda$ for all $1 \le u \le s$. 
\end{lem}

\begin{proof}
First we show that the partial order 
$>_{J}$ on $W_{J}\lambda$ coincides with 
the restriction to $W_{J}\lambda$ 
of the partial order $>$ on $W\lambda$. 
Let $\mu,\,\nu \in W_{J}\lambda$. 
It is obvious from the definitions that 
if $\mu >_{J} \nu$, then $\mu > \nu$. 
Conversely, suppose that $\mu > \nu$, and let 
$\mu=\xi_{0},\,\xi_{1},\,\dots,\,\xi_{l}=\nu$ be 
a chain for $(\mu,\nu)$ in $W\lambda$, 
with corresponding positive roots 
$\beta_{1},\,\dots,\,\beta_{l} \in \Delta^{+}$. 
Then, by \eqref{eq:Q+}, we have 
$\nu-\mu \in \sum_{m=1}^{l}\BZ_{> 0} \beta_{m}$.
Also, since $\mu,\,\nu \in W_{J}\lambda$, 
we deduce that $\nu-\mu \in \bigoplus_{i \in J}\BZ\alpha_{i}$. 
Therefore, if follows that 
$\beta_{m} \in \Delta^{+} \cap 
\bigoplus_{i \in J}\BZ\alpha_{i}=\Delta_{J}^{+}$ 
for all $1 \le m \le l$, and hence 
$\xi_{m} \in W_{J}\lambda$ for all $0 \le m \le l$.
This implies that $\mu >_{J} \nu$, as desired. 
Furthermore, we see from the argument above that 
if $\mu,\,\nu \in W_{J}\lambda$ are such that 
$\mu > \nu$ (or equivalently, $\mu >_{J} \nu$), 
then $\dist(\mu,\,\nu)=\dist_{J}(\mu,\,\nu)$. 

Now, let $\mu,\,\nu \in W_{J}\lambda$ be such that $\mu > \nu$, 
and let $0 \le a \le 1$ be a rational number. 
Then we deduce from what is shown above, along with 
the definitions of $a$-chains
in $W\lambda$ and of those in $W_{J}\lambda$, that 
there exists an $a$-chain for $(\mu,\,\nu)$ in $W\lambda$ 
if and only if there exists an $a$-chain for $(\mu,\,\nu)$ 
in $W_{J}\lambda$. 
Consequently, it follows immediately from the definitions of 
an LS path of shape $\lambda$ for $\Fg$ and 
of the one for $\Fg_{J}$ that
\begin{equation*}
\BB_{J}(\lambda)= \bigl\{
 (\nu_{1},\,\nu_{2},\,\dots,\,\nu_{s}\,;\,
 a_{0},\,a_{1},\,\dots,\,a_{s}) \in \BB(\lambda) \mid 
 \text{$s \ge 1$, and 
 $\nu_{u} \in W_{J}\lambda,\,1 \le u \le s$}
  \bigr\}.
\end{equation*}
This proves the lemma. 
\end{proof}

It follows from Lemma~\ref{lem:BBn} that 
%
%%%%%%%%%%%%%%%%%
%%% eq:limBBn %%%
%%%%%%%%%%%%%%%%%
%
\begin{equation} \label{eq:limBBn}
\BBn(\lambda) \subset \BBo(\lambda) \quad 
\text{for all $n \in \BZ_{\ge 0}$, and} \quad
\BB(\lambda)=\bigcup_{n \ge 0} \BBn(\lambda).
\end{equation}

Now, for real numbers $a,\,b \in \BR$ with $a \le b$, 
we set $[a,\,b]_{\BR}:=\bigl\{t \in \BR \mid a \le t \le b\bigr\}$. 
A path is, by definition, 
a piecewise linear, continuous map 
$\pi:[0,\,1]_{\BR} \rightarrow \BR \otimes_{\BZ} P$ 
from $[0,\,1]_{\BR}$ to $\BR \otimes_{\BZ} P$ such that $\pi(0)=0$.
Let $\pi=(\ud{\nu}\,;\,\ud{a})$ be a pair of a sequence 
$\ud{\nu}:\nu_{1},\,\nu_{2},\,\dots,\,\nu_{s}$ of 
integral weights in $P$ and a sequence 
$\ud{a}:0=a_{0} < a_{1} < \cdots < a_{s}=1$ of 
rational numbers. 
We associate to the pair $\pi=(\ud{\nu}\,;\,\ud{a})$ 
the following path $\pi:[0,\,1]_{\BR} \rightarrow \BR \otimes_{\BZ} P$: 
%
%%%%%%%%%%%%%%%
%%% eq:path %%%
%%%%%%%%%%%%%%%
%
\begin{equation} \label{eq:path}
\pi(t)=\sum_{v=1}^{u-1} (a_{v}-a_{v-1})\nu_{v}+
(t-a_{u-1})\nu_{u} 
\quad
\text{for \, } 
a_{u-1} \le t \le a_{u}, \ 1 \le u \le s. 
\end{equation}
It is easily seen that 
for pairs $\pi=(\ud{\nu}\,;\,\ud{a}) \in \BB(\lambda)$ and 
$\pi'=(\ud{\nu}'\,;\,\ud{a}') \in \BB(\lambda)$, 
$\pi$ is identical to $\pi'$ 
(i.e., $\ud{\nu}=\ud{\nu}'$ and $\ud{a}=\ud{a}'$) 
if and only if $\pi(t)=\pi'(t)$ for all $t \in [0,\,1]_{\BR}$. 
Hence we can identify the set $\BB(\lambda)$ of all LS paths of shape $\lambda$ 
with the set of corresponding paths via \eqref{eq:path}. Note that 
the element $\pi_{\nu}=(\nu\,;\,0,\,1) \in \BB(\lambda)$ for $\nu \in W\lambda$ 
corresponds to the straight line path connecting $0 \in \BR \otimes_{\BZ} P$ 
to $\nu \in P \subset \BR \otimes_{\BZ} P$. 

%%%%%%%%%%%%%%%
%%% rem:lcm %%%
%%%%%%%%%%%%%%%
%
\begin{rem} \label{rem:lcm}
Let $\pi=(\nu_{1},\,\nu_{2},\,\dots,\,\nu_{s}\,;\,
 a_{0},\,a_{1},\,\dots,\,a_{s})$ be as above. 
If $a_{u}=u/s$ for all $0 \le u \le s$, then 
we write the corresponding path simply as: 
$\pi=(\nu_{1},\,\nu_{2},\,\dots,\,\nu_{s})$.
Now, noting that the set $\bigl\{\pair{\lambda}{\beta^{\vee}} \mid 
\beta \in \Delta\bigr\} \subset \BZ$ is a finite set 
(see \eqref{eq:binf-rs}--\eqref{eq:dinf-rs}), 
we define $N=N_{\lambda} \in \BZ_{\ge 1}$ to 
be the least common multiple of all nonzero integers 
in $\bigl\{\pair{\lambda}{\beta^{\vee}} \mid 
\beta \in \Delta\bigr\} \cup \bigl\{1\bigr\}$. 
It follows from the definition of $a$-chains that 
if $\pi=(\nu_{1},\,\nu_{2},\,\dots,\,\nu_{s}\,;\,
 a_{0},\,a_{1},\,\dots,\,a_{s}) \in \BB(\lambda)$, then 
$Na_{u} \in \BZ$ for all $0 \le u \le s$. 
Consequently, the path corresponding to $\pi$ is 
identical to the path
\begin{equation*}
\bigl(
 \underbrace{\nu_{1},\,\dots,\,\nu_{1}}_{%
 \text{$b_{1}$ times}},\,
 \underbrace{\nu_{2},\,\dots,\,\nu_{2}}_{%
 \text{$b_{2}$ times}},\,\dots,\,
 \underbrace{\nu_{s},\,\dots,\,\nu_{s}}_{%
 \text{$b_{s}$ times}}
\bigr),
\end{equation*}
where $b_{u}:=N(a_{u}-a_{u-1})$ for $1 \le u \le s$. 
\end{rem}

We define a $U_{q}(\Fg)$-crystal structure 
on the set $\BB(\lambda)$ of all LS paths 
of shape $\lambda$ as follows. 
First, recalling from 
\cite[Lemma 4.5 a)]{Lit-Ann} that $\pi(1) \in P$ 
for all $\pi \in \BB(\lambda)$, we define 
$\wt:\BB(\lambda) \rightarrow P$ by: $\wt \pi=\pi(1)$.
%
%%%%%%%%%%%%%%
%%% rem:wt %%%
%%%%%%%%%%%%%%
%
\begin{rem} \label{rem:wt}
For every $\pi \in \BB(\lambda)$, the level of $\wt \pi=\pi(1)$ is 
equal to the level $L_{\lambda}$ of $\lambda$. 
Indeed, it follows from Remark~\ref{rem:lcm} that 
$\pi \in \BB(\lambda)$ can be written as: 
$\pi=(\nu_{1},\,\nu_{2},\,\dots,\,\nu_{N})$ for some 
$\nu_{1},\,\nu_{2},\,\dots,\,\nu_{N} \in W\lambda$. 
Therefore, by the definition of $\wt:\BB(\lambda) \rightarrow P$, 
we have
%
%%%%%%%%%%%%%%
%%% eq:wtN %%%
%%%%%%%%%%%%%%
%
\begin{equation} \label{eq:wtN}
\nu:=\wt \pi=\pi(1)=\frac{1}{N} \sum_{M=1}^{N} \nu_{M}.
\end{equation}
If we take $n \in \BZ_{\ge 3}$ such that 
$\nu_{M} \in \Wn\lambda$ for all $1 \le M \le N$, and 
such that $\col{m}=L_{\lambda}$ for all $m \ge n+1$, 
then we deduce by using \eqref{eq:binf-r}--\eqref{eq:dinf-r} 
that for all $m \ge n+1$ and $1 \le M \le N$, 
$\co{\nu_{M}}{m}=\col{m}=L_{\lambda}$, and hence 
\begin{equation*}
\con{m}=\frac{1}{N} \sum_{M=1}^{N} \co{\nu_{M}}{m}=
\frac{1}{N} \sum_{M=1}^{N} L_{\lambda}=L_{\lambda}. 
\end{equation*}
This implies that $L_{\nu}=L_{\lambda}$. 
In particular, if $\lambda \in E$, 
then $\wt \pi \in E$ for all $\pi \in \BB(\lambda)$; 
recall that $E \subset P$ is identical to the set of 
integral weights of level zero.
\end{rem}

Next, for $\pi \in \BB(\lambda)$ and $i \in I$, 
we define $e_{i}\pi$ as follows. Set 
\begin{equation*}
H^{\pi}_{i}(t):=\pair{\pi(t)}{h_{i}}
  \quad \text{for $t \in [0,\,1]_{\BR}$}, \quad \text{and} \quad
m^{\pi}_{i}
  :=\min \bigl\{H^{\pi}_{i}(t) \mid t \in [0,\,1]_{\BR}\bigr\}. 
\end{equation*}
We know from \cite[Lemma 4.5 d)]{Lit-Ann} 
that $m^{\pi}_{i} \in \BZ_{\le 0}$. 
If $m^{\pi}_{i}=0$, then we set $e_{i}\pi:=\bzero$. 
If $m^{\pi}_{i} \le -1$, then we set 
\begin{equation*}
(e_{i}\pi)(t):=
\begin{cases}
\pi(t) & \text{if $0 \le t \le t_{0}$}, \\[2mm]
\pi(t_{0})+r_{i}(\pi(t)-\pi(t_{0}))
       & \text{if $t_{0} \le t \le t_{1}$}, \\[2mm]
\pi(t)+\alpha_{i} & \text{if $t_{1} \le t \le 1$},
\end{cases}
\end{equation*}
where 
\begin{equation*}
\begin{array}{l}
t_{1}:=\min\bigl\{t \in [0,\,1]_{\BR} \mid 
       H^{\pi}_{i}(t)=m^{\pi}_{i} \bigr\}, \\[3mm]
t_{0}:=\max\bigl\{t \in [0,\,t_{1}]_{\BR} \mid
       H^{\pi}_{i}(t)=m^{\pi}_{i}+1\bigr\}. 
\end{array}
\end{equation*}
It follows from \cite[Corollary~2 in \S4]{Lit-Ann} 
that $e_{i}\pi \in \BB(\lambda) \cup \{\bzero\}$. Thus, 
we obtain a map $e_{i}:\BB(\lambda) \rightarrow 
\BB(\lambda) \cup \{\bzero\}$. 
Similarly, for $\pi \in \BB(\lambda)$ and $i \in I$, 
we define $f_{i}\pi$ as follows. 
Note that $H^{\pi}_{i}(1)-m^{\pi}_{i} \in \BZ_{\ge 0}$.
If $H^{\pi}_{i}(1)-m^{\pi}_{i}=0$, 
then we set $f_{i}\pi:=\bzero$. 
If $H^{\pi}_{i}(1)-m^{\pi}_{i} \ge 1$, then we set
\begin{equation*}
(f_{i}\pi)(t):=
\begin{cases}
\pi(t) & \text{if $0 \le t \le t_{0}$}, \\[2mm]
\pi(t_{0})+r_{i}(\pi(t)-\pi(t_{0}))
       & \text{if $t_{0} \le t \le t_{1}$}, \\[2mm]
\pi(t)-\alpha_{i} & \text{if $t_{1} \le t \le 1$},
\end{cases}
\end{equation*}
where 
\begin{equation*}
\begin{array}{l}
t_{0}:=\max\bigl\{t \in [0,\,1]_{\BR} \mid 
       H^{\pi}_{i}(t)=m^{\pi}_{i}\bigr\}, \\[3mm]
t_{1}:=\min\bigl\{t \in [t_{0},\,1]_{\BR} \mid
       H^{\pi}_{i}(t) = m^{\pi}_{i}+1 \bigr\}.
\end{array}
\end{equation*}
It follows from \cite[Corollary~2 in \S4]{Lit-Ann} 
that $f_{i}\pi \in \BB(\lambda) \cup \{\bzero\}$. 
Thus, we obtain a map $f_{i}:\BB(\lambda) \rightarrow 
\BB(\lambda) \cup \{\bzero\}$. Finally, 
for $\pi \in \BB(\lambda)$ and $i \in I$, we set
\begin{equation*}
\ve_{i}(\pi):=
 \max\bigl\{k \ge 0 \mid e_{i}^{k}\pi \ne \bzero\bigr\}, \qquad
\vp_{i}(\pi):=
 \max\bigl\{k \ge 0 \mid f_{i}^{k}\pi \ne \bzero\bigr\}.
\end{equation*}

\begin{thm}[{\cite[\S2 and \S4]{Lit-Ann}}] \label{thm:path}
The set $\BB(\lambda)$, together with the maps 
$\wt:\BB(\lambda) \rightarrow P$, 
$e_{i},\,f_{i}:\BB(\lambda) \rightarrow 
\BB(\lambda) \cup \{\bzero\}$ for $i \in I$, and 
$\ve_{i},\,\vp_{i}:\BB(\lambda) \rightarrow \BZ_{\ge 0}$
for $i \in I$, forms a $U_{q}(\Fg)$-crystal. 
\end{thm}

Let $J$ be a finite interval of $I=\BZ_{\ge 0}$; 
recall from Lemma~\ref{lem:BBn} that 
$\BB_{J}(\lambda)$ is a subset of $\BB(\lambda)$. 
We see from \cite[Corollary~2 in \S4]{Lit-Ann} that
if $\pi \in \BB_{J}(\lambda) \subset \BB(\lambda)$, 
then $e_{i}\pi,\,f_{i}\pi \in \BB_{J}(\lambda) \cup \{\bzero\}$ 
for all $i \in J$. Consequently, the set $\BB_{J}(\lambda)$, 
together with the restrictions to $\BB_{J}(\lambda)$ 
of the maps $\wt:\BB(\lambda) \rightarrow P$, 
$e_{i},\,f_{i}:\BB(\lambda) \rightarrow 
\BB(\lambda) \cup \{\bzero\}$ for $i \in J$, and 
$\ve_{i},\,\vp_{i}:\BB(\lambda) \rightarrow \BZ_{\ge 0}$
for $i \in J$, forms a $U_{q}(\Fg_{J})$-crystal. 
%
%%%%%%%%%%%%%%%%
%%% rem:LS-n %%%
%%%%%%%%%%%%%%%%
%
\begin{rem} \label{rem:LS-n}
It follows from Remark~\ref{rem:LS}\,(2) and Lemma~\ref{lem:BBn} that 
$\BB_{J}(\lambda)=\BB_{J}(w\lambda)$ for all $w \in W_{J}$. 
Hence we have $\BB_{J}(\lambda)=\BB_{J}(\lambda_{J})$. 
Also, we know from \cite[Corollary~6.4.27]{J} or 
\cite[Theorem~4.1]{Kas96} that $\BB_{J}(\lambda_{J})$ is 
isomorphic, as a $U_{q}(\Fg_{J})$-crystal, to 
the crystal basis of the finite-dimensional irreducible 
$U_{q}(\Fg_{J})$-module of highest weight $\lambda_{J}$. 
Combining this fact and Remark~\ref{rem:extmod-n} yields 
the following isomorphism of $U_{q}(\Fg_{J})$-crystals: 
%
%%%%%%%%%%%%%%%%%
%%% eq:isom-n %%%
%%%%%%%%%%%%%%%%%
%
\begin{equation} \label{eq:isom-n}
\CB_{J}(\lambda) \cong \CB_{J}(\lambda_{J}) \cong 
\BB_{J}(\lambda_{J})=\BB_{J}(\lambda).
\end{equation}
In addition, the element $\pi_{\lambda_{J}}$ is contained in 
$\BB_{J}(\lambda)=\BB_{J}(\lambda_{J})$ by Remark~\ref{rem:LS}\,(1) and 
Lemma~\ref{lem:BBn}, and it is the (unique) maximal element of 
weight $\lambda_{J}$ in $\BB_{J}(\lambda)=\BB_{J}(\lambda_{J})$. 
\end{rem}

%==============================%
%     START SUBSECTION 0304    %
%==============================%
%
\subsection{Isomorphism theorem.}
\label{subsec:isom}

The following theorem gives a combinatorial realization of the 
crystal basis $\CB(\lambda)$ of the extremal weight module $V(\lambda)$ of 
extremal weight $\lambda$ in a unified way that is independent 
of the level $L_{\lambda}$ of $\lambda \in P$; compare with 
the related results: 
\cite[Theorem~3.5]{Kw-Adv} (in type $A_{+\infty}$), 
\cite[Corollary~4.7 along with Remark~4.8]{Kw-Ep} (in type $A_{\infty}$), 
and also \cite[\S5]{Lec-TAMS} (in all types 
$A_{+\infty}$, $B_{\infty}$, $C_{\infty}$, $D_{\infty}$, 
but for $\lambda \in E$). 
%
%%%%%%%%%%%%%%%%
%%% thm:isom %%%
%%%%%%%%%%%%%%%%
%
\begin{thm} \label{thm:isom}
Let $\lambda \in P$ be a (not necessarily dominant)
integral weight. Then, the crystal basis 
$\CB(\lambda)$ of the extremal weight $U_{q}(\Fg)$-module $V(\lambda)$
of extremal weight $\lambda$ is isomorphic, 
as a $U_{q}(\Fg)$-crystal, to the crystal $\BB(\lambda)$ 
consisting of all LS paths of shape $\lambda$. 
\end{thm}

\begin{proof}
We use the notation in the proof of 
Proposition~\ref{prop:cb-conn}. It follows from 
the Claim in the proof of Proposition~\ref{prop:cb-conn} and 
\eqref{eq:isom-n} with $J=[n]$ that for each $n \in \BZ_{\ge 0}$, 
there exists an isomorphism
\begin{equation*}
\Phi_{n}:\CB(\lambda) \cap \CBn^{\lambda} 
\stackrel{\sim}{\rightarrow} \BBn(\lambda)
\end{equation*}
of $U_{q}(\Fgn)$-crystals. 
We show that the following diagram is 
commutative for all $n \in \BZ_{\ge 0}$ 
(see also \eqref{eq:limCBn} and \eqref{eq:limBBn}):
%
%%%%%%%%%%%%%%%%
%%% eq:CD-Bn %%%
%%%%%%%%%%%%%%%%
%
\begin{equation} \label{eq:CD-Bn}
\begin{CD}
\CB(\lambda) \cap \CBn^{\lambda} @>{\subset}>> 
\CB(\lambda) \cap \CBo^{\lambda} \\[3mm]
@V{\Phi_{n}}VV @VV{\Phi_{n+1}}V \\
\BBn(\lambda) @>{\subset}>> \BBo(\lambda);
\end{CD}
\end{equation}
this amounts to showing the equality 
$\Phi_{n+1}(b)=\Phi_{n}(b)$ for all 
$b \in \CB(\lambda) \cap \CBn^{\lambda}$. 
Now, observe that by the definition of 
$\CB(\lambda) \cap \CBn^{\lambda}$, 
the extremal element $u_{\lambda}=u_{\infty} \otimes 
t_{\lambda} \otimes u_{-\infty} \in \CB(\lambda)$ is 
contained in $\CB(\lambda) \cap \CBn^{\lambda}$ 
(and hence in $\CB(\lambda) \cap \CBo^{\lambda}$). 
Therefore, for each $b \in \CB(\lambda) \cap \CBn^{\lambda}$, 
there exists a monomial $X$ in the Kashiwara operators 
$e_{i}$ and $f_{i}$ for $i \in [n]$ such that $b=Xu_{\lambda}$, 
since the crystal graph of the $U_{q}(\Fgn)$-crystal 
$\CB(\lambda) \cap \CBn^{\lambda} \ (\cong \CBn(\lamn))$ 
is connected. 
Also, we see from Lemma~\ref{lem:BBn} that 
$\pi_{\lambda}$ is contained in $\BBn(\lambda)$ 
(and hence in $\BBo(\lambda)$), and 
from Remark~\ref{rem:LS-n} that 
$\BBn(\lambda)_{\lambda}=\BBo(\lambda)_{\lambda}=
\{\pi_{\lambda}\}$, which implies that 
$\Phi_{n}(u_{\lambda})=
\Phi_{n+1}(u_{\lambda})=\pi_{\lambda}$. 
Combining these, we have
\begin{equation*}
\Phi_{n}(b)=\Phi_{n}(Xu_{\lambda})=X\Phi_{n}(u_{\lambda})=
X\pi_{\lambda}=X\Phi_{n+1}(u_{\lambda})=
\Phi_{n+1}(Xu_{\lambda})=\Phi_{n+1}(b),
\end{equation*}
as desired. 

The commutative diagram \eqref{eq:CD-Bn} allows us to 
define a map $\Phi:\CB(\lambda) \rightarrow 
\BB(\lambda)$ as follows: for $b \in \CB(\lambda)$, 
\begin{equation*}
\Phi(b):=\Phi_{n}(b) \quad 
 \text{if $b \in \CB(\lambda) \cap \CBn^{\lambda}$ 
 for some $n \in \BZ_{\ge 0}$ (see \eqref{eq:limCBn})}. 
\end{equation*}
By using the definition, 
we can easily verify that the map 
$\Phi:\CB(\lambda) \rightarrow \BB(\lambda)$ is indeed 
an isomorphism of $U_{q}(\Fg)$-crystals. 
Thus we have proved the theorem. 
\end{proof}

The next corollary follows immediately from 
Theorem~\ref{thm:isom}, and 
Propositions~\ref{prop:cb-conn}, \ref{prop:wtmult}, 
\ref{prop:non-isom}. 
%
%%%%%%%%%%%%%%%%
%%% cor:isom %%%
%%%%%%%%%%%%%%%%
%
\begin{cor} \label{cor:isom}
{\rm (1)} For each $\lambda \in P$, 
the crystal $\BB(\lambda)$ consisting of 
all LS paths of shape $\lambda$ is a normal 
$U_{q}(\Fg)$-crystal whose crystal graph is connected. 

{\rm (2)} For each $\lambda \in P$ and $w \in W$, 
we have $\BB(\lambda)_{w\lambda}=
\bigl\{S_{w}\pi_{\lambda}\bigr\}=\bigl\{\pi_{w\lambda}\bigr\}$. 

{\rm (3)} Let $\lambda,\,\mu \in P$. Then, 
$\BB(\lambda) \cong \BB(\mu)$ as $U_{q}(\Fg)$-crystals 
if and only if $\lambda \in W\mu$.
\end{cor}

Note that the tensor product of normal $U_{q}(\Fg)$-crystals is 
also a normal $U_{q}(\Fg)$-crystal. Hence, 
by Corollary~\ref{cor:isom}\,(1), 
the tensor product $\BB(\lambda) \otimes \BB(\mu)$ 
for $\lambda,\,\mu \in P$ is a normal $U_{q}(\Fg)$-crystal. 
The following proposition plays a key role 
in the next section. 
%
%%%%%%%%%%%%%%%%%
%%% prop:isom %%%
%%%%%%%%%%%%%%%%%
%
\begin{prop} \label{prop:isom}
Let $\lambda,\,\mu \in P$, and suppose that 
$\pi \otimes \eta \in \BB(\lambda) \otimes \BB(\mu)$ is 
an extremal element of weight $\nu \in P$.
If we denote by $\BB(\pi \otimes \eta)$ 
the connected component of 
$\BB(\lambda) \otimes \BB(\mu)$ 
containing the extremal element $\pi \otimes \eta$, then 
$\BB(\pi \otimes \eta)$ is isomorphic, as a $U_{q}(\Fg)$-crystal, 
to $\BB(\nu)$ (and hence to $\CB(\nu)$). 
\end{prop}

\begin{proof}
For $n \in \BZ_{\ge 0}$, 
let $\BBn(\pi \otimes \eta)$ denote the subset of $\BB(\pi \otimes \eta)$ 
consisting of all elements of the form: $X(\pi \otimes \eta)$ for some 
monomial $X$ in the Kashiwara operators $e_{i}$ and $f_{i}$ 
for $i \in [n]$; the crystal graph of the $U_{q}(\Fgn)$-crystal 
$\BBn(\pi \otimes \eta)$ is clearly connected. 
We claim that there exists an isomorphism 
$\Psi_{n}:\BBn(\pi \otimes \eta) \stackrel{\sim}{\rightarrow} 
\BBn(\nu)$ of $U_{q}(\Fgn)$-crystals. 
Indeed, let $w \in \Wn$ be such that $w\nu=\nun$. 
We deduce from the definitions of 
$S_{w}$ and $\BBn(\pi \otimes \eta)$ that 
$S_{w}(\pi \otimes \eta)$ 
is contained in $\BBn(\pi \otimes \eta)$. 
In addition, since $\pi \otimes \eta$ is extremal by assumption, 
it is $[n]$-extremal. Consequently, the element 
$S_{w}(\pi \otimes \eta) \in \BBn(\pi \otimes \eta)$ 
is an $[n]$-extremal element of weight $w\nu=\nun$. 
Furthermore, since $\nun$ is $[n]$-dominant by definition, 
we see from Definition~\ref{dfn:extremal}\,(1) that 
$S_{w}(\pi \otimes \eta)$ is an $[n]$-maximal element 
of $\BBn(\pi \otimes \eta)$.
Also, because $\BB(\lambda) \otimes \BB(\mu)$ 
is a normal $U_{q}(\Fg)$-crystal, and because the crystal graph of 
the $U_{q}(\Fgn)$-crystal $\BBn(\pi \otimes \eta)$ is connected, 
we find that $\BBn(\pi \otimes \eta)$ is isomorphic, as a $U_{q}(\Fgn)$-crystal, 
to the crystal basis of a finite-dimensional irreducible 
$U_{q}(\Fgn)$-module. 
From these facts, along with Remark~\ref{rem:LS-n}, 
we conclude that $\BBn(\pi \otimes \eta) \cong 
\BBn(\nun)=\BBn(\nu)$ as $U_{q}(\Fgn)$-crystals.

Now, it is obvious from the definitions that 
%
%%%%%%%%%%%%%%%%
%%% eq:limBn %%%
%%%%%%%%%%%%%%%%
%
\begin{equation} \label{eq:limBn}
\BBn(\pi \otimes \eta) \subset \BBo(\pi \otimes \eta) \quad 
\text{for all $n \in \BZ_{\ge 0}$, and} \quad
\BB(\pi \otimes \eta)=
\bigcup_{n \ge 0} \BBn(\pi \otimes \eta).
\end{equation}
Furthermore, in exactly the same way as 
in the proof of Theorem~\ref{thm:isom}, 
we obtain the following commutative diagram: 
\begin{equation*}
\begin{CD}
\BBn(\pi \otimes \eta) @>{\subset}>> 
\BBo(\pi \otimes \eta) \\[3mm]
@V{\Psi_{n}}VV @VV{\Psi_{n+1}}V \\
\BBn(\nu) @>{\subset}>> \BBo(\nu).
\end{CD}
\end{equation*}
This commutative diagram allows us to define a map 
$\Psi:\BB(\pi \otimes \eta) \rightarrow \BB(\nu)$ as 
follows: for $b \in \BB(\pi \otimes \eta)$, 
\begin{equation*}
\Psi(b):=\Psi_{n}(b) \quad 
 \text{if $b \in \BBn(\pi \otimes \eta)$
 for some $n \in \BZ_{\ge 0}$}. 
\end{equation*}
By using the definition, 
we can easily verify that the map $\Psi$ 
is an isomorphism of $U_{q}(\Fg)$-crystals. 
Thus we have proved the proposition. 
\end{proof}

%=========================%
%     START SECTION 04    %
%=========================%
%
\section{Decomposition of tensor products into connected components.}
\label{sec:decomp}

In this section, we consider the decomposition 
(into connected components) of the tensor product 
$\BB(\lambda) \otimes \BB(\mu)$ for $\lambda,\,\mu \in P$
with $L_{\lambda},\,L_{\mu} \ge 0$; in fact, it turns out 
that each connected component is isomorphic to $\BB(\nu)$ 
for some $\nu \in P$. Our main aim is to give an explicit 
description of the multiplicity $m_{\lambda,\,\mu}^{\nu}$ 
of a connected component $\BB(\nu)$ for $\nu \in P$ 
in this decomposition. It should be mentioned that 
our results in this section can be regarded as extensions 
of the corresponding results in \cite{Kw-Adv} (in type $A_{+\infty}$)
and \cite{Kw-Ep} (in type $A_{\infty}$) to the cases of type 
$B_{\infty}$, $C_{\infty}$, $D_{\infty}$; see also \cite{Lec-TAMS}
for the case $\lambda,\,\mu \in E$ (in all types $A_{+\infty}$, 
$B_{\infty}$, $C_{\infty}$, $D_{\infty}$). 

%==============================%
%     START SUBSECTION 0401    %
%==============================%
%
\subsection{The case $\lambda,\,\mu \in E$.}
\label{subsec:EE}
A partition is, by definition, a weakly increasing sequence
$\rho=(\co{\rho}{0} \ge \co{\rho}{1} \ge \co{\rho}{2} \ge \cdots)$
of nonnegative integers such that 
$\co{\rho}{j}=0$ for all but finitely many $j \in \BZ_{\ge 0}$; 
we call $\co{\rho}{j}$ the $j$-th part of $\rho$ for $j \in \BZ_{\ge 0}$. 
Let $\Par$ denote the set of all partitions. 
For $\rho=(\co{\rho}{0} \ge \co{\rho}{1} \ge \co{\rho}{2} \ge \cdots) \in \Par$, 
we define the length of $\rho$, denoted by $\ell(\rho)$, 
to be the number of nonzero parts of $\rho$, and 
define the size of $\rho$, denoted by $|\rho|$, 
to be the sum of all parts of $\rho$, 
i.e., $|\rho|:=\sum_{j \in \BZ_{\ge 0}}\co{\rho}{j}$. 
Also, for $\rho,\,\kappa,\,\omega \in \Par$, 
let $\LR_{\rho,\,\kappa}^{\omega}$ denote 
the Littlewood-Richardson coefficient for $\rho$, $\kappa$, and $\omega$ 
(see, for example, \cite[Chapter~5]{Fulton} or \cite[\S2]{Ko98}); 
it is well-known that $\LR_{\rho,\,\kappa}^{\omega} \ne 0$ only if 
$|\omega|=|\rho|+|\kappa|$. 

Now, let $\EW$ denote the subset of $E$ consisting of all elements $\nu$ 
such that the sequence $(\con{0},\,\con{1},\,\con{2},\,\dots)$ is a partition; 
observe that if $\nu \in \EW$, then 
$\pair{\nu}{h_{i}} \le 0$ for all $i \in \BZ_{\ge 1}$ 
(see \eqref{eq:simple}), and hence $\nu$ is 
$[1,\,n]$-antidominant for all $n \in \BZ_{\ge 1}$.
We identify the set $\EW$ with the set $\Par$ of all partitions 
via the correspondence: $\nu \mapsto (\con{0},\,\con{1},\,\con{2},\,\dots)$; 
under this correspondence, we have 
$|\nu|=\sum_{j \in \BZ_{\ge 0}} |\con{j}|$, and 
$\ell(\nu)=\#\Supp (\nu)=\max \Supp(\nu)+1$ for $\nu \in \EW$. 
Furthermore, using \eqref{eq:binf-r}--\eqref{eq:dinf-r}, 
we can easily verify that for each $\lambda \in E$, 
there exists a unique element $\lambda_{\dagger} \in W\lambda$ 
such that $\lambda_{\dagger} \in \EW \ (\cong \Par)$. 
In fact, the corresponding partition 
$(\co{\lambda_{\dagger}}{0},\,\co{\lambda_{\dagger}}{1},\,
  \co{\lambda_{\dagger}}{2},\,\dots) \in \Par$ 
is obtained by arranging the sequence 
$(|\col{0}|,\,|\col{1}|,\,|\col{2}|,\,\dots)$ of 
nonnegative integers in weakly increasing order. 
Thus, $\EW \ (\cong \Par)$ is a complete set of 
representatives for $W$-orbits in $E$. 
It follows from Corollary~\ref{cor:isom}\,(3) that
if $\nu \ne \nu'$ for $\nu,\,\nu' \in \EW$, 
then $\BB(\nu) \not\cong \BB(\nu')$ 
as $U_{q}(\Fg)$-crystals. 

The following is the main result of this subsection 
(cf. \cite[Corollary~3.2.5 and Remark following Theorem~4.1.6]{Lec-TAMS}, 
and also \cite[Theorem~4.10]{Kw-Adv} in type $A_{+\infty}$, 
\cite[Proposition~4.9]{Kw-Ep} in type $A_{\infty}$). 
%
%%%%%%%%%%%%%%
%%% thm:EE %%%
%%%%%%%%%%%%%%
%
\begin{thm} \label{thm:EE}
Let $\lambda,\,\mu \in E$. Then, we have the following 
decomposition into connected components\,{\rm:}
%
%%%%%%%%%%%%%
%%% eq:EE %%%
%%%%%%%%%%%%%
%
\begin{equation} \label{eq:EE}
\BB(\lambda) \otimes \BB(\mu) = 
 \bigoplus_{\nu \in \EW} 
 \BB(\nu)^{\oplus m_{\lambda,\,\mu}^{\nu}}
\quad \text{\rm as $U_{q}(\Fg)$-crystals},
\end{equation}
where for each $\nu \in \EW \ (\cong \Par)$, 
the multiplicity $m_{\lambda,\,\mu}^{\nu}$ is equal to 
the Littlewood-Richardson coefficient 
$\LR_{\lambda_{\dagger},\,\mu_{\dagger}}^{\nu}$ 
for the partitions $\lambda_{\dagger}$, 
$\mu_{\dagger}$, and $\nu$.
\end{thm}
%
%%%%%%%%%%%%%%
%%% rem:EE %%%
%%%%%%%%%%%%%%
%
\begin{rem} \label{rem:EE}
Because 
$\LR_{\lambda_{\dagger},\,\mu_{\dagger}}^{\nu} \ne 0$ 
only if $|\nu|=|\lambda_{\dagger}|+|\mu_{\dagger}|$ as noted above, 
it follows that the total number of connected components of 
$\BB(\lambda) \otimes \BB(\mu)$ is finite 
for $\lambda,\,\mu \in E$. 
\end{rem}

Before proving Theorem~\ref{thm:EE}, 
we need the following proposition. 

%%%%%%%%%%%%%%%%%%
%% prop:EE-ext %%%
%%%%%%%%%%%%%%%%%%

\begin{prop} \label{prop:EE-ext}
Keep the setting of Theorem~\ref{thm:EE}. 
Each connected component of 
$\BB(\lambda) \otimes \BB(\mu)$ 
contains an extremal element.
\end{prop}

\begin{proof}
Recall from Remark~\ref{rem:wt} that 
if $\pi \in \BB(\lambda)$, then $\wt \pi=\pi(1) \in E$. 
First we claim that 
\begin{equation*}
(\wt \pi,\,\wt \pi) \le (\lambda,\,\lambda)
\quad \text{for all $\pi \in \BB(\lambda)$}.
\end{equation*}
Indeed, define $N=N_{\lambda}$ as in Remark~\ref{rem:lcm}, and 
write $\pi \in \BB(\lambda)$ as: 
$\pi=(\nu_{1},\,\nu_{2},\,\dots,\,\nu_{N})$ for some 
$\nu_{1},\,\nu_{2},\,\dots,\,\nu_{N} \in W\lambda \subset E$;
note that $\wt \pi=(1/N)\sum_{M=1}^{N}\nu_{M}$ by \eqref{eq:wtN}.
Since $(\cdot\,,\,\cdot)$ is positive definite on 
$\bigoplus_{j \in \BZ_{\ge 0}} \BR\eps_{j}$, 
it follows from the Cauchy-Schwarz inequality that 
\begin{equation*}
(\wt \pi,\,\wt \pi) = \frac{1}{N^{2}}
\left(
 \sum_{M=1}^{N} \nu_{M},\ 
 \sum_{M=1}^{N} \nu_{M}
\right)
\le \frac{1}{N^{2}}
\left(
 \sum_{M=1}^{N} (\nu_{M},\,\nu_{M})^{\frac{1}{2}}
\right)^{2}.
\end{equation*}
In addition, 
since $(\cdot\,,\,\cdot)$ is $W$-invariant, 
we have 
\begin{equation*}
\frac{1}{N^{2}}
\left(
 \sum_{M=1}^{N} (\nu_{M},\,\nu_{M})^{\frac{1}{2}}
\right)^{2}=
\frac{1}{N^{2}}
\left(
 \sum_{M=1}^{N} (\lambda,\,\lambda)^{\frac{1}{2}}
\right)^{2}=(\lambda,\,\lambda). 
\end{equation*}
Combining these, we obtain 
$(\wt \pi,\,\wt \pi) \le (\lambda,\,\lambda)$, as desired. 
Similarly, we obtain 
\begin{equation*}
(\wt \eta,\,\wt \eta) \le (\mu,\,\mu)
\quad \text{for all $\eta \in \BB(\mu)$}.
\end{equation*}
From these inequalities, we see, 
again by the Cauchy-Schwarz inequality, 
that for every $\pi \otimes \eta \in 
\BB(\lambda) \otimes \BB(\mu)$, 
\begin{align} 
\bigl(\wt (\pi \otimes \eta),\,\wt (\pi \otimes \eta)\bigr)
& = \bigl(\wt \pi + \wt \eta,\,\wt \pi + \wt \eta\bigr) \nonumber \\
& \le \bigl\{(\wt \pi,\,\wt \pi)^{\frac{1}{2}}+
             (\wt \eta,\,\wt \eta)^{\frac{1}{2}}
      \bigr\}^{2} \nonumber \\
& \le \bigl\{
      (\lambda,\,\lambda)^{\frac{1}{2}} + 
      (\mu,\,\mu)^{\frac{1}{2}}
      \bigr\}^{2}. \label{eq:EE-ext}
\end{align}

Now, let $\BB$ be an arbitrary connected component of 
$\BB(\lambda) \otimes \BB(\mu)$. 
We deduce from \eqref{eq:EE-ext} that the subset 
$\bigl\{\bigl(\wt b,\,\wt b\bigr) 
\mid b \in \BB\bigr\}$ of $\BZ_{\ge 0}$ is bounded above, 
and hence that there exists $b_{0} \in \BB$ for which 
the equality
\begin{equation*}
(\wt b_{0},\,\wt b_{0})=
\max \bigl\{\bigl(\wt b,\,\wt b\bigr) 
\mid b \in \BB\bigr\}
\end{equation*}
holds. From this equality, by arguing as in \cite[\S9.3]{Kas94}, 
we find that the element $b_{0} \in \BB$ is extremal.
Thus we have proved the proposition.
\end{proof}

Combining Propositions~\ref{prop:isom} and \ref{prop:EE-ext}, 
we conclude that each connected component 
of $\BB(\lambda) \otimes \BB(\mu)$ is isomorphic, 
as a $U_{q}(\Fg)$-crystal, to $\BB(\xi)$ for some $\xi \in E$, 
and hence to $\BB(\nu)$ for some $\nu \in \EW$ 
by Remark~\ref{rem:LS}\,(2); recall that 
$\EW$ is a complete set of representatives for $W$-orbits in $E$.
We will prove that for each $\nu \in \EW$, 
the multiplicity $m_{\lambda,\,\mu}^{\nu}$ 
in the decomposition \eqref{eq:EE}
is equal to the Littlewood-Richardson coefficient 
$\LR_{\lambda_{\dagger},\,\mu_{\dagger}}^{\nu}$.
Because $\BB(\lambda) \otimes \BB(\mu) = 
\BB(\lambda_{\dagger}) \otimes \BB(\mu_{\dagger})$ 
by Remark~\ref{rem:LS}\,(2), we may and do assume that 
$\lambda=\lambda_{\dagger} \in \EW$ and 
$\mu=\mu_{\dagger} \in \EW$ for the rest of this subsection. 
%
%%%%%%%%%%%%%%%%%%%
%%% prop:EE-key %%%
%%%%%%%%%%%%%%%%%%%
%
\begin{prop} \label{prop:EE-key}
Keep the setting above. 

{\rm (1)} Let $\nu \in \EW$. 
If $\BB(\nu)$ is isomorphic, as a $U_{q}(\Fg)$-crystal, to
a connected component of $\BB(\lambda) \otimes \BB(\mu)$, 
then we have $|\nu|=|\lambda|+|\mu|$. 

{\rm (2)} Let $m \in \BZ_{\ge 1}$ be such that $m \ge |\lambda|+|\mu|$, 
and let $\nu \in \EW$. 
If $\pi \otimes \eta \in \BB(\lambda) \otimes \BB(\mu)$ 
is an extremal element of weight $\nu$, then 
$\pi \otimes \eta$ is a $[1,\,m]$-minimal element contained in 
$\BBm{1,\,m}(\lambda) \otimes \BBm{1,\,m}(\mu)$.

{\rm (3)} Let $m \in \BZ_{\ge 1}$ be such that $m \ge |\lambda|+|\mu|$. 
Every $[1,\,m]$-minimal element of $\BBm{1,\,m}(\lambda) \otimes 
\BBm{1,\,m}(\mu)$ is extremal. 
\end{prop}

\begin{proof}
(1) We prove the assertion by induction on 
$|\mu|=\sum_{j \in \BZ_{\ge 0}}|\com{j}|$. 
If $|\mu|=0$ (and hence $\mu=0$), then the assertion is obvious. 
Assume now that $|\mu|=1$, and hence $\mu=\eps_{0}$. 
By direct computation, we can check that if $\Fg$ is of type $B_{\infty}$
(resp., $C_{\infty}$, $D_{\infty}$), then 
the crystal graph of $\BB(\eps_{0})$ is given by 
\eqref{eq:binf-cg} (resp., \eqref{eq:cinf-cg}, 
\eqref{eq:dinf-cg}; cf. \cite[\S3.2]{Lec-TAMS}): 

\begin{equation} \label{eq:binf-cg}
\begin{array}{l}
%WinTpicVersion3.08
\unitlength 0.1in
\begin{picture}( 48.7500,  3.0000)(  5.2500, -9.1500)
% CIRCLE 2 0 3 0
% 4 1200 800 1200 750 1200 750 1200 750
% 
\special{pn 8}%
\special{ar 1200 800 50 50  0.0000000 6.2831853}%
% CIRCLE 2 0 3 0
% 4 1800 800 1800 750 1800 750 1800 750
% 
\special{pn 8}%
\special{ar 1800 800 50 50  0.0000000 6.2831853}%
% CIRCLE 2 0 3 0
% 4 2400 800 2400 750 2400 750 2400 750
% 
\special{pn 8}%
\special{ar 2400 800 50 50  0.0000000 6.2831853}%
% CIRCLE 2 0 3 0
% 4 3000 800 3000 750 3000 750 3000 750
% 
\special{pn 8}%
\special{ar 3000 800 50 50  0.0000000 6.2831853}%
% CIRCLE 2 0 3 0
% 4 3600 800 3600 750 3600 750 3600 750
% 
\special{pn 8}%
\special{ar 3600 800 50 50  0.0000000 6.2831853}%
% CIRCLE 2 0 3 0
% 4 4200 800 4200 750 4200 750 4200 750
% 
\special{pn 8}%
\special{ar 4200 800 50 50  0.0000000 6.2831853}%
% CIRCLE 2 0 3 0
% 4 4800 800 4800 750 4800 750 4800 750
% 
\special{pn 8}%
\special{ar 4800 800 50 50  0.0000000 6.2831853}%
% LINE 2 0 3 0
% 2 1200 800 900 800
% 
\special{pn 8}%
\special{pa 1200 800}%
\special{pa 900 800}%
\special{fp}%
% LINE 2 2 3 0
% 2 900 800 600 800
% 
\special{pn 8}%
\special{pa 900 800}%
\special{pa 600 800}%
\special{dt 0.045}%
% LINE 2 2 3 0
% 2 5400 800 5100 800
% 
\special{pn 8}%
\special{pa 5400 800}%
\special{pa 5100 800}%
\special{dt 0.045}%
% STR 2 0 3 0
% 3 3000 900 3000 1000 5 0
% $\pi_0$
\put(30.0000,-10.0000){\makebox(0,0){$\pi_0$}}%
% STR 2 0 3 0
% 3 2400 900 2400 1000 5 0
% $\pi_{-\eps_0}$
\put(24.0000,-10.0000){\makebox(0,0){$\pi_{-\eps_0}$}}%
% STR 2 0 3 0
% 3 1800 900 1800 1000 5 0
% $\pi_{-\eps_1}$
\put(18.0000,-10.0000){\makebox(0,0){$\pi_{-\eps_1}$}}%
% STR 2 0 3 0
% 3 1200 900 1200 1000 5 0
% $\pi_{-\eps_2}$
\put(12.0000,-10.0000){\makebox(0,0){$\pi_{-\eps_2}$}}%
% STR 2 0 3 0
% 3 3600 900 3600 1000 5 0
% $\pi_{\eps_0}$
\put(36.0000,-10.0000){\makebox(0,0){$\pi_{\eps_0}$}}%
% STR 2 0 3 0
% 3 4200 900 4200 1000 5 0
% $\pi_{\eps_1}$
\put(42.0000,-10.0000){\makebox(0,0){$\pi_{\eps_1}$}}%
% STR 2 0 3 0
% 3 4800 900 4800 1000 5 0
% $\pi_{\eps_2}$
\put(48.0000,-10.0000){\makebox(0,0){$\pi_{\eps_2}$}}%
% VECTOR 2 0 3 0
% 2 5100 800 4800 800
% 
\special{pn 8}%
\special{pa 5100 800}%
\special{pa 4800 800}%
\special{fp}%
\special{sh 1}%
\special{pa 4800 800}%
\special{pa 4868 820}%
\special{pa 4854 800}%
\special{pa 4868 780}%
\special{pa 4800 800}%
\special{fp}%
% VECTOR 2 0 3 0
% 2 4800 800 4200 800
% 
\special{pn 8}%
\special{pa 4800 800}%
\special{pa 4200 800}%
\special{fp}%
\special{sh 1}%
\special{pa 4200 800}%
\special{pa 4268 820}%
\special{pa 4254 800}%
\special{pa 4268 780}%
\special{pa 4200 800}%
\special{fp}%
% VECTOR 2 0 3 0
% 2 4200 800 3600 800
% 
\special{pn 8}%
\special{pa 4200 800}%
\special{pa 3600 800}%
\special{fp}%
\special{sh 1}%
\special{pa 3600 800}%
\special{pa 3668 820}%
\special{pa 3654 800}%
\special{pa 3668 780}%
\special{pa 3600 800}%
\special{fp}%
% VECTOR 2 0 3 0
% 2 3600 800 3000 800
% 
\special{pn 8}%
\special{pa 3600 800}%
\special{pa 3000 800}%
\special{fp}%
\special{sh 1}%
\special{pa 3000 800}%
\special{pa 3068 820}%
\special{pa 3054 800}%
\special{pa 3068 780}%
\special{pa 3000 800}%
\special{fp}%
% VECTOR 2 0 3 0
% 2 3000 800 2400 800
% 
\special{pn 8}%
\special{pa 3000 800}%
\special{pa 2400 800}%
\special{fp}%
\special{sh 1}%
\special{pa 2400 800}%
\special{pa 2468 820}%
\special{pa 2454 800}%
\special{pa 2468 780}%
\special{pa 2400 800}%
\special{fp}%
% VECTOR 2 0 3 0
% 2 2400 800 1800 800
% 
\special{pn 8}%
\special{pa 2400 800}%
\special{pa 1800 800}%
\special{fp}%
\special{sh 1}%
\special{pa 1800 800}%
\special{pa 1868 820}%
\special{pa 1854 800}%
\special{pa 1868 780}%
\special{pa 1800 800}%
\special{fp}%
% VECTOR 2 0 3 0
% 2 1800 800 1200 800
% 
\special{pn 8}%
\special{pa 1800 800}%
\special{pa 1200 800}%
\special{fp}%
\special{sh 1}%
\special{pa 1200 800}%
\special{pa 1268 820}%
\special{pa 1254 800}%
\special{pa 1268 780}%
\special{pa 1200 800}%
\special{fp}%
% STR 2 0 3 0
% 3 5100 600 5100 700 5 0
% $3$
\put(51.0000,-7.0000){\makebox(0,0){$3$}}%
% STR 2 0 3 0
% 3 4500 600 4500 700 5 0
% $2$
\put(45.0000,-7.0000){\makebox(0,0){$2$}}%
% STR 2 0 3 0
% 3 3900 600 3900 700 5 0
% $1$
\put(39.0000,-7.0000){\makebox(0,0){$1$}}%
% STR 2 0 3 0
% 3 3300 600 3300 700 5 0
% $0$
\put(33.0000,-7.0000){\makebox(0,0){$0$}}%
% STR 2 0 3 0
% 3 2700 600 2700 700 5 0
% $0$
\put(27.0000,-7.0000){\makebox(0,0){$0$}}%
% STR 2 0 3 0
% 3 2100 600 2100 700 5 0
% $1$
\put(21.0000,-7.0000){\makebox(0,0){$1$}}%
% STR 2 0 3 0
% 3 1500 600 1500 700 5 0
% $2$
\put(15.0000,-7.0000){\makebox(0,0){$2$}}%
% STR 2 0 3 0
% 3 900 600 900 700 5 0
% $3$
\put(9.0000,-7.0000){\makebox(0,0){$3$}}%
\end{picture}%
 \\[3mm]
\end{array}
\end{equation}
Here, $\pi_{0}=(-\eps_{0},\,\eps_{0}\,;\,0,\,1/2,\,1)$. 
\begin{equation} \label{eq:cinf-cg}
\begin{array}{l}
%WinTpicVersion3.08
\unitlength 0.1in
\begin{picture}( 42.7500,  3.0000)(  5.2500, -9.1500)
% CIRCLE 2 0 3 0
% 4 1200 800 1200 750 1200 750 1200 750
% 
\special{pn 8}%
\special{ar 1200 800 50 50  0.0000000 6.2831853}%
% CIRCLE 2 0 3 0
% 4 1800 800 1800 750 1800 750 1800 750
% 
\special{pn 8}%
\special{ar 1800 800 50 50  0.0000000 6.2831853}%
% CIRCLE 2 0 3 0
% 4 2400 800 2400 750 2400 750 2400 750
% 
\special{pn 8}%
\special{ar 2400 800 50 50  0.0000000 6.2831853}%
% CIRCLE 2 0 3 0
% 4 3000 800 3000 750 3000 750 3000 750
% 
\special{pn 8}%
\special{ar 3000 800 50 50  0.0000000 6.2831853}%
% CIRCLE 2 0 3 0
% 4 3600 800 3600 750 3600 750 3600 750
% 
\special{pn 8}%
\special{ar 3600 800 50 50  0.0000000 6.2831853}%
% CIRCLE 2 0 3 0
% 4 4200 800 4200 750 4200 750 4200 750
% 
\special{pn 8}%
\special{ar 4200 800 50 50  0.0000000 6.2831853}%
% LINE 2 0 3 0
% 2 1200 800 900 800
% 
\special{pn 8}%
\special{pa 1200 800}%
\special{pa 900 800}%
\special{fp}%
% LINE 2 2 3 0
% 2 900 800 600 800
% 
\special{pn 8}%
\special{pa 900 800}%
\special{pa 600 800}%
\special{dt 0.045}%
% LINE 2 2 3 0
% 2 4800 800 4500 800
% 
\special{pn 8}%
\special{pa 4800 800}%
\special{pa 4500 800}%
\special{dt 0.045}%
% STR 2 0 3 0
% 3 2400 900 2400 1000 5 0
% $\pi_{-\eps_0}$
\put(24.0000,-10.0000){\makebox(0,0){$\pi_{-\eps_0}$}}%
% STR 2 0 3 0
% 3 1800 900 1800 1000 5 0
% $\pi_{-\eps_1}$
\put(18.0000,-10.0000){\makebox(0,0){$\pi_{-\eps_1}$}}%
% STR 2 0 3 0
% 3 1200 900 1200 1000 5 0
% $\pi_{-\eps_2}$
\put(12.0000,-10.0000){\makebox(0,0){$\pi_{-\eps_2}$}}%
% STR 2 0 3 0
% 3 3000 900 3000 1000 5 0
% $\pi_{\eps_0}$
\put(30.0000,-10.0000){\makebox(0,0){$\pi_{\eps_0}$}}%
% STR 2 0 3 0
% 3 3600 900 3600 1000 5 0
% $\pi_{\eps_1}$
\put(36.0000,-10.0000){\makebox(0,0){$\pi_{\eps_1}$}}%
% STR 2 0 3 0
% 3 4200 900 4200 1000 5 0
% $\pi_{\eps_2}$
\put(42.0000,-10.0000){\makebox(0,0){$\pi_{\eps_2}$}}%
% VECTOR 2 0 3 0
% 2 4500 800 4200 800
% 
\special{pn 8}%
\special{pa 4500 800}%
\special{pa 4200 800}%
\special{fp}%
\special{sh 1}%
\special{pa 4200 800}%
\special{pa 4268 820}%
\special{pa 4254 800}%
\special{pa 4268 780}%
\special{pa 4200 800}%
\special{fp}%
% VECTOR 2 0 3 0
% 2 4200 800 3600 800
% 
\special{pn 8}%
\special{pa 4200 800}%
\special{pa 3600 800}%
\special{fp}%
\special{sh 1}%
\special{pa 3600 800}%
\special{pa 3668 820}%
\special{pa 3654 800}%
\special{pa 3668 780}%
\special{pa 3600 800}%
\special{fp}%
% VECTOR 2 0 3 0
% 2 3600 800 3000 800
% 
\special{pn 8}%
\special{pa 3600 800}%
\special{pa 3000 800}%
\special{fp}%
\special{sh 1}%
\special{pa 3000 800}%
\special{pa 3068 820}%
\special{pa 3054 800}%
\special{pa 3068 780}%
\special{pa 3000 800}%
\special{fp}%
% VECTOR 2 0 3 0
% 2 3000 800 2400 800
% 
\special{pn 8}%
\special{pa 3000 800}%
\special{pa 2400 800}%
\special{fp}%
\special{sh 1}%
\special{pa 2400 800}%
\special{pa 2468 820}%
\special{pa 2454 800}%
\special{pa 2468 780}%
\special{pa 2400 800}%
\special{fp}%
% VECTOR 2 0 3 0
% 2 2400 800 1800 800
% 
\special{pn 8}%
\special{pa 2400 800}%
\special{pa 1800 800}%
\special{fp}%
\special{sh 1}%
\special{pa 1800 800}%
\special{pa 1868 820}%
\special{pa 1854 800}%
\special{pa 1868 780}%
\special{pa 1800 800}%
\special{fp}%
% VECTOR 2 0 3 0
% 2 1800 800 1200 800
% 
\special{pn 8}%
\special{pa 1800 800}%
\special{pa 1200 800}%
\special{fp}%
\special{sh 1}%
\special{pa 1200 800}%
\special{pa 1268 820}%
\special{pa 1254 800}%
\special{pa 1268 780}%
\special{pa 1200 800}%
\special{fp}%
% STR 2 0 3 0
% 3 4500 600 4500 700 5 0
% $3$
\put(45.0000,-7.0000){\makebox(0,0){$3$}}%
% STR 2 0 3 0
% 3 3900 600 3900 700 5 0
% $2$
\put(39.0000,-7.0000){\makebox(0,0){$2$}}%
% STR 2 0 3 0
% 3 3300 600 3300 700 5 0
% $1$
\put(33.0000,-7.0000){\makebox(0,0){$1$}}%
% STR 2 0 3 0
% 3 2700 600 2700 700 5 0
% $0$
\put(27.0000,-7.0000){\makebox(0,0){$0$}}%
% STR 2 0 3 0
% 3 2100 600 2100 700 5 0
% $1$
\put(21.0000,-7.0000){\makebox(0,0){$1$}}%
% STR 2 0 3 0
% 3 1500 600 1500 700 5 0
% $2$
\put(15.0000,-7.0000){\makebox(0,0){$2$}}%
% STR 2 0 3 0
% 3 900 600 900 700 5 0
% $3$
\put(9.0000,-7.0000){\makebox(0,0){$3$}}%
\end{picture}%
 \\[3mm]
\end{array}
\end{equation}
\begin{equation} \label{eq:dinf-cg}
\begin{array}{l}
%WinTpicVersion3.08
\unitlength 0.1in
\begin{picture}( 42.7500, 12.0000)(  5.2500,-13.1500)
% CIRCLE 2 0 3 0
% 4 1200 800 1200 750 1200 750 1200 750
% 
\special{pn 8}%
\special{ar 1200 800 50 50  0.0000000 6.2831853}%
% CIRCLE 2 0 3 0
% 4 1800 800 1800 750 1800 750 1800 750
% 
\special{pn 8}%
\special{ar 1800 800 50 50  0.0000000 6.2831853}%
% CIRCLE 2 0 3 0
% 4 2400 1200 2400 1150 2400 1150 2400 1150
% 
\special{pn 8}%
\special{ar 2400 1200 50 50  0.0000000 6.2831853}%
% CIRCLE 2 0 3 0
% 4 3600 800 3600 750 3600 750 3600 750
% 
\special{pn 8}%
\special{ar 3600 800 50 50  0.0000000 6.2831853}%
% CIRCLE 2 0 3 0
% 4 4200 800 4200 750 4200 750 4200 750
% 
\special{pn 8}%
\special{ar 4200 800 50 50  0.0000000 6.2831853}%
% LINE 2 0 3 0
% 2 1200 800 900 800
% 
\special{pn 8}%
\special{pa 1200 800}%
\special{pa 900 800}%
\special{fp}%
% LINE 2 2 3 0
% 2 900 800 600 800
% 
\special{pn 8}%
\special{pa 900 800}%
\special{pa 600 800}%
\special{dt 0.045}%
% LINE 2 2 3 0
% 2 4800 800 4500 800
% 
\special{pn 8}%
\special{pa 4800 800}%
\special{pa 4500 800}%
\special{dt 0.045}%
% STR 2 0 3 0
% 3 1800 900 1800 1000 5 0
% $\pi_{-\eps_2}$
\put(18.0000,-10.0000){\makebox(0,0){$\pi_{-\eps_2}$}}%
% STR 2 0 3 0
% 3 1200 900 1200 1000 5 0
% $\pi_{-\eps_3}$
\put(12.0000,-10.0000){\makebox(0,0){$\pi_{-\eps_3}$}}%
% STR 2 0 3 0
% 3 3600 900 3600 1000 5 0
% $\pi_{\eps_2}$
\put(36.0000,-10.0000){\makebox(0,0){$\pi_{\eps_2}$}}%
% STR 2 0 3 0
% 3 4200 900 4200 1000 5 0
% $\pi_{\eps_3}$
\put(42.0000,-10.0000){\makebox(0,0){$\pi_{\eps_3}$}}%
% VECTOR 2 0 3 0
% 2 4500 800 4200 800
% 
\special{pn 8}%
\special{pa 4500 800}%
\special{pa 4200 800}%
\special{fp}%
\special{sh 1}%
\special{pa 4200 800}%
\special{pa 4268 820}%
\special{pa 4254 800}%
\special{pa 4268 780}%
\special{pa 4200 800}%
\special{fp}%
% VECTOR 2 0 3 0
% 2 4200 800 3600 800
% 
\special{pn 8}%
\special{pa 4200 800}%
\special{pa 3600 800}%
\special{fp}%
\special{sh 1}%
\special{pa 3600 800}%
\special{pa 3668 820}%
\special{pa 3654 800}%
\special{pa 3668 780}%
\special{pa 3600 800}%
\special{fp}%
% VECTOR 2 0 3 0
% 2 1800 800 1200 800
% 
\special{pn 8}%
\special{pa 1800 800}%
\special{pa 1200 800}%
\special{fp}%
\special{sh 1}%
\special{pa 1200 800}%
\special{pa 1268 820}%
\special{pa 1254 800}%
\special{pa 1268 780}%
\special{pa 1200 800}%
\special{fp}%
% STR 2 0 3 0
% 3 4500 600 4500 700 5 0
% $4$
\put(45.0000,-7.0000){\makebox(0,0){$4$}}%
% STR 2 0 3 0
% 3 3900 600 3900 700 5 0
% $3$
\put(39.0000,-7.0000){\makebox(0,0){$3$}}%
% STR 2 0 3 0
% 3 1500 600 1500 700 5 0
% $3$
\put(15.0000,-7.0000){\makebox(0,0){$3$}}%
% STR 2 0 3 0
% 3 900 600 900 700 5 0
% $4$
\put(9.0000,-7.0000){\makebox(0,0){$4$}}%
% CIRCLE 2 0 3 0
% 4 2400 400 2400 350 2400 350 2400 350
% 
\special{pn 8}%
\special{ar 2400 400 50 50  0.0000000 6.2831853}%
% CIRCLE 2 0 3 0
% 4 3000 1200 3000 1150 3000 1150 3000 1150
% 
\special{pn 8}%
\special{ar 3000 1200 50 50  0.0000000 6.2831853}%
% CIRCLE 2 0 3 0
% 4 3000 400 3000 350 3000 350 3000 350
% 
\special{pn 8}%
\special{ar 3000 400 50 50  0.0000000 6.2831853}%
% STR 2 0 3 0
% 3 2400 100 2400 200 5 0
% $\pi_{-\eps_0}$
\put(24.0000,-2.0000){\makebox(0,0){$\pi_{-\eps_0}$}}%
% STR 2 0 3 0
% 3 2400 1300 2400 1400 5 0
% $\pi_{-\eps_1}$
\put(24.0000,-14.0000){\makebox(0,0){$\pi_{-\eps_1}$}}%
% STR 2 0 3 0
% 3 3000 100 3000 200 5 0
% $\pi_{\eps_1}$
\put(30.0000,-2.0000){\makebox(0,0){$\pi_{\eps_1}$}}%
% STR 2 0 3 0
% 3 3000 1300 3000 1400 5 0
% $\pi_{\eps_0}$
\put(30.0000,-14.0000){\makebox(0,0){$\pi_{\eps_0}$}}%
% VECTOR 2 0 3 0
% 2 3600 800 3000 400
% 
\special{pn 8}%
\special{pa 3600 800}%
\special{pa 3000 400}%
\special{fp}%
\special{sh 1}%
\special{pa 3000 400}%
\special{pa 3044 454}%
\special{pa 3044 430}%
\special{pa 3068 420}%
\special{pa 3000 400}%
\special{fp}%
% VECTOR 2 0 3 0
% 2 3000 400 3000 1200
% 
\special{pn 8}%
\special{pa 3000 400}%
\special{pa 3000 1200}%
\special{fp}%
\special{sh 1}%
\special{pa 3000 1200}%
\special{pa 3020 1134}%
\special{pa 3000 1148}%
\special{pa 2980 1134}%
\special{pa 3000 1200}%
\special{fp}%
% VECTOR 2 0 3 0
% 2 3000 1200 2400 1200
% 
\special{pn 8}%
\special{pa 3000 1200}%
\special{pa 2400 1200}%
\special{fp}%
\special{sh 1}%
\special{pa 2400 1200}%
\special{pa 2468 1220}%
\special{pa 2454 1200}%
\special{pa 2468 1180}%
\special{pa 2400 1200}%
\special{fp}%
% VECTOR 2 0 3 0
% 2 3000 400 2400 400
% 
\special{pn 8}%
\special{pa 3000 400}%
\special{pa 2400 400}%
\special{fp}%
\special{sh 1}%
\special{pa 2400 400}%
\special{pa 2468 420}%
\special{pa 2454 400}%
\special{pa 2468 380}%
\special{pa 2400 400}%
\special{fp}%
% VECTOR 2 0 3 0
% 2 2400 1200 1800 800
% 
\special{pn 8}%
\special{pa 2400 1200}%
\special{pa 1800 800}%
\special{fp}%
\special{sh 1}%
\special{pa 1800 800}%
\special{pa 1844 854}%
\special{pa 1844 830}%
\special{pa 1868 820}%
\special{pa 1800 800}%
\special{fp}%
% STR 2 0 3 0
% 3 2150 800 2150 900 5 0
% $2$
\put(21.5000,-9.0000){\makebox(0,0){$2$}}%
% STR 2 0 3 0
% 3 3350 400 3350 500 5 0
% $2$
\put(33.5000,-5.0000){\makebox(0,0){$2$}}%
% STR 2 0 3 0
% 3 2700 200 2700 300 5 0
% $0$
\put(27.0000,-3.0000){\makebox(0,0){$0$}}%
% STR 2 0 3 0
% 3 2700 1000 2700 1100 5 0
% $0$
\put(27.0000,-11.0000){\makebox(0,0){$0$}}%
% STR 2 0 3 0
% 3 3100 700 3100 800 5 0
% $1$
\put(31.0000,-8.0000){\makebox(0,0){$1$}}%
% VECTOR 2 0 3 0
% 2 2400 400 2400 1200
% 
\special{pn 8}%
\special{pa 2400 400}%
\special{pa 2400 1200}%
\special{fp}%
\special{sh 1}%
\special{pa 2400 1200}%
\special{pa 2420 1134}%
\special{pa 2400 1148}%
\special{pa 2380 1134}%
\special{pa 2400 1200}%
\special{fp}%
% STR 2 0 3 0
% 3 2500 700 2500 800 5 0
% $1$
\put(25.0000,-8.0000){\makebox(0,0){$1$}}%
\end{picture}%
 \\[5mm]
\end{array}
\end{equation}
Thus, we have 
\begin{align*}
\BB(\eps_{0}) & =
 \bigl\{\pi_{\pm \eps_{j}} \mid j \in \BZ_{\ge 0}\bigr\} \cup 
 \bigl\{\pi_{0}=(-\eps_{0},\,\eps_{0}\,;\,0,\,1/2,\,1)\bigr\} \qquad
 \text{if $\Fg$ is of type $B_{\infty}$}, \\
\BB(\eps_{0}) & =
 \bigl\{\pi_{\pm \eps_{j}} \mid j \in \BZ_{\ge 0}\bigr\} \qquad 
 \text{if $\Fg$ is of type $C_{\infty}$ or $D_{\infty}$}.
\end{align*}
By our assumption on $\nu\in \EW$, there exists an extremal element 
$\pi \otimes \eta \in \BB(\lambda) \otimes \BB(\eps_{0})$ of 
weight $\nu$. Take $n \in \BZ_{\ge 0}$ such that 
$n > \#\Supp(\lambda)$ 
($=\ell(\lambda)$ since $\lambda \in \EW$) 
and such that $\pi \in \BBn(\lambda)$ (see \eqref{eq:limBBn}). 
We deduce from Remark~\ref{rem:LS-n} that 
there exists a monomial $X$ in $e_{i}^{\max}$, $i \in [n]$, 
such that $X\pi=\pi_{\lamn}$. 
It follows from the tensor product rule for 
crystals that $X(\pi \otimes \eta)=(X\pi) \otimes \eta'=
\pi_{\lamn} \otimes \eta'$ for some $\eta' \in \BB(\eps_{0})$. 
Here we remark that the element $\pi_{\lamn} \otimes \eta'$ is 
extremal, since so is $\pi \otimes \eta$, 
and $X$ is a monomial in $e_{i}^{\max}$, 
$i \in [n]$; also, if we set $\xi:=\wt(\pi_{\lamn} \otimes \eta')$, 
then $\xi$ is contained in $\Wn\nu \subset W\nu$. 

Suppose, by contradiction, that $\Fg$ is of type $B_{\infty}$
and $\eta'=\pi_{0}=(-\eps_{0},\,\eps_{0}\,;\,0,\,1/2,\,1)$. 
Since $0 \notin \Supp(\lamn)$ by Lemma~\ref{lem:lamn} 
and the choice of $n$, 
we have $\pair{\lamn}{h_{0}}=0$, which implies that 
$e_{0}\pi_{\lamn}=f_{0}\pi_{\lamn}=\bzero$. 
Also, we see from the crystal graph \eqref{eq:binf-cg} that 
neither $e_{0}\eta'$ nor $f_{0}\eta'$ is equal to $\bzero$. 
Therefore, by the tensor product rule for crystals, 
we deduce that 
neither $e_{j}(\pi_{\lamn} \otimes \eta')$ nor
$f_{j}(\pi_{\lamn} \otimes \eta')$ 
is equal to $\bzero$, which contradicts the fact that 
$\pi_{\lamn} \otimes \eta'$ is extremal.
Also, suppose, by contradiction, that $\eta'=\pi_{-\eps_{k}}$ 
for some $k \in \Supp(\lamn)$.
We set $w:=r_{k+2}r_{k+3} \cdots r_{n}r_{n+1} \in W$ 
(if $k=n$, then we set $w:=e$, the identity element of $W$). 
Since $\pair{-\eps_{k}}{h_{j}}=0$ 
for all $k+2 \le j \le n+1$, 
we deduce by the tensor product rule for crystals 
along with Corollary~\ref{cor:isom}\,(2) that 
\begin{equation*}
S_{w}(\pi_{\lamn} \otimes \eta')
 =S_{w}(\pi_{\lamn} \otimes \pi_{-\eps_{k}})
 =(S_{w}\pi_{\lamn}) \otimes \pi_{-\eps_{k}}
 =\pi_{w\lamn} \otimes \pi_{-\eps_{k}}.
\end{equation*}
It is easily seen from 
Lemma~\ref{lem:lamn}, by using \eqref{eq:simple} and 
\eqref{eq:binf-r}--\eqref{eq:dinf-r}, that 
$\pair{w\lamn}{h_{k+1}} < 0$,
which implies that $e_{k+1}\pi_{w\lamn} \ne \bzero$. 
In addition, since $\pair{-\eps_{k}}{h_{k+1}}=1 > 0$, we have 
$f_{k+1}\pi_{-\eps_{k}} \ne \bzero$. 
Consequently, by the tensor product rule for crystals, 
we deduce that neither $e_{k+1}(\pi_{w\lamn} \otimes \pi_{-\eps_{k}})$
nor $f_{k+1}(\pi_{w\lamn} \otimes \pi_{-\eps_{k}})$ is 
equal to $\bzero$, which contradicts the fact that 
$\pi_{w\lamn} \otimes \pi_{-\eps_{k}}=
 S_{w}(\pi_{\lamn} \otimes \eta')$ is extremal. 
Thus, we conclude that $\eta'=\pi_{\eps_{k}}$ 
for some $k \in \BZ_{\ge 0}$, or 
$\eta'=\pi_{-\eps_{k}}$ for some $k \in \BZ_{\ge 0} 
\setminus \Supp(\lamn)$. 

Now, we set $|\xi|:=\sum_{j \in \BZ_{\ge 0}} |\cox{j}|$ 
(recall that $\xi=\wt(\pi_{\lamn} \otimes \eta')$). 
If $\eta'=\pi_{\eps_{k}}$ for some $k \in \BZ_{\ge 0}$,
then since $\co{\lamn}{j} \ge 0$, $j \in \BZ_{\ge 0}$  
(see Lemma~\ref{lem:lamn}), we find that 
\begin{equation*}
|\xi|
 =\sum_{j \in \BZ_{\ge 0}} |\co{\lamn}{j}+\delta_{jk}|
 =\left(\sum_{j \in \BZ_{\ge 0}}|\co{\lamn}{j}|\right) + 1. 
\end{equation*} 
Also, if $\eta'=\pi_{-\eps_{k}}$ for some 
$k \in \BZ_{\ge 0} \setminus \Supp(\lamn)$, 
then we find that 
\begin{equation*}
|\xi|
 =\sum_{j \in \Supp(\lamn)} |\co{\lamn}{j}|+
  \sum_{j \in \BZ_{\ge 0} \setminus \Supp(\lamn)} |\delta_{jk}|
 =\left(\sum_{j \in \BZ_{\ge 0}}|\co{\lamn}{j}|\right) + 1. 
\end{equation*} 
Here it is easily seen by using \eqref{eq:binf-r}--\eqref{eq:dinf-r} 
that $\sum_{j \in \BZ_{\ge 0}}|\co{\lamn}{j}|=|\lambda|$, 
since $\lamn \in W\lambda$. Similarly, we see that 
$|\xi|=|\nu|$ since  $\xi \in W\nu$. Combining these equalities, we obtain 
$|\nu|=|\xi|=|\lambda|+1=|\lambda|+|\mu|$, as desired. 

Assume, therefore, that $|\mu| > 1$. 
We set $p:=\max \Supp (\mu) \in \BZ_{\ge 0}$, and $\mu':=\mu-\eps_{p}$; 
note that $\mu' \in \EW$ and $|\mu'|=|\mu|-1$. 

\begin{claim*}
The element $\pi_{\eps_{p}} \otimes \pi_{\mu'} 
\in \BB(\eps_{0}) \otimes \BB(\mu')$ is 
an extremal element of weight $\mu$. 
Therefore, the connected component 
$\BB(\pi_{\eps_{p}} \otimes \pi_{\mu'})$ 
of $\BB(\eps_{0}) \otimes \BB(\mu')$ containing 
$\pi_{\eps_{p}} \otimes \pi_{\mu'}$ is isomorphic, 
as a $U_{q}(\Fg)$-crystal, to $\BB(\mu)$. 
\end{claim*}

\noindent
{\it Proof of Claim.} 
First, we show that $f_{i}(\pi_{\eps_{p}} \otimes \pi_{\mu'})=\bzero$ 
for all $i \in \BZ_{\ge 1}$. 
Since $\mu' \in \EW$, we have 
$\pair{\mu'}{h_{i}} \le 0$ for all $i \in \BZ_{\ge 1}$, 
which implies that $f_{i}\pi_{\mu'}=\bzero$ 
for all $i \in \BZ_{\ge 1}$. 
If $p=0$, then 
we see from \eqref{eq:binf}--\eqref{eq:dinf} that 
$\pair{\eps_{0}}{h_{i}} \le 0$ for all $i \in \BZ_{\ge 1}$, 
and hence $f_{i}\pi_{\eps_{0}}=\bzero$ 
for all $i \in \BZ_{\ge 1}$. 
In this case, by the tensor product rule for crystals, 
we obtain $f_{i}(\pi_{\eps_{0}} \otimes \pi_{\mu'})=\bzero$ 
for all $i \in \BZ_{\ge 1}$. 
So, suppose that $p > 0$. 
We see from \eqref{eq:binf}--\eqref{eq:dinf} that 
$\pair{\eps_{p}}{h_{i}} \le 0$ for all $i \in \BZ_{\ge 1}$ 
with $i \ne p$. Hence, by the same reasoning as above, 
we obtain $f_{i}(\pi_{\eps_{p}} \otimes \pi_{\mu'})=\bzero$ 
for all $i \in \BZ_{\ge 1}$ with $i \ne p$. 
Now, since $\mu'=\mu-\eps_{p}$ and 
$\pair{\mu}{h_{p}} \le 0$, with $\pair{\eps_{p}}{h_{p}}=1$ 
by \eqref{eq:binf}--\eqref{eq:dinf}, it follows that 
$\pair{\mu'}{h_{p}}=\pair{\mu}{h_{p}}-1 \le -1$, which implies that 
$\ve_{p}(\pi_{\mu'})=
\max \bigl\{k \in \BZ_{\ge 0} \mid 
 e_{p}^{k}\pi_{\mu'}=\bzero
\bigr\} \ge 1$.
Also, since $\pair{\eps_{p}}{h_{p}}=1$, 
we deduce that 
$\vp_{p}(\pi_{\eps_{p}})=
\max \bigl\{k \in \BZ_{\ge 0} \mid 
 f_{p}^{k}\pi_{\eps_{p}}=\bzero
\bigr\}=1$.
Consequently, by the tensor product rule for crystals, 
we obtain $f_{p}(\pi_{\eps_{p}} \otimes \pi_{\mu'})=
\pi_{\eps_{p}} \otimes f_{p}\pi_{\mu'}=\bzero$. 
Thus, we have shown that 
$f_{i}(\pi_{\eps_{p}} \otimes \pi_{\mu'})=\bzero$ 
for all $i \in \BZ_{\ge 1}$. 

To prove that $\pi_{\eps_{p}} \otimes \pi_{\mu'}$ is extremal, 
we show that $\pi_{\eps_{p}} \otimes \pi_{\mu'}$ is 
$[n]$-extremal for all $n \in \BZ_{\ge 1}$ (see Remark~\ref{rem:ext}\,(1)). 
Fix $n \in \BZ_{\ge 1}$. Since $\pi_{\eps_{p}} \otimes \pi_{\mu'}$ is 
$[1,\,n]$-minimal as shown above, it is $[1,\,n]$-extremal 
(see Remark~\ref{rem:ext}\,(3)). 
Let $w=w^{[1,\,n]}_{0}$ be 
the longest element of $\Wm{1,\,n}$, 
and set $\pi \otimes \eta:=
S_{w}(\pi_{\eps_{p}} \otimes \pi_{\mu'})
 \in \BB(\eps_{0}) \otimes \BB(\mu')$; 
note that $\pi \otimes \eta$ is $[1,\,n]$-extremal.
Since the weight of $\pi \otimes \eta$ is 
equal to $w\mu$, which is $[1,\,n]$-dominant, 
and since $\pi \otimes \eta$ is 
$[1,\,n]$-extremal, it follows immediately 
from \eqref{eq:extremal} that 
$\pi \otimes \eta$ is a $[1,\,n]$-maximal element. 
Now, we show that $e_{0}(\pi \otimes \eta)=\bzero$. 
Because the operator $S_{w}$ is defined by using 
only the Kashiwara operators $e_{i}$ and $f_{i}$ 
for $i \in [1,\,n]$, the element $\pi \in \BB(\eps_{0})$ 
must be of the form: 
$\pi=X\pi_{\eps_{p}}$ for some monomial $X$ in 
the Kashiwara operators $e_{i}$ and $f_{i}$ 
for $i \in [1,\,n]$. 
From this fact, we can easily verify by using 
\eqref{eq:binf-cg}--\eqref{eq:dinf-cg} that 
$\pi=\pi_{\eps_{q}}$ for some $q \in \BZ_{\ge 0}$. 
Hence, noting that 
$\pair{\eps_{q}}{h_{0}} \ge 0$ (see \eqref{eq:simple}), 
we get $e_{0}\pi=e_{0}\pi_{\eps_{q}}=\bzero$. 
Similarly, the element $\eta \in \BB(\mu')$ must be of the form:
$\eta=Y\pi_{\mu'}$ for some monomial $Y$ in 
the Kashiwara operators $e_{i}$ and $f_{i}$ 
for $i \in [1,\,n]$. 
Define $N=N_{\mu'}$ as in Remark~\ref{rem:lcm}, and 
write $\eta=Y\pi_{\mu'} \in \BB(\mu')$ as: 
$\eta=(\nu_{1},\,\nu_{2},\,\dots,\,\nu_{N})$
for some $\nu_{1},\,\nu_{2},\,\dots,\,\nu_{N} \in W\mu'$. 
Then, we deduce from the definition of 
the Kashiwara operators $e_{i}$ and $f_{i}$ 
for $i \in [1,\,n]$ that $\nu_{M} \in \Wm{1,\,n}\mu'$ 
for all $1 \le M \le N$. Since $\mu' \in \EW$, 
it is easily seen by using \eqref{eq:binf-r}--\eqref{eq:dinf-r} 
that $\co{\nu_{M}}{j} \ge 0$ for all $j \in \BZ_{\ge 0}$ and 
$1 \le M \le N$. 
Hence we see from \eqref{eq:simple} that 
$\pair{\nu_{M}}{h_{0}} \ge 0$ for all $1 \le M \le N$, 
which implies that $e_{0}\eta=\bzero$. 
From the fact $e_{0}\pi=e_{0}\eta=\bzero$ just observed, 
and the tensor product rule for crystals, 
we obtain $e_{0}(\pi \otimes \eta)=\bzero$, 
as desired. 
Thus, we have shown that 
$\pi \otimes \eta$ is $[n]$-maximal, and hence 
it is $[n]$-extremal by Remark~\ref{rem:ext}\,(2). 
Since $\pi \otimes \eta=
S_{w}(\pi_{\eps_{p}} \otimes \pi_{\mu'})$ and 
$w \in \Wm{1,\,n} \subset \Wn$, 
we conclude that 
$\pi_{\eps_{p}} \otimes \pi_{\mu'}$ is $[n]$-extremal. 
This proves the claim. \bqed

%%%%
\vsp
%%%%

By the claim above, 
we obtain the following embedding of $U_{q}(\Fg)$-crystals: 
\begin{equation*}
\BB(\lambda) \otimes \BB(\mu) \hookrightarrow 
\BB(\lambda) \otimes \BB(\eps_{0}) \otimes \BB(\mu'). 
\end{equation*}
By our assertion for the case $\mu=\eps_{0}$ 
(which is already proved), 
each connected component of $\BB(\lambda) \otimes \BB(\eps_{0})$
is isomorphic, as a $U_{q}(\Fg)$-crystal, to $\BB(\xi)$ 
for some $\xi \in \EW$ such that $|\xi|=|\lambda|+1$.
Furthermore, it follows from our induction hypothesis that 
for $\xi \in \EW$, each connected component of 
$\BB(\xi) \otimes \BB(\mu')$ is isomorphic, 
as a $U_{q}(\Fg)$-crystal, to $\BB(\nu)$ 
for some $\nu \in \EW$ such that $|\nu|=|\xi|+|\mu'|$; 
here, recall that $|\mu'|=|\mu|-1$. 
Therefore, we conclude that 
each connected component of $\BB(\lambda) \otimes \BB(\mu)$
is isomorphic, as a $U_{q}(\Fg)$-crystal, to $\BB(\nu)$ 
for some $\nu \in \EW$ such that $|\nu|=|\lambda|+|\mu|$. 
This completes the proof of part (1). 

(2) Since $\pair{\nu}{h_{i}} \le 0$ for all $i \in \BZ_{\ge 1}$, 
and since $\pi \otimes \eta$ is an extremal element of weight $\nu$ 
by assumption, it follows immediately from \eqref{eq:extremal} that 
$f_{i}(\pi \otimes \eta)=\bzero$ for all $i \in \BZ_{\ge 1}$. 
In particular, $\pi \otimes \eta$ is $[1,\,m]$-minimal. 
It therefore remains to prove that $\pi \in \BBm{1,\,m}(\lambda)$ and 
$\eta \in \BBm{1,\,m}(\mu)$. 
Define $N_{\lambda}$ and $N_{\mu}$ 
as in Remark~\ref{rem:lcm}, and write $\pi \in \BB(\lambda)$ 
and $\eta \in \BB(\mu)$ as:
\begin{equation*}
\pi=(\zeta_{1},\,\zeta_{2},\,\dots,\,\zeta_{N_{\lambda}}), \qquad
\eta=(\xi_{1},\,\xi_{2},\,\dots,\,\xi_{N_{\mu}})
\end{equation*}
for some $\zeta_{1},\,\zeta_{2},\,\dots,\,\zeta_{N_{\lambda}} \in W\lambda$ and 
$\xi_{1},\,\xi_{2},\,\dots,\,\xi_{N_{\mu}} \in W\mu$, respectively. 
Then we have (see \eqref{eq:wtN})
%
%%%%%%%%%%%%%%%
%%% eq:wtNN %%%
%%%%%%%%%%%%%%%
%
\begin{equation} \label{eq:wtNN}
\nu=\frac{1}{N_{\lambda}}\sum_{M=1}^{N_{\lambda}}\zeta_{M}+
\frac{1}{N_{\mu}}\sum_{L=1}^{N_{\mu}}\xi_{L}.
\end{equation}
Hence we obtain
\begin{align*}
|\nu| & =\sum_{j \in \BZ_{\ge 0}}\con{j}=
\sum_{j \in \BZ_{\ge 0}}
\left\{
\frac{1}{N_{\lambda}}\sum_{M=1}^{N_{\lambda}}\co{\zeta_{M}}{j}+
\frac{1}{N_{\mu}}\sum_{L=1}^{N_{\mu}}\co{\xi_{L}}{j}
\right\} \\[3mm]
& = 
\frac{1}{N_{\lambda}}\sum_{M=1}^{N_{\lambda}}
\left\{\sum_{j \in \BZ_{\ge 0}}\co{\zeta_{M}}{j}\right\}+
\frac{1}{N_{\mu}}\sum_{L=1}^{N_{\mu}}
\left\{\sum_{j \in \BZ_{\ge 0}}\co{\xi_{L}}{j}\right\}.
\end{align*}
Since $\zeta_{M} \in W\lambda$ for all $1 \le M \le N_{\lambda}$, 
and $\lambda \in \EW$, 
we deduce that 
\begin{equation*}
\sum_{j \in \BZ_{\ge 0}}\co{\zeta_{M}}{j} \le 
|\lambda|=\sum_{j \in \BZ_{\ge 0}}|\col{j}|
\end{equation*}
for all $1 \le M \le N_{\lambda}$; note that 
%
%%%%%%%%%%%%%%%%%%
%%% eq:eekey01 %%%
%%%%%%%%%%%%%%%%%%
%
\begin{equation} \label{eq:eekey01}
\sum_{j \in \BZ_{\ge 0}}\co{\zeta_{M}}{j}=|\lambda| 
\quad \text{if and only if} \quad
\text{$\co{\zeta_{M}}{j} \ge 0$ for all $j \in \BZ_{\ge 0}$}.
\end{equation} 
Similarly, since $\xi_{L} \in W\mu$ 
for all $1 \le L \le N_{\mu}$, and $\mu \in \EW$, 
we have 
\begin{equation*}
\sum_{j \in \BZ_{\ge 0}}\co{\xi_{L}}{j} \le 
|\mu|=\sum_{j \in \BZ_{\ge 0}} |\com{j}|
\end{equation*}
for all $1 \le L \le N_{\mu}$; note that 
%
%%%%%%%%%%%%%%%%%%
%%% eq:eekey02 %%%
%%%%%%%%%%%%%%%%%%
%
\begin{equation} \label{eq:eekey02}
\sum_{j \in \BZ_{\ge 0}}\co{\xi_{L}}{j}=|\mu|
\quad \text{if and only if} \quad
\text{$\co{\xi_{L}}{j} \ge 0$ for all $j \in \BZ_{\ge 0}$}.
\end{equation}
Combining these, we infer that 
%
%%%%%%%%%%%%%%%%%
%%% eq:eekey1 %%%
%%%%%%%%%%%%%%%%%
%
\begin{equation} \label{eq:eekey1}
|\nu|=
\frac{1}{N_{\lambda}}\sum_{M=1}^{N_{\lambda}}
\underbrace{\left\{ \sum_{j \in \BZ_{\ge 0}}\co{\zeta_{M}}{j} \right\}}_{\le |\lambda|}+
\frac{1}{N_{\mu}}\sum_{L=1}^{N_{\mu}}
\underbrace{\left\{ \sum_{j \in \BZ_{\ge 0}}\co{\xi_{L}}{j} \right\}}_{\le |\mu|}
\le |\lambda|+|\mu|.
\end{equation}
Note that by Proposition~\ref{prop:isom}, 
the connected component of $\BB(\lambda) \otimes \BB(\mu)$ 
containing $\pi \otimes \eta$ is isomorphic, as $U_{q}(\Fg)$-crystal, to 
$\BB(\nu)$. Therefore, by part (1), we obtain $|\nu|=|\lambda|+|\mu|$. 
From this fact and \eqref{eq:eekey1}, we deduce that 
$\sum_{j \in \BZ_{\ge 0}}\co{\zeta_{M}}{j}=|\lambda|$ 
for all $1 \le M \le N_{\lambda}$, and 
$\sum_{j \in \BZ_{\ge 0}}\co{\xi_{L}}{j}=|\mu|$ 
for all $1 \le L \le N_{\mu}$, and hence that 
by \eqref{eq:eekey01} and \eqref{eq:eekey02}, 
%
%%%%%%%%%%%%%%%%%
%%% eq:eekey2 %%%
%%%%%%%%%%%%%%%%%
%
\begin{equation} \label{eq:eekey2}
\begin{cases}
\co{\zeta_{M}}{j} \ge 0 & 
\text{for all $j \in \BZ_{\ge 0}$ and $1 \le M \le N_{\lambda}$}, \\[1.5mm]
\co{\xi_{L}}{j} \ge 0 & 
\text{for all $j \in \BZ_{\ge 0}$ and $1 \le L \le N_{\mu}$}.
\end{cases}
\end{equation}
Since $\ell(\nu) \le |\nu|=|\lambda|+|\mu|$, and 
$m \ge |\lambda|+|\mu|$ by assumption, 
we have $\con{j}=0$ for all $j \ge m+1$. 
Consequently, by \eqref{eq:eekey2} and \eqref{eq:wtNN}, 
%
%%%%%%%%%%%%%%%%%
%%% eq:eekey3 %%%
%%%%%%%%%%%%%%%%%
%
\begin{equation} \label{eq:eekey3}
\begin{cases}
\co{\zeta_{M}}{j}=0 & 
\text{for all $j \ge m+1$ and $1 \le M \le N_{\lambda}$}, \\[1.5mm]
\co{\xi_{L}}{j}=0 & 
\text{for all $j \ge m+1$ and $1 \le L \le N_{\mu}$}.
\end{cases}
\end{equation}
Now, taking \eqref{eq:eekey2} and \eqref{eq:eekey3} into consideration, 
we see through use of \eqref{eq:binf-r}--\eqref{eq:dinf-r} that 
for each $1 \le M \le N_{\lambda}$, 
there exists $w_{M} \in \Wm{1,\,m}$ such that 
$w_{M}\zeta_{M} \in \EW$. However, since 
$\zeta_{M} \in W\lambda$ for $1 \le M \le N_{\lambda}$ 
and $\lambda=\lambda_{\dagger} \in \EW$, it follows from 
the uniqueness of $\lambda_{\dagger} \in W\lambda$ that 
$w_{M}\zeta_{M}=\lambda_{\dagger}=\lambda$ for all $1 \le M \le N_{\lambda}$. 
Similarly, since $\mu=\mu_{\dagger} \in \EW$,
we see from \eqref{eq:eekey2} and \eqref{eq:eekey3} that 
for each $1 \le L \le N_{\mu}$, 
there exists $z_{L} \in \Wm{1,\,m}$ such that 
$z_{L}\xi_{L}=\mu$. Therefore, by Lemma~\ref{lem:BBn}, 
we conclude that $\pi \in \BBm{1,\,m}(\lambda)$ and 
$\eta \in \BBm{1,\,m}(\mu)$. This completes the proof of part (2). 

(3) Let $\pi \otimes \eta \in \BBm{1,\,m}(\lambda) \otimes 
\BBm{1,\,m}(\mu)$ be an arbitrary $[1,\,m]$-minimal element. 
First, we show that $f_{i}(\pi \otimes \eta)=\bzero$ 
for all $i \ge m+1$, and hence for all $i \in \BZ_{\ge 1}$. 
Define $N=N_{\lambda}$ as in Remark~\ref{rem:lcm}, and write 
$\pi \in \BBm{1,\,m}(\lambda) \subset \BB(\lambda)$ as:
\begin{equation*}
\pi=(\zeta_{1},\,\zeta_{2},\,\dots,\,\zeta_{N})
\end{equation*}
for some $\zeta_{1},\,\zeta_{2},\,\dots,\,\zeta_{N} \in W\lambda$.
Since $\pi \in \BBm{1,\,m}(\lambda)$, 
it follows from Lemma~\ref{lem:BBn} that 
$\zeta_{M} \in \Wm{1,\,m}\lambda$ 
for all $1 \le M \le N$. 
Because $\lambda \in \EW$ and $\Supp(\lambda)=[\ell(\lambda)-1] 
\subset [m]$, we deduce through use of 
\eqref{eq:binf-r}--\eqref{eq:dinf-r} that 
for each $1 \le M \le N$, 
$\co{\zeta_{M}}{j} \ge 0$ for all $j \in \BZ_{\ge 0}$, 
and $\co{\zeta_{M}}{j}=0$ for all $j \ge m+1$. 
Therefore, by \eqref{eq:simple}, 
we obtain $\pair{\zeta_{M}}{h_{i}} \le 0$ 
for all $1 \le M \le N$ and $i \ge m+1$, 
from which it easily follows that $f_{i}\pi=\bzero$ for all $i \ge m+1$. 
An entirely similar argument shows that 
$f_{i}\eta=\bzero$ for all $i \ge m+1$. 
Consequently, by the tensor product rule for crystals, 
we obtain $f_{i}(\pi \otimes \eta)=\bzero$ for all $i \ge m+1$. 
Furthermore, in the course of the argument above, 
we also find (see \eqref{eq:wtNN}) that 
if we set $\nu:=\wt(\pi \otimes \eta)$, then 
$\con{j} \ge 0$ for all $j \in \BZ_{\ge 0}$. 
In addition, since $\BB(\lambda) \otimes \BB(\mu)$ 
is a normal $U_{q}(\Fg)$-crystal, 
and $f_{i}(\pi \otimes \eta)=\bzero$ for all $i \in \BZ_{\ge 1}$,  
it follows (for example, from the representation theory of $\Fsl_{2}(\BC)$)
that $\pair{\nu}{h_{i}} \le 0$ for all $i \in \BZ_{\ge 1}$. 
Combining these inequalities, we conclude that $\nu \in \EW$. 

To prove that $\pi \otimes \eta$ is extremal, 
we show that $\pi \otimes \eta$ is 
$[n]$-extremal for all $n \in \BZ_{\ge m}$ 
(see Remark~\ref{rem:ext}\,(1)). Fix $n \in \BZ_{\ge m}$. 
Let $w=w^{[1,\,n]}_{0}$ be 
the longest element of $\Wm{1,\,n}$, 
and set $\pi' \otimes \eta':=S_{w}(\pi \otimes \eta) 
\in \BBm{1,\,n}(\lambda) \otimes \BBm{1,\,n}(\mu)$. 
Since $\pi \otimes \eta$ is $[1,\,n]$-minimal from what is shown above, 
we see from Remark~\ref{rem:ext}\,(3) that 
$\pi \otimes \eta$ is $[1,\,n]$-extremal, and hence 
so is $\pi' \otimes \eta'=S_{w}(\pi \otimes \eta)$. 
Because the weight of $\pi' \otimes \eta'$ is equal to $w\nu$, 
which is $[1,\,n]$-dominant, and because $\pi' \otimes \eta'$ is 
$[1,\,n]$-extremal, it follows immediately from \eqref{eq:extremal} that 
$\pi' \otimes \eta'$ is $[1,\,n]$-maximal.
Furthermore, by the argument used 
to show the equality ``$e_{0}\eta=\bzero$'' 
in the proof of the Claim (for part (1)), 
we can show that $e_{0}\pi'=e_{0}\eta'=\bzero$, and hence 
$e_{0}(\pi' \otimes \eta')=\bzero$. 
Thus, we have shown that 
$\pi' \otimes \eta'$ is $[n]$-maximal, and hence 
it is $[n]$-extremal by Remark~\ref{rem:ext}\,(2). 
This implies that 
$\pi \otimes \eta$ is $[n]$-extremal 
since $\pi' \otimes \eta'=
S_{w}(\pi \otimes \eta)$. 
This completes the proof of part (3). 
\end{proof}

\begin{proof}[Proof of Theorem~\ref{thm:EE}]
Let $\nu \in \EW$. If $|\nu| \ne |\lambda|+|\mu|$, then 
it follows from Proposition~\ref{prop:EE-key}\,(1)
that there exists no connected component of 
$\BB(\lambda) \otimes \BB(\mu)$ isomorphic to $\BB(\nu)$, 
and hence that $m_{\lambda,\,\mu}^{\nu}=0$ 
in the decomposition \eqref{eq:EE}. 
Also, recall that $\LR_{\lambda,\,\mu}^{\nu}=0$ 
if $|\nu| \ne |\lambda|+|\mu|$. Hence we have 
$m_{\lambda,\,\mu}^{\nu}=0=\LR_{\lambda,\,\mu}^{\nu}$ in this case.

Assume, therefore, that $|\nu|=|\lambda|+|\mu|$. 
We see from Corollary~\ref{cor:isom}\,(3) and 
Proposition~\ref{prop:isom} that 
there exists a one-to-one correspondence between 
the set of extremal elements of weight $\nu$ 
in $\BB(\lambda) \otimes \BB(\mu)$ and 
the set of connected components of 
$\BB(\lambda) \otimes \BB(\mu)$ that is isomorphic to $\BB(\nu)$.
From this fact, we deduce that 
the multiplicity $m_{\lambda,\,\mu}^{\nu}$ 
in the decomposition \eqref{eq:EE} is 
equal to the number of extremal elements of weight $\nu$ in 
$\BB(\lambda) \otimes \BB(\mu)$; namely, we obtain 
\begin{equation*}
m_{\lambda,\,\mu}^{\nu}=
 \#\bigl\{\pi \otimes \eta \in \BB(\lambda) \otimes \BB(\mu) \mid 
 \text{$\pi \otimes \eta$ is an extremal element of weight $\nu$}
 \bigr\}.
\end{equation*}
Fix $m \in \BZ_{\ge 1}$ such that $m \ge |\lambda|+|\mu|$. 
Then we deduce from Proposition~\ref{prop:EE-key}\,(2), (3) that 
the set of extremal elements of weight $\nu$ in 
$\BB(\lambda) \otimes \BB(\mu)$ is identical to the set of 
$[1,\,m]$-minimal elements of weight $\nu$ in 
$\BBm{1,\,m}(\lambda) \otimes \BBm{1,\,m}(\mu)$. Thus,
\begin{equation*}
 \begin{split}
 m_{\lambda,\,\mu}^{\nu}=
 \#\bigl\{\pi \otimes \eta \in 
   & \BBm{1,\,m}(\lambda) \otimes \BBm{1,m}(\mu) \mid \\
 & \text{$\pi \otimes \eta$ is a $[1,\,m]$-minimal element of weight $\nu$}
   \bigr\}.
 \end{split}
\end{equation*}
Recall that $\Fgm{1,\,m}$ is a ``reductive'' 
Lie algebra of type $A_{m}$ (see Remark~\ref{rem:Levi}), and that
$\lambda$, $\mu$, and $\nu$ are $[1,\,m]$-antidominant since 
they are contained in $\EW$. 
Therefore, by Remark~\ref{rem:LS-n}, 
$\BBm{1,\,m}(\lambda)$ (resp., $\BBm{1,\,m}(\mu)$) 
is isomorphic, as a $U_{q}(\Fgm{1,\,m})$-crystal, to 
the crystal basis of the finite-dimensional irreducible 
$U_{q}(\Fgm{1,\,m})$-module $\Vm{1,\,m}(\lambda)$ 
(resp., $\Vm{1,\,m}(\mu)$) of lowest weight $\lambda$ 
(resp., $\mu$). Consequently, 
the number of $[1,\,m]$-minimal elements of weight $\nu$ 
in the tensor product 
$\BBm{1,\,m}(\lambda) \otimes \BBm{1,m}(\mu)$ is equal to 
the multiplicity $[\Vm{1,\,m}(\lambda) \otimes \Vm{1,\,m}(\mu):
\Vm{1,\,m}(\nu)]$ of the finite-dimensional irreducible 
$U_{q}(\Fgm{1,\,m})$-module $\Vm{1,\,m}(\nu)$
of lowest weight $\nu$ in the tensor product $U_{q}(\Fgm{1,\,m})$-module 
$\Vm{1,\,m}(\lambda) \otimes \Vm{1,\,m}(\mu)$. 
Summarizing, we have 
\begin{equation*}
m_{\lambda,\,\mu}^{\nu}=
[\Vm{1,\,m}(\lambda) \otimes \Vm{1,\,m}(\mu):\Vm{1,\,m}(\nu)].
\end{equation*}
Here, noting that 
\begin{equation*}
\begin{cases}
\ell(\lambda) \le |\lambda| \le |\lambda|+|\mu| \le m, \\[1.5mm]
\ell(\mu) \le |\mu| \le |\lambda|+|\mu| \le m, \\[1.5mm]
\ell(\nu) \le |\nu|=|\lambda|+|\mu| \le m, 
\end{cases}
\end{equation*}
we can easily check that 
the Young diagram corresponding canonically to the highest weight of 
the finite-dimensional irreducible $U_{q}(\Fgm{m})$-module 
$\Vm{1,\,m}(\lambda)$ (resp., $\Vm{1,\,m}(\mu)$, $\Vm{1,\,m}(\nu)$) 
is identical to the Young diagram of the partition $\lambda$ 
(resp., $\mu$, $\nu$) $\in \EW \cong \Par$. Therefore, 
it follows immediately (see, for example, 
\cite[Chapter~8, Section~3, Corollary~2]{Fulton}) that 
\begin{equation*}
m_{\lambda,\,\mu}^{\nu}=
[\Vm{1,\,m}(\lambda) \otimes \Vm{1,\,m}(\mu):\Vm{1,\,m}(\nu)]=
\LR_{\lambda,\,\mu}^{\nu}.
\end{equation*}
Thus, we have proved Theorem~\ref{thm:EE}.
\end{proof}

%==============================%
%     START SUBSECTION 0402    %
%==============================%
%
\subsection{The case $\lambda \in E$, $\mu \in P_{+}$.}
\label{subsec:ED}
First we prove the following proposition 
(cf. \cite[the proof of Corollary~4.11]{Kw-Adv} in type $A_{+\infty}$ and 
\cite[Proposition~3.11]{Kw-Ep} in type $A_{\infty}$). 
%
%%%%%%%%%%%%%%%%%%%%
%%% prop:ED-conn %%%
%%%%%%%%%%%%%%%%%%%%
%
\begin{prop} \label{prop:ED-conn}
Let $\lambda \in E$, and $\mu \in P_{+}$. 
The crystal graph of the tensor product 
$\BB(\lambda) \otimes \BB(\mu)$ is connected.
\end{prop}

\begin{proof}
Since $\mu \in P_{+}$, 
it follows from Theorem~\ref{thm:isom} and 
Remark~\ref{rem:extmod}\,(1) that 
$\BB(\mu)$ is isomorphic, as a $U_{q}(\Fg)$-crystal, 
to the crystal basis of the irreducible highest weight 
$U_{q}(\Fg)$-module of highest weight $\mu$. Therefore, 
the weight $\wt \eta$ of an element $\eta \in \BB(\mu)$ 
is of the form:
$\wt (\eta)= \mu- 
 \sum_{i \in I} m_{i}\alpha_{i}$
for some $m_{i} \in \BZ_{\ge 0}$, $i \in I$; 
in this case, we set 
\begin{equation*}
\Dep(\eta):=\sum_{i \in I} m_{i} \in \BZ_{\ge 0}.
\end{equation*}
Let $\BB(\pi_{\lambda} \otimes \pi_{\mu})$ denote 
the connected component of $\BB(\lambda) \otimes \BB(\mu)$ 
containing $\pi_{\lambda} \otimes \pi_{\mu}$. 
We will show by induction on $\Dep(\eta)$ that 
every element $\pi \otimes \eta \in \BB(\lambda) \otimes \BB(\mu)$ 
is contained in this connected component 
$\BB(\pi_{\lambda} \otimes \pi_{\mu})$. 

First, suppose that $\Dep(\eta)=0$. 
Then we have $\eta=\pi_{\mu}$.
Since the crystal graph of $\BB(\lambda)$ is connected by 
Corollary~\ref{cor:isom}\,(1), 
there exists a monomial $X_{1}$ in the Kashiwara operators 
$e_{i}$ and $f_{i}$ for $i \in I$ 
such that $X_{1}\pi=\pi_{\lambda}$. 
Note that if $f_{i}\pi' \ne \bzero$ 
(resp., $e_{i}\pi' \ne \bzero$)
for some $\pi' \in \BB(\lambda)$ and $i \in I$, 
then $f_{i}(\pi' \otimes \pi_{\mu})=
 (f_{i}\pi') \otimes \pi_{\mu}$ 
(resp., $e_{i}(\pi' \otimes \pi_{\mu})=
 (e_{i}\pi') \otimes \pi_{\mu}$) 
by the tensor product rule for crystals. 
From this, it follows that 
\begin{equation*}
X_{1}(\pi \otimes \pi_{\mu})=
(X_{1}\pi) \otimes \pi_{\mu}=
\pi_{\lambda} \otimes \pi_{\mu}, 
\end{equation*}
as desired. 

Next, suppose that $\Dep(\eta) > 0$; 
note that $e_{i}\eta \ne \bzero$ for some $i \in I$ 
since $\eta$ is not the (unique) maximal element 
$\pi_{\mu}$ of $\BB(\mu)$. Now we take 
$n \in \BZ_{\ge 0}$ such that 

(i) $\Supp(\lambda) \subsetneqq 
[n]=\bigl\{0,\,1,\,\dots,\,n\bigr\}$;

(ii) $\pi \in \BBn(\lambda)$ (see \eqref{eq:limBBn});

(iii) $e_{i}\eta \ne \bzero$ for some $i \in [n]$.

\noindent
It follows from condition (ii) and Remark~\ref{rem:LS-n} that 
there exists a monomial $X_{2}$ in $e_{i}^{\max}$, $i \in [n]$, 
such that $X_{2}\pi=\pi_{\lamn}$; 
by the tensor product rule for crystals, we have
\begin{equation*}
X_{2}(\pi \otimes \eta)=
(X_{2}\pi) \otimes \eta'=
\pi_{\lamn} \otimes \eta'
\end{equation*}
for some $\eta' \in \BB(\mu)$. 
If $\eta' \ne \eta$, then we have $\Dep(\eta') < \Dep(\eta)$ 
since $X_{2}$ is a monomial in $e_{i}^{\max}$, $i \in [n]$. 
Therefore, by our induction hypothesis, 
$\pi_{\lamn} \otimes \eta'$ is contained in 
the connected component $\BB(\pi_{\lambda} \otimes \pi_{\mu})$, 
and hence so is $\pi \otimes \eta$. Thus, we may assume that 
$\eta'=\eta$, i.e., 
%
%%%%%%%%%%%%%
%%% eq:A1 %%%
%%%%%%%%%%%%%
%
\begin{equation} \label{eq:A1}
X_{2}(\pi \otimes \eta)=
\pi_{\lamn} \otimes \eta.
\end{equation}

We set $w_{i}:=r_{i+1}r_{i+2} \cdots r_{n}r_{n+1} \in W$ 
for $0 \le i \le n$, and $w_{n+1}:=e$, 
the identity element of $W$.
\begin{claim*}
We have
\begin{equation*}
\pair{w_{i}\lamn}{h_{i}} \le 0 \quad
\text{\rm for all $0 \le i \le n+1$}.
\end{equation*}
\end{claim*}

\noindent
{\it Proof of Claim.}
Set $p:=\#\Supp (\lambda)$. It follows 
from Lemma~\ref{lem:lamn} and condition (i) that 
\begin{equation*}
\Supp (\lamn)=[n-p+1,\,n], \quad 
\text{\rm and} \quad
0 < \co{\lamn}{n-p+1} \le \cdots \le 
\co{\lamn}{n-1} \le \co{\lamn}{n}.
\end{equation*}
Let $0 \le i \le n+1$. By using 
\eqref{eq:binf-r}--\eqref{eq:dinf-r}, 
we deduce from \eqref{eq:binf}--\eqref{eq:dinf} that 
\begin{equation*}
w_{i}^{-1}h_{i}=
\begin{cases}
w_{i}^{-1}(-\eps'_{i-1}+\eps'_{i})=-\eps'_{i-1}+\eps'_{n+1}, 
  & \text{if $i \ge 1$}, \\[1.5mm]
w_{0}^{-1}(2\eps'_{0})=2\eps'_{n+1}, 
  & \text{if $i=0$, and $\Fg$ is of type $B_{\infty}$}, \\[1.5mm]
w_{0}^{-1}(\eps'_{0})=\eps'_{n+1}, 
  & \text{if $i=0$, and $\Fg$ is of type $C_{\infty}$}, \\[1.5mm]
w_{0}^{-1}(\eps'_{0}+\eps'_{1})=\eps'_{0}+\eps'_{n+1}, 
  & \text{if $i=0$, and $\Fg$ is of type $D_{\infty}$}.
\end{cases}
\end{equation*}
Since $n-p+1 > 0$ by condition (i), we have
$\pair{w_{i}\lamn}{h_{i}}=
 \pair{\lamn}{w_{i}^{-1}h_{i}} \le 0$ 
in all the cases above. This proves the claim. 
\bqed

\vsp

Let $0 \le i \le n$. By the tensor product rule 
for crystals, we have 
\begin{align*}
& 
e_{i+1}^{\max}e_{i+2}^{\max} \cdots 
e_{n}^{\max}e_{n+1}^{\max}X_{2}(\pi \otimes \eta) \\
& \hspace*{10mm} =
e_{i+1}^{\max}e_{i+2}^{\max} \cdots 
e_{n}^{\max}e_{n+1}^{\max}(\pi_{\lamn} \otimes \eta) 
\qquad \text{by \eqref{eq:A1}} \\
& \hspace*{10mm} =
(e_{i+1}^{\max}e_{i+2}^{\max} \cdots 
 e_{n}^{\max}e_{n+1}^{\max}\pi_{\lamn}) \otimes \eta_{i}
\end{align*}
for some $\eta_{i} \in \BB(\mu)$. 
Using the claim above successively, we find that 
\begin{equation*}
e_{i+1}^{\max}e_{i+2}^{\max} \cdots 
e_{n}^{\max}e_{n+1}^{\max}\pi_{\lamn} 
 = \pi_{r_{i+1}r_{i+2} \cdots r_{n}r_{n+1}\lamn}
 = \pi_{w_{i}\lamn},
\end{equation*}
and hence
\begin{equation*}
e_{i+1}^{\max}e_{i+2}^{\max} \cdots 
e_{n}^{\max}e_{n+1}^{\max}X_{2}(\pi \otimes \eta)=
\pi_{w_{i}\lamn} \otimes \eta_{i}. 
\end{equation*}
Let us take $i \in [n]$ such that $e_{i}\eta \ne \bzero$ 
(recall condition (iii)). If $\eta_{i} \ne \eta$, 
then we have $\Dep(\eta_{i}) < \Dep(\eta)$, and hence 
by our induction hypothesis, 
$\pi_{w_{i}\lamn} \otimes \eta_{i}$ is contained in 
$\BB(\pi_{\lambda} \otimes \pi_{\mu})$, which implies that 
$\pi \otimes \eta \in \BB(\pi_{\lambda} \otimes \pi_{\mu})$. 
Suppose, therefore, that $\eta_{i}=\eta$. 
Because $\pair{w_{i}\lamn}{h_{i}} \le 0$ 
by the claim above, and because 
$e_{i}\eta \ne \bzero$ by the choice of $i \in [n]$, 
it follows from the tensor product rule for crystals that
\begin{equation*}
e_{i}e_{i+1}^{\max}e_{i+2}^{\max} \cdots 
e_{n}^{\max}e_{n+1}^{\max}X_{2}(\pi \otimes \eta)=
e_{i}(\pi_{w_{i}\lamn} \otimes \eta)=
\pi_{w_{i}\lamn} \otimes (e_{i}\eta) \quad (\ne \bzero).
\end{equation*}
Since $\Dep(e_{i}\eta) < \Dep(\eta)$, the element 
$\pi_{w_{i}\lamn} \otimes (e_{i}\eta)$ is 
contained in $\BB(\pi_{\lambda} \otimes \pi_{\mu})$ 
by our induction hypothesis. This implies that 
$\pi \otimes \eta \in \BB(\pi_{\lambda} \otimes \pi_{\mu})$, 
thereby completing the proof of Proposition~\ref{prop:ED-conn}.  
\end{proof}

The following theorem extends 
\cite[Corollary~4.11]{Kw-Adv} in type $A_{+\infty}$ and 
\cite[Theorem~4.6]{Kw-Ep} in type $A_{\infty}$.
%
%%%%%%%%%%%%%%
%%% thm:ED %%%
%%%%%%%%%%%%%%
%
\begin{thm} \label{thm:ED}
Let $\lambda \in E$, and $\mu \in P_{+}$.
Then, 
\begin{equation*}
\BB(\lambda) \otimes \BB(\mu) \cong 
\BB(\xi+\mu)
\quad \text{\rm as $U_{q}(\Fg)$-crystals}
\end{equation*}
for some $\xi \in W\lambda$.
\end{thm}

\begin{proof}
We set $q:=\min\bigl\{j \in \BZ_{\ge 0} \mid \com{j}=L_{\mu}\bigr\}$; 
we see from Remark~\ref{rem:dom} that 
\begin{equation*}
0 \le \underbrace{\abs{\com{0}} \le \com{1} \le 
\cdots \le \com{q-1} <}_{\text{If $q=0$, then this part is omitted.}} 
L_{\mu}=\com{q}=\com{q+1}=\cdots.
\end{equation*}
Also, we set $p:=\#\Supp (\lambda)$. 
Then we deduce by using \eqref{eq:binf-r}--\eqref{eq:dinf-r} that 
there exists a unique element $\xi \in W\lambda$ satisfying 
the conditions that $\Supp(\xi)=[q,\,q+p-1]$, and that 
%
%%%%%%%%%%%%%
%%% eq:xi %%%
%%%%%%%%%%%%%
%
\begin{equation} \label{eq:xi}
0 < \co{\xi}{q} \le \co{\xi}{q+1} \le \cdots \le \co{\xi}{q+p-1}; 
\end{equation}
note that $\pi_{\xi} \in \BB(\lambda)$ 
(see Remark~\ref{rem:LS}\,(1), (2)). 
We will prove that 
$\pi_{\xi} \otimes \pi_{\mu} \in 
\BB(\lambda) \otimes \BB(\mu)$ is an extremal element. 
If we prove this assertion, 
then it follows immediately from 
Propositions~\ref{prop:isom} and \ref{prop:ED-conn} 
that $\BB(\lambda) \otimes \BB(\mu) \cong 
\BB(\xi+\mu)$ as $U_{q}(\Fg)$-crystals, 
which is what we want to prove. 

First we show the following claim.
\begin{claim*}
Let $\beta \in \Delta$ be a root for $\Fg$. 

{\rm (1)} 
If $\pair{\xi}{\beta^{\vee}} > 0$, 
then $\pair{\mu}{\beta^{\vee}} \ge 0$. 

{\rm (2)} 
If $\pair{\mu}{\beta^{\vee}} > 0$, then 
$\pair{\xi}{\beta^{\vee}} \ge 0$.
\end{claim*}

\noindent
{\it Proof of Claim.}
The assertions are obvious in the case 
$\xi=0$ (or equivalently, $\lambda=0$). 
We may therefore assume that 
$\xi \ne 0$ (or equivalently, $\lambda \ne 0$); 
note that $p=\Supp (\lambda) > 0$. 

(1) If $\beta$ is a positive root, then 
we have $\pair{\mu}{\beta^{\vee}} \ge 0$ 
since $\mu \in P_{+}$. Therefore, we may assume that 
$\beta$ is a negative root; note that 
the set $\Delta^{-}$ of negative roots is as follows 
(see \eqref{eq:binf-rs}--\eqref{eq:dinf-rs}): 
%
%%%%%%%%%%%%%
%%% eq:nr %%%
%%%%%%%%%%%%%
%
\begin{equation} \label{eq:nr}
\Delta^{-}=
\begin{cases}
\bigl\{-\eps_{j} \mid 1 \le j\bigr\} \cup 
\bigl\{\eps_{j}-\eps_{i},\ -\eps_{j}-\eps_{i} \mid 0 \le j < i\bigr\}
& \text{if $\Fg$ is of type $B_{\infty}$}, \\[1.5mm]
\bigl\{-2\eps_{j} \mid 1 \le j\bigr\} \cup 
\bigl\{\eps_{j}-\eps_{i},\ -\eps_{j}-\eps_{i} \mid 0 \le j < i\bigr\}
& \text{if $\Fg$ is of type $C_{\infty}$}, \\[1.5mm]
\bigl\{\eps_{j}-\eps_{i},\ -\eps_{j}-\eps_{i} \mid 0 \le j < i\bigr\}
& \text{if $\Fg$ is of type $D_{\infty}$}.
\end{cases}
\end{equation}
The assumption $\pair{\xi}{\beta^{\vee}} > 0$ 
now implies that $\beta$ is of the form: 
$\eps_{j}-\eps_{i}$ for some 
$q \le j \le q+p-1$ and $i \ge q+p$. 
Since $\com{j}=L_{\mu}$ for all $j \ge q$ by the definition of $q$, 
it follows that $\pair{\mu}{\beta^{\vee}}=0$. 
This proves part (1) of the claim. 

(2) The assumption $\pair{\mu}{\beta^{\vee}} > 0$ 
implies that $\beta$ is a positive root. 
If $\beta$ is a sum of positive integer multiples of 
$\eps_{j}$'s, $j \in \BZ_{\ge 0}$, then it is obvious that 
$\pair{\xi}{\beta^{\vee}} \ge 0$. Therefore, we may assume that 
$\beta$ is of the form: $-\eps_{j}+\eps_{i}$ 
for some $0 \le j < i$. 
Since $\com{j}=L_{\mu}$ for all $j \ge q$  by the definition of $q$, and 
since $\pair{\mu}{\beta^{\vee}} > 0$ by assumption,  
it follows that $j \le q-1$. Because 
$\Supp(\xi)=\bigl\{q,\,q+1,\,\dots,\,q+p-1\bigr\}$ and 
$0 < \co{\xi}{q} \le \co{\xi}{q+1} \le \cdots \le \co{\xi}{q+p-1}$, 
we obtain $\pair{\xi}{\beta^{\vee}} \ge 0$, as desired. 
This proves part (2) of the claim. \nolinebreak \bqed

%%%%
\vsp
%%%%

By an argument entirely similar to the one for 
\cite[Lemma~1.6\,(1)]{AK}, 
we deduce, using the claim above, that 
$S_{w}(\pi_{\xi} \otimes \pi_{\mu})=
 \pi_{w\xi} \otimes \pi_{w\mu}$ for all $w \in W$ 
(the proof proceeds by induction on the length of $w \in W$ 
with the help of Corollary~\ref{cor:isom}\,(2)). 
From this, it easily follows that 
$\pi_{\xi} \otimes \pi_{\mu}$ is extremal. 
This completes the proof of Theorem~\ref{thm:ED}.
\end{proof}
%
%%%%%%%%%%%%%%
%%% rem:ED %%%
%%%%%%%%%%%%%%
%
\begin{rem} \label{rem:ED}
Let $\lambda \in P$ be an integral weight such that 
$L_{\lambda} \ge 0$. We set 
\begin{equation*}
q:=\#\bigl\{j \in \BZ_{\ge 0} \mid |\col{j}| < L_{\lambda} \bigr\}
\quad \text{and} \quad
p:=\#\bigl\{j \in \BZ_{\ge 0} \mid |\col{j}| > L_{\lambda}\bigr\}.
\end{equation*}
Then, we deduce by using \eqref{eq:binf-r}--\eqref{eq:dinf-r} that 
there exists a unique element $\nu \in W\lambda$ 
satisfying the conditions that 
\begin{equation*}
\begin{cases}
0 \le 
\underbrace{\abs{\con{0}} \le \con{1} \le \cdots \le \con{q-1} <}_{%
\text{If $q=0$, then this part is omitted.}} L_{\lambda} 
\underbrace{< \con{q} \le \con{q+1} \le \cdots \le \con{q+p-1}}_{%
\text{If $p=0$, then this part is omitted.}}, \\[8mm]
\con{j}=L_{\lambda} \quad \text{for all $j \ge p+q$}.
\end{cases}
\end{equation*}
We set 
%
%%%%%%%%%%%%%%%%%
%%% eq:lam-p0 %%%
%%%%%%%%%%%%%%%%%
%
\begin{equation} \label{eq:lam-p0}
\lambda^{+}:=\sum_{j=0}^{q-1} \con{j}\eps_{j}+
L_{\lambda}(\eps_{q}+\eps_{q+1}+\cdots), \qquad
\lambda^{0}:=
  \sum_{j=q}^{q+p-1}(\con{j}-L_{\lambda})\eps_{j};
\end{equation}
note that $\lambda^{+} \in P_{+}$, $\lambda^{0} \in E$, 
and $\nu=\lambda^{+}+\lambda^{0}$. 
Theorem~\ref{thm:ED}, along with its proof, yields 
an isomorphism of $U_{q}(\Fg)$-crystals: 
\begin{equation*}
\BB(\lambda)=\BB(\nu) \cong 
 \BB(\lambda^{0}) \otimes \BB(\lambda^{+}).
\end{equation*}
\end{rem}

The following proposition will be used 
in the proof of Theorem~\ref{thm:general} below 
(cf. \cite[Lemma~5.1]{Kw-Adv} in type $A_{+\infty}$ and 
\cite[Proposition~3.12]{Kw-Ep} in type $A_{\infty}$). 
%
%%%%%%%%%%%%%%%%%%%%
%%% prop:ED-isom %%%
%%%%%%%%%%%%%%%%%%%%
%
\begin{prop} \label{prop:ED-isom}
Let $\lambda_{1},\,\lambda_{2} \in E$, and 
$\mu_{1},\,\mu_{2} \in P_{+}$. Then, 
\begin{equation*}
\BB(\lambda_{1}) \otimes \BB(\mu_{1}) \cong 
\BB(\lambda_{2}) \otimes \BB(\mu_{2}) \qquad 
\text{\rm as $U_{q}(\Fg)$-crystals}
\end{equation*}
if and only if $\lambda_{1} \in W\lambda_{2}$ and 
$\mu_{1}=\mu_{2}$. 
\end{prop}

\begin{proof}
The ``if'' part follows immediately from 
Corollary~\ref{cor:isom}\,(3). We show the ``only if'' part. 
As in the proof of Theorem~\ref{thm:ED}, we set 
$q_{m}:=\min \bigl\{j \in \BZ_{\ge 0} \mid 
 \co{\mu_{m}}{j}=L_{\mu_{m}}\bigr\}$ 
and $p_{m}:=\#\Supp(\lambda_{m})$ for $m=1,\,2$. 
Also, for $m=1,\,2$, let $\xi_{m} \in W\lambda_{m}$ 
be the unique element 
of $W\lambda_{m}$ such that 
\begin{equation*}
0 < \co{\xi_{m}}{q_{m}} \le \co{\xi_{m}}{q_{m}+1} \le \cdots \le 
    \co{\xi_{m}}{q_{m}+p_{m}-1}.
\end{equation*}
Then, the proof of Theorem~\ref{thm:ED} shows that 
\begin{equation*}
\BB(\lambda_{m}) \otimes \BB(\mu_{m}) \cong \BB(\xi_{m}+\mu_{m})
\quad \text{as $U_{q}(\Fg)$-crystals}
\end{equation*}
for $m=1,\,2$; observe that if we set $\nu_{m}:=\xi_{m}+\mu_{m}$ 
for $m=1,\,2$, then we have 
%
%%%%%%%%%%%%%%
%%% eq:num %%%
%%%%%%%%%%%%%%
%
\begin{equation} \label{eq:num}
\begin{cases}
0 \le 
\underbrace{\abs{\co{\nu_{m}}{0}} \le 
 \co{\nu_{m}}{1} \le \cdots \le \co{\nu_{m}}{q_{m}-1} <}_{%
\text{If $q_{m}=0$, then this part is omitted.}} L_{\mu_{m}} 
\underbrace{< \co{\nu_{m}}{q_{m}} \le 
 \co{\nu_{m}}{q_{m}+1} \le \cdots \le \co{\nu_{m}}{q_{m}+p_{m}-1}}_{%
\text{If $p_{m}=0$, then this part is omitted.}}, \\[8mm]
\co{\nu_{m}}{j}=L_{\mu_{m}} \quad \text{for all $j \ge p_{m}+q_{m}$}.
\end{cases}
\end{equation}
Because
\begin{equation*}
\BB(\nu_{1}) \cong \BB(\lambda_{1}) \otimes \BB(\mu_{1}) \cong 
\BB(\lambda_{2}) \otimes \BB(\mu_{2}) \cong \BB(\nu_{2})
\quad \text{as $U_{q}(\Fg)$-crystals}
\end{equation*}
by assumption, we infer by Corollary~\ref{cor:isom}\,(3) that 
$\nu_{1} \in W\nu_{2}$. However, by using 
\eqref{eq:binf-r}--\eqref{eq:dinf-r}, we see that 
$\nu_{2}$ is a unique element of $W\nu_{2}$ satisfying 
condition \eqref{eq:num}, and hence that $\nu_{1}=\nu_{2}$. 
From this equality, it is easily seen that $L_{\mu_{1}}=L_{\mu_{2}}$, 
and $q_{1}=q_{2}$, $p_{1}=p_{2}$. Therefore, it follows from the 
definitions of $\nu_{1}$ and $\nu_{2}$ that 
$\mu_{1}=\mu_{2}$, and hence $\xi_{1}=\xi_{2}$. 
Thus we have proved the proposition. 
\end{proof}

%==============================%
%     START SUBSECTION 0403    %
%==============================%
%
\subsection{The case $\lambda \in P_{+}$, $\mu \in E$.}
\label{subsec:DE}
Throughout this subsection, we fix $\lambda \in P_{+}$ and $\mu \in E$. 
We set $p:=\min\bigl\{j \in \BZ_{\ge 0} \mid \col{j}=L_{\lambda}\bigr\} 
\in \BZ_{\ge 0}$; it follows from Remark~\ref{rem:dom} that 
%
%%%%%%%%%%%%%%
%%% eq:col %%%
%%%%%%%%%%%%%%
%
\begin{equation} \label{eq:col}
0 \le \underbrace{\abs{\col{0}} \le \col{1} \le 
\cdots \le \col{p-1} <}_{\text{If $p=0$, then this part is omitted.}} 
L_{\lambda}=\col{p}=\col{p+1}=\cdots.
\end{equation}
Also, we set 
\begin{equation*}
q:=\min\bigl\{q \in \BZ_{\ge 1} \mid 
\text{$\Supp (\mu) \subset [q-1]$ and $q \ge |\mu|$}\bigr\},
\end{equation*}
where $|\mu|:=\sum_{j \in \BZ_{\ge 0}} |\com{j}| \in \BZ_{\ge 0}$; 
note that $q \ge \#\Supp (\mu)$. 
Recall from Remark~\ref{rem:lcm} 
that if $N=N_{\mu} \in \BZ_{\ge 1}$ denotes 
the least common multiple of nonzero integers in 
$\bigl\{\pair{\mu}{\beta^{\vee}} \mid 
\beta \in \Delta\bigr\} \cup \bigl\{1\bigr\}$, 
then each LS path $\eta \in \BB(\mu)$ of shape $\mu$ 
can be written as: 
\begin{equation*}
\eta=(\nu_{1},\,\nu_{2},\,\dots,\,\nu_{N})
\end{equation*}
for some $\nu_{1},\,\nu_{2},\,\dots,\,\nu_{N} \in W\mu$, 
with $\nu_{1} \ge \nu_{2} \ge \cdots \ge \nu_{N}$. 

Let $n \in \BZ_{\ge 0}$ be such that $n > p+(N+1)q$; 
note that $n \ge 3$. 
Let $\Pn$ denote the subset of $P$ consisting of 
all $[n]$-dominant integral weights $\nu \in P$ 
satisfying the conditions that 
\begin{equation*}
\begin{cases}
L_{\nu}=L_{\lambda}, \\[1mm]
\con{j}=L_{\nu} \, (=L_{\lambda}) 
\quad \text{for all $j \ge n+1$, and} \\[1mm]
\#\bigl\{
  0 \le j \le n \mid 
  \con{j}=L_{\nu} \, (=L_{\lambda})
\bigr\} > n-p-Nq;
\end{cases}
\end{equation*}
note that $n-p-Nq \ge q \ge |\mu|$. 

%%%%%%%%%%%%%%%%%%%%%%
%%% lem:DE-highest %%%
%%%%%%%%%%%%%%%%%%%%%%
%
\begin{lem} \label{lem:DE-highest}
Let $n \in \BZ_{\ge 0}$ be such that $n > p+(N+1)q$. 
If $\pi \otimes \eta \in 
\BBn(\lambda) \otimes \BBn(\mu)$ is an $[n]$-maximal element, 
then $\pi=\pi_{\lambda}$. Moreover, 
the weight $\wt (\pi \otimes \eta)$ of 
$\pi \otimes \eta$ is contained in $\Pn$. 
\end{lem}

\begin{proof}
For the first assertion, observe that 
if $e_{i}\pi \ne \bzero$ for some $i \in [n]$, then 
$e_{i}(\pi \otimes \eta) \ne \bzero$ 
by the tensor product rule for crystals. 
Now, since $\pi \otimes \eta$ is $[n]$-maximal by assumption, 
we have $e_{i}\pi=\bzero$ for all $i \in [n]$, 
which implies that $\pi=\pi_{\lamn}=\pi_{\lambda}$ 
(see Remark~\ref{rem:LS-n}), as desired.

We show that 
$\nu:=\wt (\pi \otimes \eta)=
 \wt (\pi_{\lambda} \otimes \eta)=\lambda+\wt \eta$ is 
contained in the set $\Pn$. 
Since $\pi \otimes \eta$ is $[n]$-maximal and 
$\BB(\lambda) \otimes \BB(\mu)$ is a normal $U_{q}(\Fg)$-crystal, 
it is clear that $\nu=\wt (\pi \otimes \eta)$ is $[n]$-dominant. 
Also, we see from Remark~\ref{rem:wt} that the level $L_{\nu}$ of 
$\nu=\lambda+\wt \eta$ is equal to 
$L_{\lambda}+0=L_{\lambda}$. 
We write $\eta \in \BBn(\mu)$ as: 
$\eta=(\nu_{1},\,\nu_{2},\,\dots,\,\nu_{N})$ 
for some $\nu_{1},\,\nu_{2},\,\dots,\,\nu_{N} \in W\mu$; 
note that $\nu_{M} \in \Wn\mu$ for all $1 \le M \le N$ 
by Lemma~\ref{lem:BBn}. 
Since $\Supp (\mu) \subset [q-1] \subset [n]$ 
by our assumptions, we have 
$\Supp (\nu_{M}) \subset [n]$ 
for all $1 \le M \le N$. 
Therefore, it follows from the equation 
\begin{equation*}
\nu=\lambda+\wt \eta=
\lambda+\frac{1}{N}\sum_{M=1}^{N} \nu_{M}
\end{equation*}
that $\con{j}=L_{\lambda}$ for all $j \ge n+1$ 
(see Remark~\ref{rem:wt}). Furthermore, since $\#\Supp (\nu_{M})=
\#\Supp (\mu) \le q$ for all $1 \le M \le N$, and 
since 
\begin{equation*}
\Supp (\wt \eta) \subset 
\bigcup_{1 \le M \le N} \Supp (\nu_{M}),
\end{equation*}
we have
\begin{align*}
& 
\#\bigl\{0 \le j \le n \mid 
 \nu^{(j)}=L_{\nu} \, (=L_{\lambda})\bigr\} 
\ge 
\#\bigl\{p \le j \le n \mid 
 \nu^{(j)}=L_{\lambda}\bigr\} \\
& \qquad 
\ge (n-p+1)-\#\Supp (\wt \eta)
\ge (n-p+1)-\sum_{M=1}^{N} \#\Supp(\nu_{M}) \\[1mm]
& \qquad 
\ge (n-p+1)-Nq > n-p-Nq.
\end{align*}
Thus, we have shown that $\nu=\wt (\pi \otimes \eta) \in \Pn$, 
thereby proving the lemma.
\end{proof}

Let $n \in \BZ_{\ge 0}$ be such that $n > p+(N+1)q$.
We deduce from the definition of $\Pn$ that 
an element $\nu \in \Pn$ satisfies the conditions
%
%%%%%%%%%%%%%%
%%% eq:nun %%%
%%%%%%%%%%%%%%
%
\begin{equation} \label{eq:nun}
\begin{cases}
0 \le \abs{\con{0}} \le \con{1} \le \cdots \le \con{u_{0}-1} < L_{\lambda}, \\[1.5mm]
\con{u_{0}}= \con{u_{0}+1}= \cdots = \con{u_{1}-1}=L_{\lambda}, \\[1.5mm]
L_{\lambda} < \con{u_{1}} \le \con{u_{1}+1} \le \cdots \le \con{n}, \\[1.5mm]
\con{j}=L_{\lambda} \quad \text{for all $j \ge n+1$}
\end{cases}
\end{equation}
for some $0 \le u_{0} < u_{1} \le n+1$; it follows from the 
definition of $\Pn$ that $u_{1}-u_{0} > n-p-Nq > q > 0$. 
Here, if $u_{0}=0$ (resp., $u_{1}=n+1$), then 
the fist condition (resp., the third condition) is dropped. 
For the element $\nu \in \Pn$ above, we define $\xi \in P$ by: 
%
%%%%%%%%%%%%%%
%%% eq:nuo %%%
%%%%%%%%%%%%%%
%
\begin{equation} \label{eq:nuo}
\begin{cases}
\cox{j}=\con{j} \quad 
\text{\rm for $0 \le j \le u_{0}-1$}, \\[1.5mm]
\cox{j}=\con{j}=L_{\lambda} \quad 
\text{\rm for $u_{0} \le j \le u_{1}-1$}, \\[1.5mm]
\cox{u_{1}}=L_{\lambda}, \\[1.5mm]
\cox{j}=\con{j-1} \quad 
\text{\rm for $u_{1}+1 \le j \le n+1$}, \\[1.5mm]
\cox{j}=\con{j}=L_{\lambda} \quad 
\text{\rm for $j \ge n+2$},
\end{cases}
\end{equation}
where $u_{0}$ and $u_{1}$ are as in \eqref{eq:nun}. 
It is easily verified that $\xi \in \Po$. 
%
%%%%%%%%%%%%%%
%%% lem:Pn %%%
%%%%%%%%%%%%%%
%
\begin{lem} \label{lem:Pn}
Keep the setting above. We have $\xi=\nuo$. 
Moreover, if we define a map $\Theta_{n}:\Pn \rightarrow \Po$ by 
$\Theta_{n}(\nu):=\nuo$ for $\nu \in \Pn$, 
then the map $\Theta_{n}$ is bijective. 
\end{lem}

\begin{proof}
We set $w:=r_{u_{1}+1}r_{u_{1}+2} \cdots r_{n}r_{n+1} \in \Wo$, 
where $u_{1}$ is as in \eqref{eq:nun}. 
Then, by using \eqref{eq:binf-r}--\eqref{eq:dinf-r}, we find that 
the $\xi$ given by \eqref{eq:nuo} is identical to the element 
$w\nu \in \Wo\nu$. Since $\xi$ is $[n+1]$-dominant 
by the definition \eqref{eq:nuo} of $\xi$, 
we conclude that $\xi=\nuo$. Thus we have shown that 
if $\nu \in \Pn$, then $\Theta_{n}(\nu)=\nuo \in \Po$. 
Now, the bijectivity of the map $\Theta_{n}$ easily follows 
by examining \eqref{eq:nun} and \eqref{eq:nuo}. 
This proves the lemma.
\end{proof}

Let $n \in \BZ_{\ge 0}$ be such that $n > p+(N+1)q$. 
Let $\pi \otimes \eta$ be an $[n+1]$-maximal element of 
$\BBo(\lambda) \otimes \BBo(\mu)$. 
Then, we have $\wt (\pi \otimes \eta) \in \Po$ 
by Lemma~\ref{lem:DE-highest}, and hence 
$\wt(\pi \otimes \eta)=\nuo$ for some $\nu \in \Pn$ 
by Lemma~\ref{lem:Pn}. 
%
%%%%%%%%%%%%%%%%%%%%%%%
%%% prop:DE-highest %%%
%%%%%%%%%%%%%%%%%%%%%%%
%
\begin{prop} \label{prop:DE-highest}
Keep the setting above. Let $x \in \Wo$ be such that $x\nuo=\nu$. 
Then, $S_{x}(\pi \otimes \eta)$ is an $[n]$-maximal element 
of weight $\nu$ contained in $\BBn(\lambda) \otimes \BBn(\mu)$. 
\end{prop}

\begin{proof}
We want to prove that 
there exists $w \in \Wo$ such that 
$S_{w}(\pi \otimes \eta)$ is an $[n]$-maximal element 
contained in $\BBn(\lambda) \otimes \BBn(\mu)$. 
First of all, we see by 
Lemma~\ref{lem:DE-highest} that $\pi=\pi_{\lambda}$. 
We write $\eta \in \BBo(\mu)$ as: 
$\eta=(\nu_{1},\,\nu_{2},\,\dots,\,\nu_{N})$ 
for some $\nu_{1},\,\nu_{2},\,\dots,\,\nu_{N} \in W\mu$, 
with $\nu_{1} \ge \nu_{2} \ge \cdots \ge \nu_{N}$. 
By the same reasoning as in the proof of Lemma~\ref{lem:DE-highest}, 
we see that $\nu_{M} \in \Wo\mu$ and 
$\Supp (\nu_{M}) \subset [n+1]$ for all $1 \le M \le N$.
Furthermore, 
since $\pair{\lambda}{h_{i}}=0$ for all $i \ge p+1$ by \eqref{eq:col}, 
and since $\pi \otimes \eta=\pi_{\lambda} \otimes \eta$ is 
$[n+1]$-maximal by assumption, it follows from 
the tensor product rule for crystals that 
$e_{i}\eta=\bzero$ for all $p+1 \le i \le n+1$. 
Consequently, by the definition of 
the Kashiwara operators $e_{i}$, $p+1 \le i \le n+1$, 
we must have $m^{\eta}_{i}=\min \bigl\{\pair{\eta(t)}{h_{i}} \mid 
t \in [0,\,1]_{\BR}\bigr\}=0$ for all $p+1 \le i \le n+1$. 
Therefore, it follows that $\pair{\nu_{1}}{h_{i}} \ge 0$ 
for all $p+1 \le i \le n+1$.  
From this fact, using \eqref{eq:simple}, 
we deduce that 
%
%%%%%%%%%%%%%%
%%% eq:con %%%
%%%%%%%%%%%%%%
%
\begin{equation} \label{eq:con}
\begin{cases}
\co{\nu_{1}}{p} \le \co{\nu_{1}}{p+1} \le \cdots \le \co{\nu_{1}}{s_{0}} < 0, \\[1.5mm]
\co{\nu_{1}}{s_{0}+1}= \co{\nu_{1}}{s_{0}+2}= \cdots = \co{\nu_{1}}{s_{1}}=0, \\[1.5mm]
0 < \co{\nu_{1}}{s_{1}+1} \le \co{\nu_{1}}{s_{1}+2} \le \cdots \le \co{\nu_{1}}{n+1}
\end{cases}
\end{equation}
for some $p-1 \le s_{0} \le s_{1} \le n+1$. 
Here, since $\#\Supp(\nu_{1})=\#\Supp(\mu) \le q$, 
the number $(s_{0}-p+1)+(n-s_{1}+1)$ of nonzero elements 
in \eqref{eq:con} is less that or equal to $q$, and hence 
we have
%
%%%%%%%%%%%%%%%
%%% eq:u1u0 %%%
%%%%%%%%%%%%%%%
%
\begin{align}
s_{1}-s_{0} & =(n-p+2)-\bigl\{(s_{0}-p+1)+(n-s_{1}+1)\bigr\} \nonumber \\ 
& > (n-p+2)-q \ge p+(N+1)q-p+2-q = Nq+2. \label{eq:u1u0}
\end{align}
Also, by using \eqref{eq:wtN}, we obtain 
%
%%%%%%%%%%%%%%%
%%% eq:supp %%%
%%%%%%%%%%%%%%%
%
\begin{equation} \label{eq:supp}
\Supp(\wt \eta) 
 \subset \bigcup_{1 \le M \le N} \Supp(\nu_{M}) 
 \subset [n+1], 
\end{equation}
and hence
\begin{equation*}
\#\Supp(\wt \eta) \le \sum_{M=1}^{N} \#\Supp(\nu_{M}) =
\sum_{M=1}^{N} \#\Supp(\mu) \le \sum_{M=1}^{N} q = Nq. 
\end{equation*}
From this inequality and \eqref{eq:u1u0}, 
we conclude that there exists $s_{0}+1 < s \le s_{1}$ 
such that $s \notin \Supp(\wt \eta)$. We set 
\begin{equation*}
w_{1}:=r_{n+1}r_{n} \cdots r_{s+2}r_{s+1} \in \Wo,
\end{equation*}
and $\eta_{1}:=S_{w_{1}}\eta \in \BBo(\mu)$; 
note that $s+1 > s_{0}+2 \ge p+1$. 
Since $\pair{\lambda}{h_{i}}=0$ for all $i \ge p+1$, 
by the tensor product rule for crystals, there follows 
\begin{equation*}
S_{w_{1}}(\pi \otimes \eta)=
S_{w_{1}}(\pi_{\lambda} \otimes \eta)=
 \pi_{\lambda} \otimes (S_{w_{1}}\eta)=
 \pi_{\lambda} \otimes \eta_{1}.
\end{equation*}

We claim that $\eta_{1} \in \BBn(\mu)$. Write it as: 
$\eta_{1}=(\xi_{1},\,\xi_{2},\,\dots,\,\xi_{N})$ 
for some $\xi_{1},\,\xi_{2},\,\dots,\,\xi_{N} \in W\mu$, 
with $\xi_{1} \ge \xi_{2} \ge \cdots \ge \xi_{N}$. 
Then we have 
$\co{\xi_{M}}{n+1} \ge \co{\xi_{1}}{n+1} \ge 0$ 
for all $1 \le M \le N$. Indeed, 
it follows from the definitions of 
the Kashiwara operators $e_{i}$ and $f_{i}$ for $i \in [n+1]$ 
(see also \cite[Proposition~4.7]{Lit-Ann}) that 
$\xi_{1}=z\nu_{1}$ for some subword $z$ of 
$w_{1}=r_{n+1}r_{n} \cdots r_{s+2}r_{s+1}$. 
Since $s_{0}+1 < s \le s_{1}$, we infer from \eqref{eq:con}, 
by using \eqref{eq:binf-r}--\eqref{eq:dinf-r}, that 
$\co{\xi_{1}}{n+1} \ge 0$. 
Now, by the same reasoning as in the proof of 
Lemma~\ref{lem:DE-highest}, we see that 
$\Supp (\xi_{M}) \subset [n+1]$ for all $1 \le M \le N$, 
and hence $\Supp(\xi_{M}-\xi_{1}) \subset [n+1]$. 
Also, since $\xi_{1} \ge \xi_{M}$, we have 
$\xi_{M}-\xi_{1} \in \sum_{i \in I}\BZ_{\ge 0}\alpha_{i} \subset E$ 
by Remark~\ref{rem:Q+}. 
Combining these facts, we deduce that 
$\xi_{M}-\xi_{1} \in \sum_{i=0}^{n+1}\BZ_{\ge 0}\alpha_{i}$, 
since $\alpha_{i}=-\eps_{i-1}+\eps_{i}$ for $i \ge 1$ 
(see \eqref{eq:simple});  observe that by \eqref{eq:simple}, 
$\xi \in \sum_{i=0}^{n+1}\BZ_{\ge 0}\alpha_{i}$ implies that 
$\cox{n+1} \ge 0$, since $\co{\alpha_{n+1}}{n+1}=1 > 0$ and 
$\co{\alpha_{i}}{n+1}=0$ for all $0 \le i \le n$. 
Consequently, we conclude that 
$\co{(\xi_{M}-\xi_{1})}{n+1} \ge 0$, and hence 
$\co{\xi_{M}}{n+1} \ge \co{\xi_{1}}{n+1} \ge 0$, as desired. 
Furthermore, since $s \notin \Supp(\wt \eta)$, 
and $\Supp(\wt \eta) \subset [n+1]$ by \eqref{eq:supp}, 
we can easily verify by using \eqref{eq:binf-r}--\eqref{eq:dinf-r} 
that 
\begin{equation*}
\Supp(\wt \eta_{1})=
\Supp(w_{1}(\wt \eta)) \subset [n].
\end{equation*}
Therefore, by taking \eqref{eq:wtN} into consideration, 
we obtain $\co{\xi_{M}}{n+1}=0$ for all $1 \le M \le N$, 
and hence $\Supp (\xi_{M}) \subset [n]$ for all $1 \le M \le N$. 
In addition, since $\xi_{M} \in W\mu$ for $1 \le M \le N$, 
we see through use of \eqref{eq:binf-r}--\eqref{eq:dinf-r} that 
$\#\Supp (\xi_{M})=\#\Supp(\mu) \le q < n$, and hence 
$\Supp (\xi_{M}) \subsetneqq [n]$ for all $1 \le M \le N$; 
recall that $\Supp (\mu) \subset [q-1] \subsetneqq [n]$ 
by the definition of $q$. Hence we can apply 
Lemma~\ref{lem:lamn} to $\xi_{M}$ for $1 \le M \le N$, 
and also to $\mu$. Then, noting that 
$\xi_{M} \in W\mu$ for $1 \le M \le N$, 
we find through use of \eqref{eq:binf-r}--\eqref{eq:dinf-r}
that the unique $[n]$-dominant element of $\Wn\xi_{M}$ is identical to 
the unique $[n]$-dominant element of $\Wn\mu$ for all $1 \le M \le N$. 
As a consequence, 
we conclude that $\xi_{M} \in \Wn\mu$ for all $1 \le M \le N$, 
which implies that $\eta_{1} \in \BBn(\mu)$ by Lemma~\ref{lem:BBn}. 
Thus we have shown that $S_{w_{1}}(\pi \otimes \eta)=
\pi_{\lambda} \otimes \eta_{1} \in 
\BBn(\lambda) \otimes \BBn(\mu)$. 

Because $\BB(\lambda) \otimes \BB(\mu)$ is 
a normal $U_{q}(\Fg)$-crystal, and 
because $\pi \otimes \eta=\pi_{\lambda} \otimes \eta$ 
is $[n+1]$-maximal by assumption, 
it follows from Remark~\ref{rem:ext}\,(2) that 
$\pi \otimes \eta=\pi_{\lambda} \otimes \eta$ is $[n+1]$-extremal, 
and hence so is $\pi_{\lambda} \otimes \eta_{1}=
 S_{w_{1}}(\pi \otimes \eta)$
(recall that $w_{1} \in \Wo$). If we take $w_{2} \in \Wn$ such that 
$w_{2}(\wt(\pi_{\lambda} \otimes \eta_{1}))$ is 
$[n]$-dominant, then we see from \eqref{eq:extremal} that 
the element $S_{w_{2}}(\pi_{\lambda} \otimes \eta_{1}) \in 
\BBn(\lambda) \otimes \BBn(\mu)$ is 
$[n]$-maximal since it is $[n]$-extremal. 
Now we set $w=w_{2}w_{1} \in \Wo$. Then the element 
$S_{w}(\pi \otimes \eta)=
 S_{w_{2}}(\pi_{\lambda} \otimes \eta_{1})$ is 
an $[n]$-maximal element contained in 
$\BBn(\lambda) \otimes \BBn(\mu)$. 
This proves the assertion at the very beginning of our proof. 

It remains to show that $S_{w}(\pi \otimes \eta)$ is 
identical to $S_{x}(\pi \otimes \eta)$. 
If we set $\xi:=\wt(S_{w}(\pi \otimes \eta))$, 
then $\xi \in \Pn$ by Lemma~\ref{lem:DE-highest}. 
Since $\nuo=\wt(\pi \otimes \eta)=
w^{-1}\xi \in \Wo\xi$, we have $\xio=\nuo$. 
Therefore, by the bijectivity of the map 
$\Theta_{n}:\Pn \rightarrow \Po$ (see Lemma~\ref{lem:Pn}), 
we obtain $\xi=\nu$, and hence $w\nuo=\nu$.
Since $\nuo$ is $[n+1]$-dominant, and 
since $x,\,w$ are elements of $\Wo$ such that $w\nuo=\nu=x\nuo$, 
there exists $u \in \Wo$ such that $u\nuo=\nuo$ and $w=xu$; 
note that $u$ is equal to a product of the $r_{i}$'s 
for $i \in [n+1]$ such that $\pair{\nuo}{h_{i}}=0$. 
Consequently, by using \eqref{eq:si}, we obtain 
\begin{equation*}
S_{w}(\pi \otimes \eta)=
S_{xu}(\pi \otimes \eta)=
S_{x}S_{u}(\pi \otimes \eta)=
S_{x}(\pi \otimes \eta).
\end{equation*}
This completes the proof of the proposition. 
\end{proof}

Let $n \in \BZ_{\ge 0}$ be such that $n > p+(N+1)q$. 
Let $\BBhi{n}$ denote the subset of 
$\BBn(\lambda) \otimes \BBn(\mu)$ consisting of 
all $[n]$-maximal elements, and set 
$\BBhiw{n}{\nu}:=\BBhi{n} \cap 
\bigl(\BBn(\lambda) \otimes \BBn(\mu)\bigr)_{\nu}$ 
for $\nu \in \Pn$. 
Then, by Lemma~\ref{lem:DE-highest}, 
\begin{equation*}
\BBhi{n}=
\bigsqcup_{\nu \in \Pn} \BBhiw{n}{\nu}. 
\end{equation*}
Also, we see by Lemmas~\ref{lem:DE-highest} 
and \ref{lem:Pn} that 
\begin{equation*}
\BBhi{n+1}=
\bigsqcup_{\xi \in \Po} \BBhiw{n+1}{\xi}=
\bigsqcup_{\nu \in \Pn} \BBhiw{n+1}{\nuo}. 
\end{equation*}
Let $\nu \in \Pn$, and let $x \in \Wo$ be 
such that $x\nuo=\nu$. By Proposition~\ref{prop:DE-highest}, 
we obtain an injective map $S_{x}$ from $\BBhiw{n+1}{\nuo}$ 
into $\BBhiw{n}{\nu}$;
\begin{equation*}
S_{x}:\BBhiw{n+1}{\nuo} \hookrightarrow \BBhiw{n}{\nu}.
\end{equation*}
%
%%%%%%%%%%%%%%%%%%%
%%% prop:DE-bij %%%
%%%%%%%%%%%%%%%%%%%
%
\begin{prop} \label{prop:DE-bij}
Keep the setting above. 
The map $S_{x}:\BBhiw{n+1}{\nuo} \rightarrow \BBhiw{n}{\nu}$ 
is bijective. 
\end{prop}

The proof of this proposition is given in \S\ref{sec:DE-bij}. 
In the rest of this subsection, we take and fix an arbitrary 
$m \in \BZ_{\ge 0}$ such that $m > p+(N+1)q$. 
%
%%%%%%%%%%%%%%%%%%
%%% cor:DE-bij %%%
%%%%%%%%%%%%%%%%%%
%
\begin{cor} \label{cor:DE-bij}
Every element $\pi \otimes \eta$ of $\BBhi{m}$ is extremal.
\end{cor}

\begin{proof}
We show that $\pi \otimes \eta$ is $[n]$-extremal 
for all $n \ge m$ (see Remark~\ref{rem:ext}\,(1)). 
We set $\nu:=\wt (\pi \otimes \eta) \in \Pm{m}$. 
Since $\num{m}=\nu$ and 
$\nuo=(\nun)_{[n+1]}$ for all $n \ge m$ by the definitions, 
we can easily show by induction on $n$, 
using Lemma~\ref{lem:Pn}, that $\nun \in \Pn$ for all $n \ge m$. 
For each $n \ge m$, we take $x_{n} \in \Wo$ such that 
$x_{n}(\nuo)=x_{n}((\nun)_{[n+1]})=\nun$. 
Since $\nun \in \Pn$ as seen above, 
Proposition~\ref{prop:DE-bij} asserts that the map
\begin{equation*}
S_{x_{n}}: 
\BBhiw{n+1}{\nuo}=
\BBhiw{n+1}{(\nun)_{[n+1]}} \rightarrow \BBhiw{n}{\nun}
\end{equation*} 
is bijective. 

Now, fix $n \in \BZ_{\ge 0}$ such that $n \ge m$, 
and set $y_{n}:=x_{m}x_{m+1} \cdots x_{n-2}x_{n-1}$; 
note that $y_{n} \in \Wn$. Then, the argument above shows that 
the composite $S_{y_{n}}=S_{x_{m}}S_{x_{m+1}} \cdots 
S_{x_{n-2}}S_{x_{n-1}}$ yields a bijective map 
from $\BBhiw{n}{\nun}$ onto $\BBhiw{m}{\nu}$ as follows: 
\begin{equation*}
\BBhiw{n}{\nun} 
\stackrel{S_{x_{n-1}}}{\longrightarrow}
\BBhiw{n-1}{\num{n-1}}
\stackrel{S_{x_{n-2}}}{\longrightarrow}
\cdots 
\stackrel{S_{x_{m+1}}}{\longrightarrow} 
\BBhiw{m+1}{\num{m+1}}
\stackrel{S_{x_{m}}}{\longrightarrow} 
\BBhiw{m}{\num{m}}=\BBhiw{m}{\nu}.
\end{equation*}
Consequently, the element 
$S_{y_{n}^{-1}}(\pi \otimes \eta)$ is 
contained in $\BBhiw{n}{\nun}$, and hence 
is an $[n]$-maximal element. 
Because $\BB(\lambda) \otimes \BB(\mu)$ 
is a normal $U_{q}(\Fg)$-crystal, 
it follows from Remark~\ref{rem:ext}\,(2) that 
$S_{y_{n}^{-1}}(\pi \otimes \eta)$ 
is $[n]$-extremal, and hence so is $\pi \otimes \eta$. 
This proves the corollary. 
\end{proof}
%
%%%%%%%%%%%%%%%%%%%%
%%% prop:DE-conn %%%
%%%%%%%%%%%%%%%%%%%%
%
\begin{prop} \label{prop:DE-conn}
Each connected component of $\BB(\lambda) \otimes \BB(\mu)$ 
contains a unique element of $\BBhi{m}$. 
\end{prop}

\begin{proof}
Let $\pi \otimes \eta \in \BB(\lambda) \otimes \BB(\mu)$, 
and take $n \in \BZ_{\ge 0}$, with $n \ge m$, such that 
$\pi \otimes \eta \in \BBn(\lambda) \otimes \BBn(\mu)$ 
(see \eqref{eq:limBBn}). From Remark~\ref{rem:LS-n}, 
we see that each connected component of 
$\BBn(\lambda) \otimes \BBn(\mu)$ is isomorphic, as a $U_{q}(\Fgn)$-crystal, 
to the crystal basis of a finite-dimensional irreducible $U_{q}(\Fgn)$-module. 
Therefore, there exists a monomial $X$ 
in the Kashiwara operators $e_{i}$ for $i \in [n]$ 
such that $X(\pi \otimes \eta) \in \BBhi{n}$. 
Let $\xi \in P$ be the weight of $X(\pi \otimes \eta)$; 
note that $\xi \in \Pn$ by Lemma~\ref{lem:DE-highest}. We set 
\begin{equation*}
\nu:=
\Theta_{m}^{-1}\Theta_{m+1}^{-1} \cdots \Theta_{n-2}^{-1}
\Theta_{n-1}^{-1}(\xi) \in \Pm{m}; 
\end{equation*}
it is clear that $\xi=\nun$ by Lemma~\ref{lem:Pn}.
Then, the argument in the proof of 
Corollary~\ref{cor:DE-bij} shows that there exists $y \in \Wn$ such that 
$S_{y^{-1}}$ yields a bijective map from $\BBhiw{n}{\xi}=\BBhiw{n}{\nun}$ 
onto $\BBhiw{m}{\nu}$. In particular, we have 
$S_{y^{-1}}X(\pi \otimes \eta) \in \BBhiw{m}{\nu} 
\subset \BBhi{m}$. Also, observe that since $y \in \Wn$, the element 
$S_{y^{-1}}X(\pi \otimes \eta)$ lies in the connected component of 
$\BB(\lambda) \otimes \BB(\mu)$ containing $\pi \otimes \eta$. 
Thus, we have proved that each connected 
component of $\BB(\lambda) \otimes \BB(\mu)$ contains 
an element of $\BBhi{m}$. 

It remains to prove the uniqueness assertion. 
Suppose that $b_{1},\,b_{2} \in \BBhi{m}$ are contained 
in the same connected component of 
$\BB(\lambda) \otimes \BB(\mu)$. 
We set $\nu:=\wt b_{1} \in \Pm{m}$ and 
$\xi:=\wt b_{2} \in \Pm{m}$. 
Since $b_{1}$ and $b_{2}$ are both extremal by Corollary~\ref{cor:DE-bij}, 
we deduce from Proposition~\ref{prop:isom} that 
the connected component containing both $b_{1}$ and $b_{2}$ 
is isomorphic to $\BB(\nu)$, and also to $\BB(\xi)$. 
Consequently, by Corollary~\ref{cor:isom}\,(3), 
we obtain $\nu \in W\xi$. 
If we take $n \in \BZ_{\ge 0}$, with $n \ge m$, such that 
$\nu \in \Wn\xi$, then we have $\nun=\xin$, 
which implies that $\nu=\xi$ by Lemma~\ref{lem:Pn}.
Thus, $b_{1}$ and $b_{2}$ are both elements of weight $\nu$ 
in a connected component of $\BB(\lambda) \otimes \BB(\mu)$ 
isomorphic to $\BB(\nu)$. 
Therefore, by Corollary~\ref{cor:isom}\,(2), 
we conclude that $b_{1}=b_{2}$, as desired. 
This proves the proposition. 
\end{proof}

The following is the main result of this subsection 
(cf. \cite[Corollary~7.3]{Kw-Adv} in type $A_{+\infty}$ and 
\cite[Proposition~5.13]{Kw-Ep} in type $A_{\infty}$). 
%
%%%%%%%%%%%%%%
%%% thm:DE %%%
%%%%%%%%%%%%%%
%
\begin{thm} \label{thm:DE}
Let $\lambda \in P_{+}$ and $\mu \in E$. 
We take $m \in \BZ_{\ge 0}$ such that $m > p+(N+1)q$ as above.
Then, we have the following decomposition 
into connected components\,{\rm:}
\begin{equation*}
\BB(\lambda) \otimes \BB(\mu)=
\bigoplus_{\nu \in \Pm{m}} 
\BB(\nu)^{\oplus m_{\lambda,\mu}^{\nu}}, 
\end{equation*}
where for each $\nu \in \Pm{m}$, 
the multiplicity $m_{\lambda,\,\mu}^{\nu}$ is equal to 
$\#\BBhiw{m}{\nu}$. 
\end{thm}
%
%%%%%%%%%%%%%%
%%% rem:DE %%%
%%%%%%%%%%%%%%
%
\begin{rem} \label{rem:DE}
Recall from Remark~\ref{rem:Levi}
that if $\Fg$ is of type $B_{\infty}$ 
(resp., $C_{\infty}$, $D_{\infty}$), then 
$\Fgm{m}$ is a ``reductive'' Lie algebra 
of type $B_{m+1}$ (resp., $C_{m+1}$, $D_{m+1}$); 
note that $m > p+(N+1)q \ge 2$. 
Furthermore, we know from Remark~\ref{rem:LS-n} that 
$\BBm{m}(\lambda)$ (resp., $\BBm{m}(\mu)$) is isomorphic, 
as a $U_{q}(\Fgm{m})$-crystal, to 
the crystal basis of the finite-dimensional irreducible 
$U_{q}(\Fgm{m})$-module $\Vm{m}(\lamm{m})=\Vm{m}(\lambda)$ 
(resp., $\Vm{m}(\mum{m})$) of highest weight 
$\lamm{m}=\lambda$ (resp., $\mum{m}$).
Therefore, for each $\nu \in \Pm{m}$, 
the number $\#\BBhiw{m}{\nu}$ is equal to 
the multiplicity $[\Vm{m}(\lambda) \otimes \Vm{m}(\mum{m}):
\Vm{m}(\nu)]$ of the finite-dimensional irreducible 
$U_{q}(\Fgm{m})$-module $\Vm{m}(\nu)$
of highest weight $\nu$ 
in the tensor product $U_{q}(\Fgm{m})$-module 
$\Vm{m}(\lambda) \otimes \Vm{m}(\mum{m})$. Thus, 
\begin{equation*}
m_{\lambda,\,\mu}^{\nu}=\#\BBhiw{m}{\nu}=
[\Vm{m}(\lambda) \otimes \Vm{m}(\mum{m}):
\Vm{m}(\nu)]
\end{equation*}
for each $\nu \in \Pm{m}$. 
In particular, the number of those elements 
$\nu \in \Pm{m}$ for which $m_{\lambda,\,\mu}^{\nu} \ne 0$ is finite, 
and hence the total number of connected components of 
$\BB(\lambda) \otimes \BB(\mu)$ is finite; see \S\ref{subsec:key} 
for an explicit description of the number $\#\BBhiw{m}{\nu}$ 
in terms of Littlewood-Richardson coefficients. 
\end{rem}

\begin{proof}[Proof of Theorem~\ref{thm:DE}]
Let $\BB$ be an arbitrary connected component of 
$\BB(\lambda) \otimes \BB(\mu)$. 
Then, we know from Proposition~\ref{prop:DE-conn} that 
there exists a unique element $\pi \otimes \eta$ of $\BB$ 
that is contained in $\BBhi{m}$; we set 
$\nu:=\wt (\pi \otimes \eta) \in \Pm{m}$. 
Since the element $\pi \otimes \eta \in \BB \cap \BBhi{m}$ 
is extremal by Corollary~\ref{cor:DE-bij}, we see from 
Proposition~\ref{prop:isom} that $\BB$ is isomorphic, 
as a $U_{q}(\Fg)$-crystal, to $\BB(\nu)$. 

Proposition~\ref{prop:DE-conn} and 
Corollary~\ref{cor:DE-bij}, 
together with Proposition~\ref{prop:isom}, 
show that for each $\nu \in \Pm{m}$, 
there exists a one-to-one correspondence between 
the set of connected components of $\BB(\lambda) \otimes \BB(\mu)$ 
isomorphic to $\BB(\nu)$, and 
the subset $\BBhiw{m}{\nu}$ of $\BBhi{m}$ 
consisting of all elements of weight $\nu$. 
This implies immediately that 
$m_{\lambda,\,\mu}^{\nu}=\#\BBhiw{m}{\nu}$. 
Thus we have proved the theorem.
\end{proof}

%==============================%
%     START SUBSECTION 0404    %
%==============================%
%
\subsection{The case $\lambda,\,\mu \in P_{+}$.}
\label{subsec:DD}
Let $\lambda,\,\mu \in P_{+}$. 
In this case, we see from 
Theorem~\ref{thm:isom} and Remark~\ref{rem:extmod}\,(1) that
$\BB(\lambda)$ (resp., $\BB(\mu)$) is isomorphic, as a $U_{q}(\Fg)$-crystal, 
to the crystal basis of the irreducible highest weight $U_{q}(\Fg)$-module
of highest weight $\lambda$ (resp., $\mu$). 
For $\eta \in \BB(\mu)$, we say that 
$\lambda+\eta$ is dominant (resp., $[n]$-dominant for $n \in \BZ_{\ge 0}$) 
if $\pair{\lambda+\eta(t)}{h_{i}} \ge 0$ 
for all $t \in [0,\,1]_{\BR}$ and $i \in I$ (resp., $i \in [n]$); 
remark that such an element $\eta \in \BB(\mu)$ is said to be 
``$\lambda$-dominant'' (resp., ``$\lambda$-dominant'' 
with respect to $\Fgn$) in the terminology in \cite{Lit-Inv}. 
It is easy to verify that 
$\lambda+\eta$ is dominant (resp., $[n]$-dominant) if and only if 
$\pi_{\lambda} \otimes \eta \in \BB(\lambda) \otimes \BB(\mu)$ 
is maximal (resp., $[n]$-maximal). 

We have the following (see \cite[Proposition~2.3.2]{Lec-TAMS}; 
 cf. \cite[Proposition~4.10]{Kw-Ep} in type $A_{\infty}$). 
%
%%%%%%%%%%%%%%
%%% thm:DD %%%
%%%%%%%%%%%%%%
%
\begin{thm} \label{thm:DD}
Let $\lambda,\mu \in P_{+}$. 

{\rm (1)} We have the following decomposition 
into connected components\,{\rm:}
%
%%%%%%%%%%%%%%
%%% eq:DD1 %%%
%%%%%%%%%%%%%%
%
\begin{equation} \label{eq:DD1}
\BB(\lambda) \otimes \BB(\mu) =
 \bigoplus_{
  \begin{subarray}{c}
  \eta \in \BB(\mu), \\[0.5mm]
  \text{\rm $\lambda+\eta$ is dominant}
  \end{subarray}
 } \BB(\lambda+\wt(\eta)).
\end{equation}
In particular, each connected component of $\BB(\lambda) \otimes \BB(\mu)$ 
is isomorphic, as a $U_{q}(\Fg)$-crystal, to $\BB(\nu)$ 
for some $\nu \in P_{+}$.

{\rm (2)} Let $\nu \in P_{+}$. 
If $L_{\nu} \ne L_{\lambda}+L_{\mu}$, then 
the multiplicity $m_{\lambda,\,\mu}^{\nu}$ of 
$\BB(\nu)$ in the decomposition \eqref{eq:DD1} 
is equal to $0$. If $L_{\nu}=L_{\lambda}+L_{\mu}$, 
then the multiplicity $m_{\lambda,\,\mu}^{\nu}$ of 
$\BB(\nu)$ in the decomposition \eqref{eq:DD1}
is equal to the number 
%
%%%%%%%%%%%%%%
%%% eq:DD2 %%%
%%%%%%%%%%%%%%
%
\begin{equation} \label{eq:DD2}
\#\bigl\{
  \eta \in \BBn(\mu) \mid 
  \text{\rm $\lambda+\eta$ is $[n]$-dominant, and $\wt(\pi_{\lambda} \otimes \eta)=\nu$}
 \bigr\} 
\end{equation}
for an arbitrary $n \in \BZ_{\ge 3}$ such that 
$\col{j}=L_{\lambda}$, $\com{j}=L_{\mu}$, and 
$\con{j}=L_{\nu}=L_{\lambda}+L_{\mu}$ for all $j \ge n$.
\end{thm}

\begin{rem}
Recall from Remark~\ref{rem:Levi} that if $\Fg$ is of type $B_{\infty}$ 
(resp., $C_{\infty}$, $D_{\infty}$), then 
$\Fgn$ is a ``reductive'' Lie algebra of 
type $B_{n+1}$ (resp., $C_{n+1}$, $D_{n+1}$). 
Furthermore, we know from Remark~\ref{rem:LS-n} 
that $\BBn(\lambda)$ (resp., $\BBn(\mu)$) is 
isomorphic, as a $U_{q}(\Fgn)$-crystal, to 
the crystal basis of the finite-dimensional irreducible 
$U_{q}(\Fgn)$-module of highest weight $\lamn=\lambda$ (resp., $\mun=\mu$). 
Therefore, it follows from part (2) of Theorem~\ref{thm:DD}, 
together with the result in \cite[\S10]{Lit-Ann}, that 
for each $\nu=\nun \in P_{+}$ such that $L_{\nu}=L_{\lambda}+L_{\mu}$, 
the multiplicity $m_{\lambda,\,\mu}^{\nu}$ in the decomposition \eqref{eq:DD1} 
is equal to the tensor product multiplicity of the corresponding 
finite-dimensional irreducible $U_{q}(\Fgn)$-modules; 
see \S\ref{subsec:KT} for an explicit description of 
this tensor product multiplicity in terms of 
Littlewood-Richardson coefficients. In particular, 
we have $m_{\lambda,\,\mu}^{\nu} < \infty$ for all $\nu \in P_{+}$. 
However, in the case $\lambda,\,\mu \in P_{+}$, 
the total number of connected components of 
$\BB(\lambda) \otimes \BB(\mu)$ is infinite in general; 
compare with the other cases, in which the total number 
of connected components of $\BB(\lambda) \otimes \BB(\mu)$ is finite 
(see Remarks~\ref{rem:EE} and \ref{rem:DE}, and Theorem~\ref{thm:ED}). 
For example, let 
\begin{equation*}
\lambda=\mu:=\eps_{1}+\eps_{2}+\eps_{3}+\cdots. 
\end{equation*}
Then, for every $p \in \BZ_{\ge 1}$, 
\begin{equation*}
\xi_{p}:=-(\eps_{1}+\eps_{2}+\cdots+\eps_{p-1})+\eps_{p}+\eps_{p+1}+\cdots
\end{equation*}
is an element of $W\mu$, and $\pi_{\xi_{p}}$ is an element of 
$\BB(\mu)$ such that $\lambda+\pi_{\xi_{p}}$ is dominant. 
Consequently, it follows from part (1) of Theorem~\ref{thm:DD} 
that $\BB(\lambda) \otimes \BB(\mu)$ has infinitely many 
connected components. 
\end{rem}

\begin{proof}[Proof of Theorem~\ref{thm:DD}]
The formula \eqref{eq:DD1} is just a restatement of 
the result in \cite[\S10]{Lit-Ann}. 
While part (2) is essentially the same as 
\cite[Proposition~2.3.2]{Lec-TAMS} (see also its proof), 
we prefer to give another proof. Let $\nu \in P_{+}$. 
It is obvious from Remark~\ref{rem:wt} that 
if $L_{\nu} \ne L_{\lambda}+L_{\mu}$, then 
$m_{\lambda,\,\mu}^{\nu}=0$ since the level of the weight of
an element in $\BB(\lambda) \otimes \BB(\mu)$ 
is equal to $L_{\lambda}+L_{\mu}$. Assume, therefore, that 
$L_{\nu}=L_{\lambda}+L_{\mu}$. Fix $n \in \BZ_{\ge 3}$ such that 
$\col{j}=L_{\lambda}$, $\com{j}=L_{\mu}$, and 
$\con{j}=L_{\nu}=L_{\lambda}+L_{\mu}$ for all $j \ge n$.
We claim that
%
%%%%%%%%%%%%%%
%%% eq:DD3 %%%
%%%%%%%%%%%%%%
%
\begin{equation} \label{eq:DD3}
\begin{split}
\bigl\{
  \eta \in \BB(\mu) & \mid 
  \text{\rm $\lambda+\eta$ is dominant, and 
  $\wt(\pi_{\lambda} \otimes \eta)=\nu$}
 \bigr\} = \\
& 
\bigl\{
  \eta \in \BBn(\mu) \mid 
  \text{\rm $\lambda+\eta$ is $[n]$-dominant, and 
  $\wt(\pi_{\lambda} \otimes \eta)=\nu$}
 \bigr\};
\end{split}
\end{equation}
part (2) follows immediately from this equation and part (1). 
Let $\eta$ be an element in the set on the left-hand side of \eqref{eq:DD3}. 
Since $\lambda+\eta$ is dominant, it is obvious that 
$\lambda+\eta$ is $[n]$-dominant. 
We show that $\eta \in \BBn(\mu)$. 
Define $N=N_{\mu}$ as in Remark~\ref{rem:lcm}, and write 
$\eta \in \BB(\mu)$ as:
$\eta=(\xi_{1},\,\xi_{2},\,\dots,\,\xi_{N})$
for some $\xi_{1},\,\xi_{2},\,\dots,\,\xi_{N} \in W\mu$.
Then, we have by \eqref{eq:wtN}, 
\begin{equation*}
\nu=\wt (\pi_{\lambda} \otimes \eta)=
\lambda+\frac{1}{N}\sum_{M=1}^{N}\xi_{M}. 
\end{equation*}
Let $j \in \BZ_{\ge 0}$ be such that $j \ge n$. 
Then, by the choice of $n$, we have 
$\col{j}=L_{\lambda}$ and 
$\con{j}=L_{\nu}=L_{\lambda}+L_{\mu}$. 
Consequently, we obtain 
\begin{equation*}
L_{\lambda}+L_{\mu}=\con{j}=
\col{j}+\frac{1}{N}\sum_{M=1}^{N}\co{\xi_{M}}{j}=
L_{\lambda}+\frac{1}{N}\sum_{M=1}^{N}\co{\xi_{M}}{j}, 
\end{equation*}
and hence $L_{\mu}=(1/N)\sum_{M=1}^{N}\co{\xi_{M}}{j}$. 
Also, by using \eqref{eq:binf-r}--\eqref{eq:dinf-r}, 
we see from Remark~\ref{rem:dom} that $\co{\xi_{M}}{j} \le L_{\mu}$ 
for all $1 \le M \le N$. Combining these, we find that 
$\co{\xi_{M}}{j}=L_{\mu}$ for all $1 \le M \le N$. 
Now, for each $1 \le M \le N$, let $w_{M} \in \Wn$ be 
such that $\zeta_{M}:=w_{M}\xi_{M}$ is the unique $[n]$-dominant 
element in $\Wn\xi_{M}$. Then we see from Remark~\ref{rem:n-dom} 
that 
\begin{equation*}
0 \le \abs{\co{\zeta_{M}}{0}} \le 
 \co{\zeta_{M}}{1} \le \cdots \le 
 \co{\zeta_{M}}{n-1} \le
 \co{\zeta_{M}}{n}.
\end{equation*}
In addition, by using \eqref{eq:binf-r}--\eqref{eq:dinf-r}, 
we see that $\co{\zeta_{M}}{n} \le L_{\mu}$ 
from Remark~\ref{rem:dom}, and that 
$\co{\zeta_{M}}{j}=L_{\mu}$ for all $j \ge n+1$ 
since $w_{M} \in \Wn$ and $\co{\xi_{M}}{j}=L_{\mu}$ 
for all $j \ge n$ as shown above. 
Therefore, it follows through use of \eqref{eq:simple} that
$\zeta_{M}$ is dominant, and hence that 
$\zeta_{M} \in W\mu$ is identical to $\mu \in P_{+}$. 
Thus, we conclude that $\xi_{M} \in \Wn\mu$ for all $1 \le M \le N$, 
which implies that $\eta \in \BBn(\mu)$ by Lemma~\ref{lem:BBn}. 

Conversely, let $\eta$ be an element 
in the set on the right-hand side of \eqref{eq:DD3}. 
We show that $\lambda+\eta$ is dominant. 
Since $\lambda+\eta$ is $[n]$-dominant, it suffices to 
show that $\pair{\lambda+\eta(t)}{h_{i}} \ge 0$ 
for all $t \in [0,\,1]_{\BR}$ and $i \ge n+1$. 
Define $N=N_{\mu}$ 
as in Remark~\ref{rem:lcm}, and write 
$\eta \in \BB(\mu)$ as:
$\eta=(\xi_{1},\,\xi_{2},\,\dots,\,\xi_{N})$
for some $\xi_{1},\,\xi_{2},\,\dots,\,\xi_{N} \in W\mu$. 
Then, by the same reasoning as above, we obtain 
$\co{\xi_{M}}{j}=L_{\mu}$ for all $j \ge n$ and 
$1 \le M \le N$. Consequently, we derive from \eqref{eq:binf}--\eqref{eq:dinf} 
that $\pair{\xi_{M}}{h_{i}}=0$ for all $1 \le M \le N$ and $i \ge n+1$, 
which implies that $\pair{\eta(t)}{h_{i}}=0$ 
for all $t \in [0,\,1]_{\BR}$ and $i \ge n+1$. 
Hence it follows that $\pair{\lambda+\eta(t)}{h_{i}}=
\pair{\lambda}{h_{i}} \ge 0$ 
for all $t \in [0,\,1]_{\BR}$ and $i \ge n+1$. 
Thus we have proved that $\lambda+\eta$ is dominant. 
This completes the proof of Theorem~\ref{thm:DD}. 
\end{proof}

%==============================%
%     START SUBSECTION 0405    %
%==============================%
%
\subsection{The general case.}
\label{subsec:general}
Finally, in this subsection, 
we consider the decomposition (into connected components) 
of the tensor product 
$\BB(\lambda) \otimes \BB(\mu)$ for general 
$\lambda,\mu \in P$ 
such that $L_{\lambda},\,L_{\mu} \ge 0$. 
Define $\lambda^{+},\,\mu^{+} \in P_{+}$ and $\lambda^{0},\,\mu^{0} \in E$ 
as in Remark~\ref{rem:ED}. Then we have
\begin{equation*}
\BB(\lambda) \otimes \BB(\mu) \cong 
\BB(\lambda^{0}) \otimes \BB(\lambda^{+}) \otimes 
\BB(\mu^{0}) \otimes \BB(\mu^{+})
\end{equation*}
as $U_{q}(\Fg)$-crystals. 
Since $\lambda^{+} \in P_{+}$ and $\mu^{0} \in E$, 
it follows from Theorem~\ref{thm:DE} that
\begin{equation}
\BB(\lambda^{0}) \otimes \BB(\lambda^{+}) \otimes 
\BB(\mu^{0}) \otimes \BB(\mu^{+}) 
\cong \bigoplus_{\xi \in P_{+}^{[m]}(\lambda^{+},\,\mu^{0})} 
\bigl(
 \BB(\lambda^{0}) \otimes \BB(\xi) \otimes \BB(\mu^{+})
\bigr)^{\oplus m_{\lambda^{+},\,\mu_{0}}^{\xi}}
\end{equation}
as $U_{q}(\Fg)$-crystals, where we take $m \in \BZ_{\ge 0}$ 
(sufficiently large) as in \S\ref{subsec:DE}, 
with $\lambda$ replaced by $\lambda^{+}$, and $\mu$ by $\mu^{0}$. 
If we define $\xi^{+} \in P_{+}$ and $\xi^{0} \in E$ 
for each $\xi \in P_{+}^{[m]}(\lambda^{+},\,\mu^{0})$ 
as in Remark~\ref{rem:ED}, then we have 
\begin{equation*}
\BB(\lambda^{0}) \otimes \BB(\xi) \otimes \BB(\mu^{+}) \cong 
\BB(\lambda^{0}) \otimes \BB(\xi^{0}) \otimes 
\BB(\xi^{+}) \otimes \BB(\mu^{+})
\end{equation*}
as $U_{q}(\Fg)$-crystals. 
Since $\lambda^{0},\,\nu^{0} \in E$, 
it follows from Theorem~\ref{thm:EE} that
\begin{equation*}
\BB(\lambda^{0}) \otimes \BB(\xi^{0}) \cong 
\bigoplus_{\zeta \in \EW} 
\BB(\zeta)^{\oplus m_{\lambda^{0}_{\dagger},\,\xi^{0}_{\dagger}}^{\zeta}}
\qquad \text{as $U_{q}(\Fg)$-crystals}.
\end{equation*}
Also, since $\xi^{+},\,\mu^{+} \in P_{+}$, 
it follows from Theorem~\ref{thm:DD} that
\begin{equation*}
\BB(\xi^{+}) \otimes \BB(\mu^{+}) \cong 
\bigoplus_{\chi \in P_{+}} 
\BB(\chi)^{\oplus m_{\xi^{+},\,\mu^{+}}^{\chi}}
\qquad \text{as $U_{q}(\Fg)$-crystals}.
\end{equation*}
Combining these, we find that 
%
%%%%%%%%%%%%%%%%
%%% eq:gen01 %%%
%%%%%%%%%%%%%%%%
%
\begin{equation} \label{eq:gen01}
\BB(\lambda) \otimes \BB(\mu) \cong 
\bigoplus_{
 \begin{subarray}{c}
 \xi \in P_{+}^{[m]}(\lambda^{+},\,\mu^{0}), \\[1mm] 
 \zeta \in \EW,\ 
 \chi \in P_{+}
 \end{subarray}
} 
\bigl(\BB(\zeta) \otimes \BB(\chi)\bigr)^{
\oplus m_{\lambda^{+},\,\mu_{0}}^{\xi} \,
m_{\lambda^{0}_{\dagger},\,\xi^{0}_{\dagger}}^{\zeta} \, 
m_{\xi^{+},\,\mu^{+}}^{\chi}}
\end{equation}
as $U_{q}(\Fg)$-crystals. 
Recall that $\EW$ is a complete set of representatives 
for $W$-orbits in $E$. Let $P/W$ denote the set of 
all $W$-orbits in $P$, which we regard as a subset of 
$P$ by taking a complete set of representatives 
for $W$-orbits in $P$; recall Corollary~\ref{cor:isom}\,(3).
Now, using Theorem~\ref{thm:ED} and Remark~\ref{rem:ED}, 
together with Proposition~\ref{prop:ED-isom}, we obtain 
from \eqref{eq:gen01} the following theorem 
(cf. \cite[Corollary~7.4]{Kw-Adv} in type $A_{+\infty}$ and 
 \cite[Theorem~5.14]{Kw-Ep} in type $A_{\infty}$).
%
%%%%%%%%%%%%%%%%%%%
%%% thm:general %%%
%%%%%%%%%%%%%%%%%%%
%
\begin{thm} \label{thm:general}
Let $\lambda,\mu \in P$ be integral weights 
of nonnegative levels\,{\rm;} namely, 
$L_{\lambda},\,L_{\mu} \ge 0$. Then, 
we have the following decomposition 
into connected components\,{\rm:}
%
%%%%%%%%%%%%%%%%%%
%%% eq:general %%%
%%%%%%%%%%%%%%%%%%
%
\begin{equation} \label{eq:general}
\BB(\lambda) \otimes \BB(\mu) = 
\bigoplus_{\nu \in P/W}\BB(\nu)^{m_{\lambda,\mu}^{\nu}}, 
\end{equation}
where for each $\nu \in P/W$, the multiplicity 
$m_{\lambda,\mu}^{\nu}$ is given as follows\,{\rm:}
\begin{equation*}
m_{\lambda,\,\mu}^{\nu}=
\sum_{
 \xi \in P_{+}^{[m]}(\lambda^{+},\,\mu^{0})
}
m_{\lambda^{+},\,\mu_{0}}^{\xi}\,
m_{\lambda^{0}_{\dagger},\,\xi^{0}_{\dagger}}^{\nu^{0}_{\dagger}}\,
m_{\xi^{+},\,\mu^{+}}^{\nu^{+}}.
\end{equation*}
\end{thm}

%=========================%
%     START SECTION 05    %
%=========================%
%
\section{Proof of Proposition~\ref{prop:DE-bij}.}
\label{sec:DE-bij}

%==============================%
%     START SUBSECTION 0501    %
%==============================%
%
\subsection{Basic notation for Young diagrams.}
\label{subsec:young}
Recall that $\Par$ denotes the set of all partitions; 
we usually identify a partition $\rho \in \Par$ 
with the corresponding Young diagram, which is also denoted by $\rho$ 
(see, for example, \cite{Fulton}). 
For $\rho,\,\kappa \in \Par$, 
we write $\rho \subset \kappa$ 
if the Young diagram of $\rho$ is contained 
in the Young diagram of $\kappa$; 
in this case, we denote by $\kappa/\rho$ the skew Young diagram 
obtained from the Young diagram of $\kappa$ by removing 
that of $\rho$, and by $|\kappa/\rho|$ the total number of boxes 
in this skew Young diagram, i.e., $|\kappa/\rho|=|\kappa|-|\rho|$. 
It is well-known that for $\rho,\,\kappa,\,\omega \in \Par$, 
the Littlewood-Richardson coefficient $\LR_{\rho,\,\kappa}^{\omega}$ 
is nonzero only if $\rho \subset \omega$ and $\kappa \subset \omega$. 
Also, the conjugate of the partition $\rho \in \Par$ is 
denoted by ${}^{t}\rho$. In what follows, 
for $L \in \BZ_{\ge 0}$ and $\rho \in \Par$, 
let $\iota_{L}(\rho)$ denote the partition 
whose Young diagram is obtained from the Young diagram 
of $\rho$ by inserting one row with exactly $L$ boxes 
between an appropriate pair of adjacent rows of the Young diagram of 
$\rho$ in such a way that the resulting diagram is also 
a Young diagram. By convention, we set $\iota_{L}(\rho)=\rho$ 
if $L=0$. 

We fix $\ell \in \BZ_{\ge 3}$; 
recall from Remark~\ref{rem:Levi} 
that if $\Fg$ is of type $B_{\infty}$ 
(resp., $C_{\infty}$, $D_{\infty}$), then 
$\Fgm{\ell}$ is a ``reductive'' Lie algebra 
of type $B_{\ell+1}$ (resp., $C_{\ell+1}$, $D_{\ell+1}$). 
Let $\rho=(\co{\rho}{0} \ge \co{\rho}{1} \ge \co{\rho}{2} \ge \cdots ) \in \Par$ 
be a partition whose length $\ell(\rho)$ is less than or equal to $\ell+1$, 
and let $J$ be a finite subset of $\BZ_{\ge 1}$. 
We define $\Jl{\rho} \in \Par$ as follows (cf. \cite[\S8]{Ko98}).
Write the conjugate partition ${}^{t}\rho$ of $\rho$ as: 
${}^{t}\rho=({}^{t}\co{\rho}{0} \ge {}^{t}\co{\rho}{1} \ge {}^{t}\co{\rho}{2} \ge \cdots) \in \Par$; 
note that the first part ${}^{t}\co{\rho}{0}$ is equal to 
the length $\ell(\rho)$ of $\rho$, 
which is less than or equal to $\ell+1$ by assumption. 
Then we set 
%
%%%%%%%%%%%%%%
%%% eq:bb1 %%%
%%%%%%%%%%%%%%
%
\begin{equation} \label{eq:bb1}
\bb({}^{t}\rho)=(b_{j})_{j \in \BZ_{\ge 1}}:=
\bigl(
 {}^{t}\co{\rho}{0} > {}^{t}\co{\rho}{1}-1 > \cdots > 
 {}^{t}\co{\rho}{j-1}-(j-1) > {}^{t}\co{\rho}{j}-j > \cdots
\bigr), 
\end{equation}
and define $\Jl{\bb({}^{t}\rho)}=
(\Jl{b_{j}})_{j \in \BZ_{\ge 1}}$ by: 
for each $j \in \BZ_{\ge 1}$, 
%
%%%%%%%%%%%%%%
%%% eq:bb2 %%%
%%%%%%%%%%%%%%
%
\begin{equation} \label{eq:bb2}
\Jl{b_{j}}=
\begin{cases}
b_{j}
 & \text{if $j \notin J$}, \\[1.5mm]
R_{\ell}-b_{j}
 & \text{if $j \in J$},
\end{cases}
\quad \text{where} \quad
R_{\ell}:=
 \begin{cases}
 2\ell+3 & \text{if $\Fg$ is of type $B_{\infty}$}, \\[1.5mm]
 2\ell+4 & \text{if $\Fg$ is of type $C_{\infty}$}, \\[1.5mm]
 2\ell+2 & \text{if $\Fg$ is of type $D_{\infty}$}.
 \end{cases}
\end{equation}
Observe that
\begin{equation*}
\Jl{b_{j}} \ge 
\begin{cases}
\ell+2 & \text{if $\Fg$ is of type $B_{\infty}$}, \\[1.5mm]
\ell+3 & \text{if $\Fg$ is of type $C_{\infty}$}, \\[1.5mm]
\ell+1 & \text{if $\Fg$ is of type $D_{\infty}$}, \\[1.5mm]
\end{cases}
\quad \text{for all $j \in J$},
\end{equation*}
and that $\Jl{b_{j}} \ne \Jl{b_{k}}$ for all 
$j,\,k \in \BZ_{\ge 1}$ such that $j \ne k$. 
Let $\tau$ be the (unique) finite permutation of 
the set $\BZ_{\ge 1}$ such that 
\begin{equation*}
\Jl{b_{\tau(1)}} > \Jl{b_{\tau(2)}} > \cdots > 
\Jl{b_{\tau(j)}} > \Jl{b_{\tau(j+1)}} > \cdots;
\end{equation*}
note that $\tau(j)=j$ for all $j > \max J$. 
Now we define $\Jl{\rho}$ to be the conjugate 
of the partition 
\begin{equation} \label{eq:rhoJ}
(\Jl{b_{\tau(1)}} \ge \Jl{b_{\tau(2)}}+1 \ge \cdots \ge 
 \Jl{b_{\tau(j)}}+(j-1) \ge \Jl{b_{\tau(j+1)}}+j \ge \cdots).
\end{equation}
Also, we define 
\begin{equation*}
\Jl{\sgn}(\rho):=
\begin{cases}
 \sgn(\tau) & 
 \text{if $\Fg$ is of type $B_{\infty}$ or $D_{\infty}$}, \\[1.5mm]
 \sgn(\tau) \times (-1)^{\#J} & 
 \text{if $\Fg$ is of type $C_{\infty}$},
\end{cases}
\end{equation*}
where $\sgn(\tau)$ denotes the sign of the finite 
permutation $\tau$ of the set $\BZ_{\ge 1}$. 
%
%%%%%%%%%%%%%%
%%% rem:LR %%%
%%%%%%%%%%%%%%
%
\begin{rem} \label{rem:LR}
Suppose that $\Fg$ is of type $D_{\infty}$ and $\ell(\rho)=\ell+1$. 
Let $J$ be a finite subset of $\BZ_{\ge 2}$, and set 
$K:=J \cup \{1\}$. Since ${}^{t}\co{\rho}{0}=\ell(\rho)=\ell+1$ in \eqref{eq:bb1}, 
it follows from \eqref{eq:bb2} that $\Jl{b_{j}}=b_{j}^{K,\,\ell}$ 
for all $j \in \BZ_{\ge 1}$. Thus, we obtain 
$\Jl{\rho}=\rho^{K,\,\ell}$ and $\Jn{\sgn}(\rho)=\sgn^{K,\,n}(\rho)$. 
\end{rem}

%==============================%
%     START SUBSECTION 0502    %
%==============================%
%
\subsection{Tensor product multiplicity formulas 
 in terms of Littlewood-Richardson coefficients.}
\label{subsec:KT}

In this subsection, we review some tensor product multiplicity 
formulas from \cite{Ko97} and \cite{Ko98}, 
which we use in the proof of Proposition~\ref{prop:DE-bij}. 
Fix $\ell \in \BZ_{\ge 3}$. 
For an $[\ell]$-dominant integral weight $\lambda \in P$, 
we denote $\ch \Vm{\ell}(\lambda)$ by the formal character of 
the finite-dimensional irreducible $U_{q}(\Fgm{\ell})$-module 
$\Vm{\ell}(\lambda)$ of highest weight $\lambda$. 
We set
\begin{equation*}
\Ldom{\ell}:=
\bigl\{\lambda \in P \mid 
 \text{$\lambda$ is $[\ell]$-dominant, and 
 $\col{j}=L_{\lambda}$ for all $j \in \BZ_{\ge \ell+1}$}
\bigr\}.
\end{equation*}
For each $\lambda \in \Ldom{\ell}$, 
we define a sequence $\phi_{\ell}(\lambda)$ by: 
\begin{equation*}
\phi_{\ell}(\lambda)=
(\col{\ell},\,\dots,\,\col{1},\,\abs{\col{0}},\,0,\,0,\,\dots). 
\end{equation*}
Note that by Remark~\ref{rem:n-dom}, we have  
$\col{\ell} \ge \dots \ge \col{1} \ge \abs{\col{0}} \ge 0$, 
and hence that if $L_{\lambda} \in \BZ$, 
then $\phi_{\ell}(\lambda)$ is a partition. 

For the rest of this subsection, 
we fix $\lambda,\,\mu \in \Ldom{\ell}$. 
Because every weight of $\Vm{\ell}(\lambda)$ 
(resp., $\Vm{\ell}(\mu)$) is contained in the set 
$\lambda-\sum_{i \in [\ell]}\BZ_{\ge 0}\alpha_{i}$ 
(resp., $\mu-\sum_{i \in [\ell]}\BZ_{\ge 0}\alpha_{i}$), 
we deduce through use of \eqref{eq:simple} that 
every weight $\nu \in P$ of 
$\Vm{\ell}(\lambda) \otimes \Vm{\ell}(\mu)$ 
satisfies the condition that
\begin{equation*}
\con{j}=\col{j}+\com{j}=L_{\lambda}+L_{\mu}
\qquad \text{for all $j \in \BZ_{\ge \ell+1}$}.
\end{equation*}
In particular, we have $L_{\nu}=L_{\lambda}+L_{\mu}$. 

%%%%%%%%%%
\paragraph{Case of type $C_{\infty}$.}
%%%%%%%%%%
%
Note that in this case, 
$\Fgm{\ell}$ is of type $C_{\ell+1}$, and 
$L_{\lambda},\,L_{\mu} \in \BZ$ (see \eqref{eq:cinf-f}), 
and hence $\phi_{\ell}(\lambda),\,\phi_{\ell}(\mu) \in \Par$. 
We have the following formula by \cite[Theorem~6.6\,(2)]{Ko98}: 
%
%%%%%%%%%%%%%%%
%%% eq:KT01 %%%
%%%%%%%%%%%%%%%
%
\begin{equation} \label{eq:KT01}
\begin{split}
& \ch \Vm{\ell}(\lambda) \times \ch \Vm{\ell}(\mu) = \\[3mm]
& 
\sum_{
 \begin{subarray}{c}
 \nu \in \Ldom{\ell}, \\[0.5mm]
 L_{\nu}=L_{\lambda}+L_{\mu}
 \end{subarray}
} 
\left\{
\sum_{
 \begin{subarray}{c}
 J \subset \BZ_{\ge 1},\,\#J < \infty, \\[1mm]
 \omega_{1},\,\omega_{2},\,\omega_{3} \in \Par
 \end{subarray}
}
 \LR^{\phi_{\ell}(\lambda)}_{\omega_{1},\,\omega_{2}} \,
 \LR^{\phi_{\ell}(\mu)}_{\omega_{2},\,\omega_{3}} \,
 \LR^{\Jl{\phi_{\ell}(\nu)}}_{\omega_{3},\,\omega_{1}} \, 
 \Jl{\sgn}(\phi_{\ell}(\nu))
\right\}
\ch \Vm{\ell}(\nu). 
\end{split}
\end{equation}
For simplicity of notation, we set
%
%%%%%%%%%%%%%%
%%% eq:CLR %%%
%%%%%%%%%%%%%%
%
\begin{equation} \label{eq:CLR}
\CLRm{\rho_{1}}{\rho_{2}}{\rho}{\ell}:=
\sum_{
 \begin{subarray}{c}
 J \subset \BZ_{\ge 1},\,\#J < \infty, \\[1mm]
 \omega_{1},\,\omega_{2},\,\omega_{3} \in \Par
 \end{subarray}
}
 \LR^{\rho_{1}}_{\omega_{1},\,\omega_{2}} \,
 \LR^{\rho_{2}}_{\omega_{2},\,\omega_{3}} \,
 \LR^{\Jl{\rho}}_{\omega_{3},\,\omega_{1}} \, 
 \Jl{\sgn}(\rho)
\end{equation}
for partitions $\rho_{1},\,\rho_{2},\,\rho \in \Par$ 
whose lengths are less than or equal to $\ell+1$. 

%%%%%%%%%%
\paragraph{Case of type $B_{\infty}$.}
%%%%%%%%%%
%
Note that in this case, $\Fgm{\ell}$ is of type $B_{\ell+1}$. 
Suppose first that $L_{\lambda},\,L_{\mu} \in \BZ$, 
and hence $\phi_{\ell}(\lambda),\,\phi_{\ell}(\mu) \in \Par$. 
Then we have the following formula by 
\cite[Theorem~6.6\,(1) along with Definition~1.2\,(1)]{Ko98}: 
%
%%%%%%%%%%%%%%%%
%%% eq:KT02a %%%
%%%%%%%%%%%%%%%%
%
\begin{equation} \label{eq:KT02a}
\begin{split}
& \ch \Vm{\ell}(\lambda) \times \ch \Vm{\ell}(\mu) = \\[3mm]
& 
\sum_{
 \begin{subarray}{c}
 \nu \in \Ldom{\ell}, \\[0.5mm]
 L_{\nu}=L_{\lambda}+L_{\mu}
 \end{subarray}
} 
\left\{
\sum_{
 \begin{subarray}{c}
 J \subset \BZ_{\ge 1},\,\#J < \infty, \\[1mm]
 \omega_{1},\,\omega_{2},\,\omega_{3} \in \Par
 \end{subarray}
}
 \LR^{\phi_{\ell}(\lambda)}_{\omega_{1},\,\omega_{2}} \,
 \LR^{\phi_{\ell}(\mu)}_{\omega_{2},\,\omega_{3}} \,
 \LR^{\Jl{\phi_{\ell}(\nu)}}_{\omega_{3},\,\omega_{1}} \, 
 \Jl{\sgn}(\phi_{\ell}(\nu))
\right\}
\ch \Vm{\ell}(\nu). 
\end{split}
\end{equation}
For simplicity of notation, we set
%
%%%%%%%%%%%%%%
%%% eq:BLR %%%
%%%%%%%%%%%%%%
%
\begin{equation} \label{eq:BLR}
\BLRm{\rho_{1}}{\rho_{2}}{\rho}{\ell}:=
\sum_{
 \begin{subarray}{c}
 J \subset \BZ_{\ge 1},\,\#J < \infty, \\[1mm]
 \omega_{1},\,\omega_{2},\,\omega_{3} \in \Par
 \end{subarray}
}
 \LR^{\rho_{1}}_{\omega_{1},\,\omega_{2}} \,
 \LR^{\rho_{2}}_{\omega_{2},\,\omega_{3}} \,
 \LR^{\Jl{\rho}}_{\omega_{3},\,\omega_{1}} \, 
 \Jl{\sgn}(\rho)
\end{equation}
for partitions $\rho_{1},\,\rho_{2},\,\rho \in \Par$ 
whose lengths are less than or equal to $\ell+1$. 

Suppose next that $L_{\lambda} \in (1/2)+\BZ$ and $L_{\mu} \in \BZ$, 
and hence $\phi_{\ell}(\mu) \in \Par$, 
but $\phi_{\ell}(\lambda) \notin \Par$ since 
$\col{j} \in (1/2)+\BZ_{\ge 0}$ for all $j \in [\ell]$. 
Note that $\Lambda_{0} \in \Ldom{\ell}$ 
by \eqref{eq:binf-f}, and that if $\xi \in \Ldom{\ell}$ is 
such that $L_{\xi} \in (1/2)+\BZ$, then the level of 
$\delta:=\xi-\Lambda_{0} \in \Ldom{\ell}$ 
is equal to $L_{\xi}-1/2 \in \BZ$; also, observe that 
$\phi_{\ell}(\delta)$ is a partition, and 
\begin{equation*}
\phi_{\ell}(\xi)=
\left(
 \co{\delta}{\ell}+\frac{1}{2},\,\dots,\,
 \co{\delta}{1}+\frac{1}{2},\,
 \co{\delta}{0}+\frac{1}{2},\,0,\,0,\,\dots\right).
\end{equation*}
We have the following formula by \cite[(1.2.1)]{Ko97}: 
%
%%%%%%%%%%%%%%%%
%%% eq:KT02b %%%
%%%%%%%%%%%%%%%%
%
\begin{equation} \label{eq:KT02b}
\ch \Vm{\ell}(\lambda) \times \ch \Vm{\ell}(\mu) = 
\sum_{
 \begin{subarray}{c}
 \nu \in \Ldom{\ell}, \\[0.5mm]
 L_{\nu}=L_{\lambda}+L_{\mu}
 \end{subarray}
} 
\left\{
\sum_{
 \begin{subarray}{c}
 \rho_{2} \in \Par, \ 
 \rho_{2} \subset \phi_{\ell}(\mu), \\[1mm]
 \text{$\phi_{\ell}(\mu)/\rho_{2}$\,:\,vertical strip}
 \end{subarray}
}
\CLRm{\phi_{\ell}(\lambda-\Lambda_{0})}{\rho_{2}}{\phi_{\ell}(\nu-\Lambda_{0})}{\ell}
\right\} \ch \Vm{\ell}(\nu). 
\end{equation}

%%%%%%%%%%
\paragraph{Case of type $D_{\infty}$.}
%%%%%%%%%%
%
Note that in this case, 
$\Fgm{\ell}$ is of type $D_{\ell+1}$. 
For $\lambda \in \Ldom{\ell}$, we set 
(cf. \cite[Definition~1.1]{Ko98})
\begin{align*}
& \ch^{(+)} \Vm{\ell}(\lambda):=
  \ch \Vm{\ell}(\lambda)+\ch \Vm{\ell}(\ol{\lambda}), \\
& \ch^{(-)} \Vm{\ell}(\lambda):=
  \ch \Vm{\ell}(\lambda)-\ch \Vm{\ell}(\ol{\lambda}), 
\end{align*}
where $\ol{\lambda}$ is the element of $\Ldom{\ell}$ 
defined by: $\co{\ol{\lambda}}{0}=-\col{0}$ and 
$\co{\ol{\lambda}}{j}=\col{j}$ for all $j \in \BZ_{\ge 1}$. 
In addition, 
for $\lambda \in \Ldom{\ell}$ such that $L_{\lambda} \in \BZ$,
we set (cf. \cite[Definition~1.2\,(2),(3)]{Ko98}) 
\begin{equation*}
\ti{\ch} \Vm{\ell}(\lambda):=
\begin{cases}
 \ch \Vm{\ell}(\lambda) & \text{if $\col{0}=0$}, \\[1.5mm]
 \ch^{(+)} \Vm{\ell}(\lambda) & \text{if $\col{0} \ne 0$}.
\end{cases}
\end{equation*}

Suppose first that $L_{\lambda},\,L_{\mu} \in \BZ$, and hence 
$\phi_{\ell}(\lambda),\,\phi_{\ell}(\mu) \in \Par$. 
Then we have the following formula by 
\cite[Theorem~6.6\,(3) along with Definition~1.2\,(2),\,(3)]{Ko98} and 
Remark~\ref{rem:LR}: 
%
%%%%%%%%%%%%%%%
%%% eq:KT03 %%%
%%%%%%%%%%%%%%%
%
\begin{equation} \label{eq:KT03}
\begin{split}
& \ti{\ch} \Vm{\ell}(\lambda) \times \ti{\ch} \Vm{\ell}(\mu) = \\[3mm]
& 
\sum_{
 \begin{subarray}{c}
 \nu \in \Ldom{\ell}, \\[0.5mm]
 L_{\nu}=L_{\lambda}+L_{\mu}, \\[0.5mm]
 \con{0} \ge 0
 \end{subarray}
} 
\left\{c(\nu) 
\sum_{
 \begin{subarray}{c}
 J \subset \BZ_{\ge 1},\,\#J < \infty, \\[1mm]
 \omega_{1},\,\omega_{2},\,\omega_{3} \in \Par
 \end{subarray}
}
 \LR^{\phi_{\ell}(\lambda)}_{\omega_{1},\,\omega_{2}} \,
 \LR^{\phi_{\ell}(\mu)}_{\omega_{2},\,\omega_{3}} \,
 \LR^{\Jl{\phi_{\ell}(\nu)}}_{\omega_{3},\,\omega_{1}} \, 
 \Jl{\sgn}(\phi_{\ell}(\nu))
\right\}
\ti{\ch} \Vm{\ell}(\nu),
\end{split}
\end{equation}
where for $\nu \in \Ldom{\ell}$ with $L_{\nu} \in \BZ$, we set
\begin{equation*}
c(\nu):=
\begin{cases}
\dfrac{1}{2} & \text{if $\ell(\phi_{\ell}(\nu))=\ell+1$}, \\[3mm]
1 & \text{otherwise}.
\end{cases}
\end{equation*}
For simplicity of notation, we set 
%
%%%%%%%%%%%%%%
%%% eq:DLR %%%
%%%%%%%%%%%%%%
%
\begin{equation} \label{eq:DLR}
\DLRm{\rho_{1}}{\rho_{2}}{\rho}{\ell}:=
\sum_{
 \begin{subarray}{c}
 J \subset \BZ_{\ge 1},\,\#J < \infty, \\[1mm]
 \omega_{1},\,\omega_{2},\,\omega_{3} \in \Par
 \end{subarray}
}
 \LR^{\rho_{1}}_{\omega_{1},\,\omega_{2}} \,
 \LR^{\rho_{2}}_{\omega_{2},\,\omega_{3}} \,
 \LR^{\Jl{\rho}}_{\omega_{3},\,\omega_{1}} \, 
 \Jl{\sgn}(\rho)
\end{equation}
for partitions $\rho_{1},\,\rho_{2},\,\rho \in \Par$ 
whose lengths are less than or equal to $\ell+1$.
Also, if $\col{0} > 0$, 
then we have the following formula by \cite[(1.2.8)]{Ko97}: 
%
%%%%%%%%%%%%%%%
%%% eq:KT04 %%%
%%%%%%%%%%%%%%%
%
\begin{equation} \label{eq:KT04}
\begin{split}
& \ch^{(-)} \Vm{\ell}(\lambda) \times \ti{\ch} \Vm{\ell}(\mu) = \\[3mm]
& 
\sum_{
 \begin{subarray}{c}
 \nu \in \Ldom{\ell}, \\[0.5mm]
 L_{\nu}=L_{\lambda}+L_{\mu}, \\[0.5mm]
 \con{0} > 0
 \end{subarray}
} 
\left\{
\sum_{
 \begin{subarray}{c}
 \rho_{2},\,\omega \in \Par, \ 
 \rho_{2} \subset \omega \subset \phi_{\ell}(\mu), \\[1mm]
 \text{$\phi_{\ell}(\mu)/\omega$\,:\,vertical strip}, \\[1mm]
 \text{$\omega/\rho_{2}$\,:\,vertical strip}
 \end{subarray}
}
(-1)^{|\omega/\rho_{2}|} \, 
\CLRm{\phi_{\ell}(\lambda)-(1^{\ell+1})}{\rho_{2}}{\phi_{\ell}(\nu)-(1^{\ell+1})}{\ell}
\right\} \ch^{(-)} \Vm{\ell}(\nu),
\end{split}
\end{equation}
where for a partition 
$\rho=(\co{\rho}{0} \ge \co{\rho}{1} \ge \cdots) \in \Par$ 
of length $\ell+1$, we set
\begin{equation*}
\rho-(1^{\ell+1}):=
(\co{\rho}{0}-1,\,\co{\rho}{1}-1,\,\dots,\,\co{\rho}{\ell}-1,\,0,\,0,\,\dots) \in \Par. 
\end{equation*}

Suppose next that $L_{\lambda} \in 1/2+\BZ$, $\col{0} > 0$, and 
$L_{\mu} \in \BZ$, and hence $\phi_{\ell}(\mu) \in \Par$, 
but $\phi_{\ell}(\lambda) \notin \Par$ since 
$\col{j} \in (1/2)+\BZ_{\ge 0}$ for all $j \in [\ell]$. 
Note that 
$\Lambda_{0} \in \Ldom{\ell}$ and $\Lambda_{1} \in \Ldom{\ell}$ 
by \eqref{eq:dinf-f}, and that if $\xi \in \Ldom{\ell}$ is 
such that $L_{\xi} \in (1/2)+\BZ$ and $\cox{0} > 0$, 
then the level of $\delta:=\xi-\Lambda_{0} \in \Ldom{\ell}$ 
is equal to $L_{\xi}-1/2 \in \BZ$, and $\co{\delta}{0} \ge 0$; 
also, observe that $\phi_{\ell}(\delta)=
(\co{\delta}{\ell},\,\dots,\,
 \co{\delta}{1},\,\abs{\co{\delta}{0}}=\co{\delta}{0},\,0,\,0,\,\dots)$ 
is a partition, and 
\begin{equation*}
\phi_{\ell}(\xi)=
\left(
 \co{\delta}{\ell}+\frac{1}{2},\,\dots,\,
 \co{\delta}{1}+\frac{1}{2},\,
 \co{\delta}{0}+\frac{1}{2},\,0,\,0,\,\dots\right). 
\end{equation*}
We have the following formulas by \cite[(1.2.3) and (1.2.2)]{Ko97}: 
%
%%%%%%%%%%%%%%%
%%% eq:KT05 %%%
%%%%%%%%%%%%%%%
%
\begin{equation} \label{eq:KT05}
\begin{split}
& \ch^{(+)} \Vm{\ell}(\lambda) \times 
  \ti{\ch} \Vm{\ell}(\mu)= \\[3mm]
& 
\sum_{
 \begin{subarray}{c}
 \nu \in \Ldom{\ell}, \\[0.5mm]
 L_{\nu}=L_{\lambda}+L_{\mu}, \\[0.5mm]
 \con{0} > 0
 \end{subarray}
} 
\left\{
\sum_{
 \begin{subarray}{c}
 \rho_{2} \in \Par, \ 
 \rho_{2} \subset \phi_{\ell}(\mu), \\[1mm]
 \text{$\phi_{\ell}(\mu)/\rho_{2}$\,:\,vertical strip}
 \end{subarray}
}
(-1)^{|\phi_{\ell}(\lambda-\Lambda_{0})|+|\rho_{2}|+|\phi_{\ell}(\nu-\Lambda_{0})|}
\BLRm{\phi_{\ell}(\lambda-\Lambda_{0})}{\rho_{2}}{\phi_{\ell}(\nu-\Lambda_{0})}{\ell}
\right\}
\ch^{(+)} \Vm{\ell}(\nu);
\end{split}
\end{equation}
%
%%%%%%%%%%%%%%%
%%% eq:KT06 %%%
%%%%%%%%%%%%%%%
%
\begin{equation} \label{eq:KT06}
\begin{split}
& \ch^{(-)} \Vm{\ell}(\lambda) \times 
  \ti{\ch} \Vm{\ell}(\mu)= \\[3mm]
& \qquad
\sum_{
 \begin{subarray}{c}
 \nu \in \Ldom{\ell}, \\[0.5mm]
 L_{\nu}=L_{\lambda}+L_{\mu}, \\[0.5mm]
 \con{0} > 0
 \end{subarray}
} 
\left\{
\sum_{
 \begin{subarray}{c}
 \rho_{2} \in \Par, \ 
 \rho_{2} \subset \phi_{\ell}(\mu), \\[1mm]
 \text{$\phi_{\ell}(\mu)/\rho_{2}$\,:\,vertical strip}
 \end{subarray}
}
(-1)^{|\phi_{\ell}(\mu)/\rho_{2}|}
\BLRm{\phi_{\ell}(\lambda-\Lambda_{0})}{\rho_{2}}{\phi_{\ell}(\nu-\Lambda_{0})}{\ell}
\right\}
\ch^{(-)} \Vm{\ell}(\nu).
\end{split}
\end{equation}
%
%==============================%
%     START SUBSECTION 0503    %
%==============================%
%
\subsection{Proof of Proposition~\ref{prop:DE-bij}.}
\label{subsec:key}

The following proposition plays an essential role in the proof of 
Proposition~\ref{prop:DE-bij}.
%
%%%%%%%%%%%%%%%%%
%%% prop:stab %%%
%%%%%%%%%%%%%%%%%
%
\begin{prop} \label{prop:stab}
Let $L \in \BZ_{\ge 0}$, and $n \in \BZ_{\ge 3}$. 
Let $\rho_{1},\,\rho_{2},\,\rho \in \Par$ 
be partitions satisfying the following conditions\,{\rm:}

{\rm (i)} The lengths of these partitions are all 
less than or equal to $n+1$. 

{\rm (ii)} The first part of $\rho_{1}$ is equal to $L$. 

{\rm (iii)} There hold the inequalities 
\begin{align*}
& y_{1}:=
 \#\bigl\{1 \le j \le n+1 \mid 
 \text{\rm the $j$-th part of $\rho_{1}$ is equal to $L$}
 \bigr\} > |\rho_{2}|, \\
& y:=
 \#\bigl\{1 \le j \le n+1 \mid 
 \text{\rm the $j$-th part of $\rho$ is equal to $L$}
 \bigr\} > |\rho_{2}|.
\end{align*}
If we set $\kappa_{1}:=\iota_{L}(\rho_{1})$ and 
$\kappa:=\iota_{L}(\rho)$, then we have 
\begin{equation*}
\XLRm{\rho_{1}}{\rho_{2}}{\rho}{n}=
\XLRm{\kappa_{1}}{\rho_{2}}{\kappa}{n+1}
\qquad \text{\rm for $X=B,\,C,\,D$}.
\end{equation*}
\end{prop}

The proof of this proposition 
will be given in \S\ref{subsec:prf-key}. 

\begin{proof}[Proof of Proposition~\ref{prop:DE-bij}]
We fix $n \in \BZ_{\ge 0}$ such that 
$n > p+(N+1)q$, and $\nu \in \Pn$ 
as in Proposition~\ref{prop:DE-bij}. 
Recall from Remark~\ref{rem:LS-n} that 
$\BBn(\lambda)$ (resp., $\BBn(\mu)=\BBn(\mun)$) 
is isomorphic, as a $U_{q}(\Fgn)$-crystal, to 
the crystal basis of the finite-dimensional irreducible 
$U_{q}(\Fgn)$-module of highest weight $\lambda$ (resp., $\mun$), 
and similar statements hold for 
$\BBo(\lambda)$ and $\BBo(\mu)=\BBo(\muo)$.
Therefore, the sets
$\BBhiw{n}{\nu} \subset \BBn(\lambda) \otimes \BBn(\mu)$ and 
$\BBhiw{n+1}{\nuo} \subset \BBo(\lambda) \otimes \BBo(\mu)$ 
are both finite sets. 
Because it is already shown that the map 
$S_{x}:\BBhiw{n+1}{\nuo} \rightarrow \BBhiw{n}{\nu}$ is injective, 
it suffices to show that $\#\BBhiw{n+1}{\nuo}=\#\BBhiw{n}{\nu}$.
Also, recall from Remark~\ref{rem:DE} that 
the number $\#\BBhiw{n}{\nu}$ (resp., $\#\BBhiw{n+1}{\nuo}$) 
is equal to the multiplicity of 
$\Vn(\nu)$ (resp., $\Vo(\nuo)$) 
in the tensor product $\Vn(\lambda) \otimes \Vn(\mun)$ 
(resp., $\Vo(\lambda) \otimes \Vo(\muo)$), and hence to 
the coefficient of $\ch \Vn(\nu)$ (resp., $\ch \Vo(\nuo)$) 
in the product $\ch \Vn(\lambda) \times \ch \Vn(\mun)$ 
(resp., $\ch \Vo(\lambda) \times \ch \Vo(\muo)$). 
Below we will use an explicit description of the numbers 
$\#\BBhiw{n}{\nu}$ and $\#\BBhiw{n+1}{\nuo}$ in terms of 
Littlewood-Richardson coefficients obtained from 
the formulas in \S\ref{subsec:KT}. 

Since $n > p+(N+1)q > q$ and $\Supp(\mu) \subset [q-1]$, 
it follows from Lemma~\ref{lem:lamn} 
that $\Supp(\mum{\ell}) \subset [\ell]$ and $\co{\mum{\ell}}{0}=0$ 
for $\ell=n,\,n+1$, and hence that 
$\mum{\ell} \in \Ldom{\ell}$ for $\ell=n,\,n+1$. 
In addition, it is easily seen that 
$\phi_{n}(\mun)=\phi_{n+1}(\muo)=\mu_{\dagger}$, 
and that $|\mu|=\sum_{j \in \BZ_{\ge 0}}|\com{j}|$ 
is equal to the sum $|\mu_{\dagger}|$ 
of all parts of $\mu_{\dagger}$; note that 
\begin{equation*}
n-p+1 \ge (N+1)q+1 > |\mu|=|\mu_{\dagger}|.
\end{equation*}
Also, we infer from Remark~\ref{rem:dom} and 
the choice of $n$ that 
%
%%%%%%%%%%%%%%%
%%% eq:coll %%%
%%%%%%%%%%%%%%%
%
\begin{equation} \label{eq:coll}
\cdots=\col{n+1}=\col{n}=\cdots=\col{p}=L_{\lambda} > 
\col{p-1} \ge \cdots \ge \col{1} \ge \abs{\col{0}} \ge 0,
\end{equation}
and hence that $\lambda$ is contained in both of the sets 
$\Ldom{n}$ and $\Ldom{n+1}$. 
Since $\nu \in \Pn$ ($\subset \Ldom{n}$), 
we see from \eqref{eq:nun} that 
%
%%%%%%%%%%%%%%%
%%% eq:con1 %%%
%%%%%%%%%%%%%%%
%
\begin{equation} \label{eq:con1}
\begin{cases}
0 \le \abs{\co{\nu}{0}} \le 
      \co{\nu}{1} \le \cdots \le \co{\nu}{u_{0}-1} < L_{\lambda}, \\[3mm]
\co{\nu}{u_{0}}= \co{\nu}{u_{0}+1}= \cdots = 
\co{\nu}{u_{1}-1}=L_{\lambda}, \\[3mm]
L_{\lambda} < \co{\nu}{u_{1}} \le 
\co{\nu}{u_{1}+1} \le \cdots \le \co{\nu}{n}, \\[3mm]
\co{\nu}{j}=L_{\lambda} \quad \text{for all $j \ge n+1$}.
\end{cases}
\end{equation}
for some $0 \le u_{0} < u_{1} \le n+1$, with 
\begin{equation*}
u_{1}-u_{0} > n-p-Nq \ge |\mu|=|\mu_{\dagger}|.
\end{equation*}
Furthermore, we deduce from Lemma~\ref{lem:Pn} that 
$\nuo \in \Po$ ($\subset \Ldom{n+1}$) is given by: 
%
%%%%%%%%%%%%%%%
%%% eq:con2 %%%
%%%%%%%%%%%%%%%
%
\begin{equation} \label{eq:con2}
\begin{cases}
\co{\nuo}{j}=\co{\nu}{j} \quad 
\text{\rm for $0 \le j \le u_{0}-1$}, \\[3mm]
\co{\nuo}{j}=\co{\nu}{j}=L_{\lambda} \quad 
\text{\rm for $u_{0} \le j \le u_{1}-1$}, \\[3mm]
\co{\nuo}{u_{1}}=L_{\lambda}, \\[3mm]
\co{\nuo}{j}=\co{\nu}{j-1} \quad 
\text{\rm for $u_{1}+1 \le j \le n+1$}, \\[3mm]
\co{\nuo}{j}=\co{\nu}{j}=L_{\lambda} \quad 
\text{\rm for $j \ge n+2$}.
\end{cases}
\end{equation}

%%%%%%%%%%
\paragraph{Case of type $C_{\infty}$.}
%%%%%%%%%%
%
Since $L_{\lambda} \in \BZ_{\ge 0}$, it follows that 
$\phi_{\ell}(\lambda)$, $\ell=n,\,n+1$, and 
$\phi_{n}(\nu)$, $\phi_{n+1}(\nuo)$ are all partitions. 
We infer from \eqref{eq:KT01} that 
%
%%%%%%%%%%%%%%%%
%%% eq:KT-C1 %%%
%%%%%%%%%%%%%%%%
%
\begin{equation} \label{eq:KT-C1}
\#\BBhiw{n}{\nu}=
\CLRm{\phi_{n}(\lambda)}{\mu_{\dagger}}{\phi_{n}(\nu)}{n}, \qquad
\#\BBhiw{n+1}{\nuo}=
\CLRm{\phi_{n+1}(\lambda)}{\mu_{\dagger}}{\phi_{n+1}(\nuo)}{n+1}.
\end{equation}
Here, we see from \eqref{eq:coll} that 
the first part of $\phi_{n}(\lambda)$ 
is equal to $L:=L_{\lambda}$, and 
\begin{align*}
& \#\bigl\{
  1 \le j \le n+1 \mid 
  \text{the $j$-th part of $\phi_{n}(\lambda)$ is equal to $L$}
  \bigr\}=n-p+1 > |\mu_{\dagger}|, \\
& \#\bigl\{
  1 \le j \le n+1 \mid 
  \text{the $j$-th part of $\phi_{n}(\nu)$ is equal to $L$}
  \bigr\}=u_{1}-u_{0} > |\mu_{\dagger}|.
\end{align*}
Also, we deduce from \eqref{eq:coll}, and \eqref{eq:con1}, \eqref{eq:con2} 
that $\phi_{n+1}(\lambda)=\iota_{L}(\phi_{n}(\lambda))$ and 
$\phi_{n+1}(\nuo)=\iota_{L}(\phi_{n}(\nu))$. 
Therefore, by Proposition~\ref{prop:stab}, we conclude that 
\begin{equation*}
\#\BBhiw{n}{\nu}=
\CLRm{\phi_{n}(\lambda)}{\mu_{\dagger}}{\phi_{n}(\nu)}{n}=
\CLRm{\phi_{n+1}(\lambda)}{\mu_{\dagger}}{\phi_{n+1}(\nuo)}{n+1}=
\#\BBhiw{n+1}{\nuo},
\end{equation*}
as desired. 

%%%%%%%%%%
\paragraph{Case of type $B_{\infty}$.}
%%%%%%%%%%
%
If $L_{\lambda} \in \BZ_{\ge 0}$, then we can show 
by the same reasoning as in the case of type $C_{\infty}$, 
this time using \eqref{eq:KT02a}, that 
\begin{equation*}
\#\BBhiw{n}{\nu}=
\BLRm{\phi_{n}(\lambda)}{\mu_{\dagger}}{\phi_{n}(\nu)}{n}=
\BLRm{\phi_{n+1}(\lambda)}{\mu_{\dagger}}{\phi_{n+1}(\nuo)}{n+1}=
\#\BBhiw{n+1}{\nuo}.
\end{equation*}
Therefore, it remains to consider the case in which 
the level $L_{\lambda}$ of $\lambda$ is 
contained in $1/2+\BZ_{\ge 0}$. In this case, 
we infer from \eqref{eq:KT02b} that 
%
%%%%%%%%%%%%%%%%
%%% eq:KT-B2 %%%
%%%%%%%%%%%%%%%%
%
\begin{equation} \label{eq:KT-B2}
\begin{split}
& 
\#\BBhiw{n}{\nu}=
\sum_{
 \begin{subarray}{c}
 \rho_{2} \in \Par, \ 
 \rho_{2} \subset \mu_{\dagger}, \\[1mm]
 \text{$\mu_{\dagger}/\rho_{2}$\,:\,vertical strip}
 \end{subarray}
}
\CLRm{\phi_{n}(\lambda-\Lambda_{0})}
     {\rho_{2}}{\phi_{n}(\nu-\Lambda_{0})}{n}, \\[3mm]
& 
\#\BBhiw{n+1}{\nuo}=
\sum_{
 \begin{subarray}{c}
 \rho_{2} \in \Par, \ 
 \rho_{2} \subset \mu_{\dagger}, \\[1mm]
 \text{$\mu_{\dagger}/\rho_{2}$\,:\,vertical strip}
 \end{subarray}
}
\CLRm{\phi_{n+1}(\lambda-\Lambda_{0})}
     {\rho_{2}}{\phi_{n+1}(\nuo-\Lambda_{0})}{n+1}.
\end{split}
\end{equation}
Here, observe that the first part of 
$\phi_{n}(\lambda-\Lambda_{0})$ 
is equal to $L:=L_{\lambda}-1/2 \in \BZ_{\ge 0}$, and
\begin{align*}
& \#\bigl\{
  1 \le j \le n+1 \mid 
  \text{the $j$-th part of $\phi_{n}(\lambda-\Lambda_{0})$ is equal to $L$}
  \bigr\} > |\mu_{\dagger}| \ge |\rho_{2}|, \\
& \#\bigl\{
  1 \le j \le n+1 \mid 
  \text{the $j$-th part of $\phi_{n}(\nu-\Lambda_{0})$ is equal to $L$}
  \bigr\} > |\mu_{\dagger}| \ge |\rho_{2}|
\end{align*}
for all $\rho_{2} \in \Par$ such that 
$\rho_{2} \subset \mu_{\dagger}$. 
Also, we deduce from \eqref{eq:coll}, and 
\eqref{eq:con1}, \eqref{eq:con2} that 
$\phi_{n+1}(\lambda-\Lambda_{0})=
 \iota_{L}(\phi_{n}(\lambda-\Lambda_{0}))$ and 
$\phi_{n+1}(\nuo-\Lambda_{0})=
 \iota_{L}(\phi_{n}(\nu-\Lambda_{0}))$. 
Therefore, by Proposition~\ref{prop:stab}, 
we conclude that 
\begin{equation*}
\CLRm{\phi_{n}(\lambda-\Lambda_{0})}
     {\rho_{2}}{\phi_{n}(\nu-\Lambda_{0})}{n}=
\CLRm{\phi_{n+1}(\lambda-\Lambda_{0})}
     {\rho_{2}}{\phi_{n+1}(\nuo-\Lambda_{0})}{n+1}
\end{equation*}
for all $\rho_{2} \in \Par$ such that 
$\rho_{2} \subset \mu_{\dagger}$, and hence that 
\begin{align*}
\#\BBhiw{n}{\nu} & =
\sum_{
 \begin{subarray}{c}
 \rho_{2} \in \Par, \ 
 \rho_{2} \subset \mu_{\dagger}, \\[1mm]
 \text{$\mu_{\dagger}/\rho_{2}$\,:\,vertical strip}
 \end{subarray}
}
\CLRm{\phi_{n}(\lambda-\Lambda_{0})}
     {\rho_{2}}{\phi_{n}(\nu-\Lambda_{0})}{n} \\[3mm]
& = 
\sum_{
 \begin{subarray}{c}
 \rho_{2} \in \Par, \ 
 \rho_{2} \subset \mu_{\dagger}, \\[1mm]
 \text{$\mu_{\dagger}/\rho_{2}$\,:\,vertical strip}
 \end{subarray}
}
\CLRm{\phi_{n+1}(\lambda-\Lambda_{0})}
     {\rho_{2}}{\phi_{n+1}(\nuo-\Lambda_{0})}{n+1}
= \#\BBhiw{n+1}{\nuo}, 
\end{align*}
as desired. 

%%%%%%%%%%
\paragraph{Case of type $D_{\infty}$.}
%%%%%%%%%%
%
Note that 
\begin{equation*}
\ti{\ch} \Vm{\ell}(\mum{\ell})=\ch \Vm{\ell}(\mum{\ell})
\quad \text{for $\ell=n,\,n+1$}
\end{equation*}
since $\co{\mum{\ell}}{0}=0$, 
and also that $\co{\nuo}{0}=\con{0}$ 
since $u_{1} > 0$ in \eqref{eq:con1} and \eqref{eq:con2}.

Suppose first that $L_{\lambda} \in \BZ_{\ge 0}$ and $\col{0}=0$; we have 
$\ti{\ch} \Vm{\ell}(\lambda)=\ch \Vm{\ell}(\lambda)$ 
for $\ell=n,\,n+1$ by definition. 
In this case, we infer from \eqref{eq:KT03} that 
%
%%%%%%%%%%%%%%%%
%%% eq:KT-D1 %%%
%%%%%%%%%%%%%%%%
%
\begin{equation} \label{eq:KT-D1}
\#\BBhiw{n}{\nu}=c(\nu) 
 \DLRm{\phi_{n}(\lambda)}{\mu_{\dagger}}{\phi_{n}(\nu)}{n} 
\qquad \text{and} \qquad
\#\BBhiw{n+1}{\nuo}=c(\nuo)
 \DLRm{\phi_{n+1}(\lambda)}{\mu_{\dagger}}{\phi_{n+1}(\nuo)}{n+1}.
\end{equation}
We can show by the same reasoning as 
in the case of type $C_{\infty}$ that 
\begin{equation*}
\DLRm{\phi_{n}(\lambda)}
     {\mu_{\dagger}}{\phi_{n}(\nu)}{n}=
\DLRm{\phi_{n+1}(\lambda)}
     {\mu_{\dagger}}{\phi_{n+1}(\nuo)}{n+1}.
\end{equation*}
Also, since $\co{\nuo}{0}=\con{0}$, 
it is obvious from the definitions that $c(\nuo)=c(\nu)$. 
Combining these facts with \eqref{eq:KT-D1}, we find that 
\begin{equation*}
\#\BBhiw{n}{\nun}=
\#\BBhiw{n+1}{\nuo}.
\end{equation*}

Suppose next that $L_{\lambda} \in \BZ_{\ge 0}$ and $\col{0} \ne 0$; we have 
$\ti{\ch} \Vm{\ell}(\lambda)=\ch^{(+)} \Vm{\ell}(\lambda)$ 
for $\ell=n,\,n+1$ by definition. 
If $\co{\nuo}{0}=\con{0}=0$, then we infer 
from \eqref{eq:KT03} and \eqref{eq:KT04} that 
%
%%%%%%%%%%%%%%%%
%%% eq:KT-D2 %%%
%%%%%%%%%%%%%%%%
%
\begin{equation} \label{eq:KT-D2}
\#\BBhiw{n}{\nu}=\frac{1}{2}
 \DLRm{\phi_{n}(\lambda)}{\mu_{\dagger}}{\phi_{n}(\nu)}{n} 
\qquad \text{and} \qquad
\#\BBhiw{n+1}{\nuo}=\frac{1}{2}
 \DLRm{\phi_{n+1}(\lambda)}{\mu_{\dagger}}{\phi_{n+1}(\nuo)}{n+1}. 
\end{equation}
Now, in exactly the same way as above, we can show that
\begin{equation*}
\DLRm{\phi_{n}(\lambda)}
     {\mu_{\dagger}}{\phi_{n}(\nu)}{n}=
\DLRm{\phi_{n+1}(\lambda)}
     {\mu_{\dagger}}{\phi_{n+1}(\nuo)}{n+1},
\quad \text{and hence} \quad
\#\BBhiw{n}{\nun}=
\#\BBhiw{n+1}{\nuo}. 
\end{equation*}
If $\co{\nuo}{0}=\co{\nu}{0} \ne 0$, then 
we infer from \eqref{eq:KT03} and \eqref{eq:KT04} that 
%
%%%%%%%%%%%%%%%%
%%% eq:KT-D3 %%%
%%%%%%%%%%%%%%%%
%
\begin{equation} \label{eq:KT-D3}
\#\BBhiw{n}{\nun}=\frac{1}{2}
\left(\frac{1}{2}\DLRm{\phi_{n}(\lambda)}{\mu_{\dagger}}{\phi_{n}(\nu)}{n}  + 
(-1)^{c}
\sum_{
 \begin{subarray}{c}
 \rho_{2},\,\omega \in \Par, \ 
 \rho_{2} \subset \omega \subset \mu_{\dagger}, \\[1mm]
 \text{$\mu_{\dagger}/\omega$\,:\,vertical strip}, \\[1mm]
 \text{$\omega/\rho_{2}$\,:\,vertical strip}
 \end{subarray}
}
(-1)^{|\omega/\rho_{2}|}
\CLRm{\rho_{1}}{\rho_{2}}{\rho}{n}
\right)
\end{equation}
and
%
%%%%%%%%%%%%%%%%
%%% eq:KT-D4 %%%
%%%%%%%%%%%%%%%%
%
\begin{equation} \label{eq:KT-D4}
\#\BBhiw{n+1}{\nuo}=\frac{1}{2}
\left(\frac{1}{2}
 \DLRm{\phi_{n+1}(\lambda)}{\mu_{\dagger}}{\phi_{n+1}(\nuo)}{n+1}+
(-1)^{c}
\sum_{
 \begin{subarray}{c}
 \rho_{2},\,\omega \in \Par, \ 
 \rho_{2} \subset \omega \subset \mu_{\dagger}, \\[1mm]
 \text{$\mu_{\dagger}/\omega$\,:\,vertical strip}, \\[1mm]
 \text{$\omega/\rho_{2}$\,:\,vertical strip}
 \end{subarray}
}
(-1)^{|\omega/\rho_{2}|}
\CLRm{\kappa_{1}}{\rho_{2}}{\kappa}{n+1}
\right),
\end{equation}
where we set
\begin{equation*}
c:=
 \begin{cases}
 0 & \text{if $\col{0}\co{\nuo}{0}=\col{0}\co{\nu}{0} > 0$}, \\[1.5mm]
 1 & \text{otherwise},
 \end{cases}
\end{equation*}
and
\begin{align*}
& \rho_{1}:=\phi_{n}(\lambda)-(1^{n+1}), & 
& \rho:=\phi_{n}(\nu)-(1^{n+1}), \\
& \kappa_{1}:=\phi_{n+1}(\lambda)-(1^{n+2}), & 
& \kappa:=\phi_{n+1}(\nuo)-(1^{n+2}).
\end{align*}
In exactly the same way as above, we can show that
\begin{equation*}
\DLRm{\phi_{n}(\lambda)}
     {\mu_{\dagger}}{\phi_{n}(\nu)}{n}=
\DLRm{\phi_{n+1}(\lambda)}
     {\mu_{\dagger}}{\phi_{n+1}(\nuo)}{n+1}.
\end{equation*}
Here, observe that 
the first part of $\rho_{1}$ is equal to 
$L-1=L_{\lambda}-1$, and 
\begin{align*}
& \#\bigl\{
  1 \le j \le n+1 \mid 
  \text{the $j$-th part of $\rho_{1}$ is equal to $L-1$}
  \bigr\} > |\mu_{\dagger}| > |\rho_{2}|, \\
& \#\bigl\{
  1 \le j \le n+1 \mid 
  \text{the $j$-th part of $\rho$ is equal to $L-1$}
  \bigr\} > |\mu_{\dagger}| > |\rho_{2}|
\end{align*}
for all $\rho_{2} \in \Par$ such that $\rho_{2} 
\subset \mu_{\dagger}$. 
Also, we deduce from \eqref{eq:coll}, and \eqref{eq:con1}, 
\eqref{eq:con2} that 
$\kappa_{1}=\iota_{L-1}(\rho_{1})$ and 
$\kappa=\iota_{L-1}(\rho)$. 
Therefore, by Proposition~\ref{prop:stab}, 
we conclude that 
\begin{equation*}
\CLRm{\rho_{1}}
     {\rho_{2}}{\rho}{n}=
\CLRm{\kappa_{1}}
     {\rho_{2}}{\kappa}{n+1}
\end{equation*}
for all $\rho_{2} \in \Par$ such that $\rho_{2} 
\subset \mu_{\dagger}$. Combining these facts with
\eqref{eq:KT-D3}, \eqref{eq:KT-D4}, we find that 
\begin{equation*}
\#\BBhiw{n}{\nun}=
\#\BBhiw{n+1}{\nuo}.
\end{equation*}

Finally, we consider the case in which 
the level $L_{\lambda}$ of $\lambda$ is contained in $1/2+\BZ_{\ge 0}$. 
In this case, we infer from \eqref{eq:KT05} and \eqref{eq:KT06} that 
%
%%%%%%%%%%%%%%%%
%%% eq:KT-D5 %%%
%%%%%%%%%%%%%%%%
%
\begin{equation} \label{eq:KT-D5}
\#\BBhiw{n}{\nu}=\frac{1}{2}
\left(
\sum_{
 \begin{subarray}{c}
 \rho_{2} \in \Par, \ 
 \rho_{2} \subset \mu_{\dagger}, \\[1mm]
 \text{$\mu_{\dagger}/\rho_{2}$\,:\,vertical strip}
 \end{subarray}
}
\bigl\{
(-1)^{|\mu_{\dagger}/\rho_{2}|}+
(-1)^{|\rho_{2}|+|\rho_{1}|+|\rho|+c}
\bigr\}
\BLRm{\rho_{1}}{\rho_{2}}{\rho}{n}
\right)
\end{equation}
and
%
%%%%%%%%%%%%%%%%
%%% eq:KT-D6 %%%
%%%%%%%%%%%%%%%%
%
\begin{equation} \label{eq:KT-D6}
\#\BBhiw{n+1}{\nuo}=\frac{1}{2}
\left(
\sum_{
 \begin{subarray}{c}
 \rho_{2} \in \Par, \ 
 \rho_{2} \subset \mu_{\dagger}, \\[1mm]
 \text{$\mu_{\dagger}/\rho_{2}$\,:\,vertical strip}
 \end{subarray}
}
\bigl\{
(-1)^{|\mu_{\dagger}/\rho_{2}|}+ 
(-1)^{|\rho_{2}|+|\kappa_{1}|+|\kappa|+c}
\bigr\}
\BLRm{\kappa_{1}}{\rho_{2}}{\kappa}{n+1}
\right), 
\end{equation}
where we set
\begin{equation*}
c:=
 \begin{cases}
 0 & \text{if $\col{0}\co{\nuo}{0}=\col{0}\co{\nu}{0} > 0$}, \\[1.5mm]
 1 & \text{otherwise},
 \end{cases}
\end{equation*}
and
\begin{align*}
& \rho_{1}:=\phi_{n}(\lambda-\Lambda_{0}), & 
& \rho:=\phi_{n}(\nu-\Lambda_{0}), \\
& \kappa_{1}:=\phi_{n+1}(\lambda-\Lambda_{0}), & 
& \kappa:=\phi_{n+1}(\nuo-\Lambda_{0}).
\end{align*}
Here, observe that the first part of $\rho_{1}$ 
is equal to $L:=L_{\lambda}-1/2$, and
\begin{align*}
& \#\bigl\{
  1 \le j \le n+1 \mid 
  \text{the $j$-th part of $\rho_{1}$ is equal to $L$}
  \bigr\} > |\mu_{\dagger}| \ge |\rho_{2}|, \\
& \#\bigl\{
  1 \le j \le n+1 \mid 
  \text{the $j$-th part of $\rho$ is equal to $L$}
  \bigr\} > |\mu_{\dagger}| \ge |\rho_{2}|
\end{align*}
for all $\rho_{2} \in \Par$ such that 
$\rho_{2} \subset \mu_{\dagger}$.
Also, we deduce from \eqref{eq:coll}, and 
\eqref{eq:con1}, \eqref{eq:con2} that 
$\kappa_{1}=\iota_{L}(\rho_{1})$ and 
$\kappa=\iota_{L}(\rho)$. 
Therefore, by Proposition~\ref{prop:stab}, 
we conclude that 
\begin{equation*}
\BLRm{\rho_{1}}{\rho_{2}}{\rho}{n}=
\BLRm{\kappa_{1}}{\rho_{2}}{\kappa}{n+1}
\end{equation*}
for all $\rho_{2} \in \Par$ such that 
$\rho_{2} \subset \mu_{\dagger}$. In addition, since 
$|\kappa_{1}|+|\kappa|=|\rho_{1}|+|\rho|+2(L_{\lambda}-1/2)$, 
we have $(-1)^{|\rho_{1}|+|\rho|}=(-1)^{|\kappa_{1}|+|\kappa|}$. 
Combining these facts with \eqref{eq:KT-D5}, \eqref{eq:KT-D6}, 
we find that 
\begin{equation*}
\#\BBhiw{n}{\nu}=
\#\BBhiw{n+1}{\nuo}. 
\end{equation*}
This completes the proof of Proposition~\ref{prop:DE-bij}.
\end{proof}

%==============================%
%     START SUBSECTION 0504    %
%==============================%
%
\subsection{Proof of Proposition~\ref{prop:stab}.}
\label{subsec:prf-key}

If $L=0$, then $\rho_{1}$ is
the empty partition $\emptyset$ by assumption (ii), 
and hence so is $\kappa_{1}=\rho_{L}(\rho_{1})=\rho_{1}$. 
In this case, we deduce from the definitions that 
\begin{equation*}
\XLRm{\emptyset}{\rho_{2}}{\rho}{n}=
 \begin{cases}
 1 & \text{if $\rho=\rho_{2}$}, \\[1.5mm]
 0 & \text{otherwise},
 \end{cases}
\qquad
\XLRm{\emptyset}{\rho_{2}}{\kappa}{n+1}=
 \begin{cases}
 1 & \text{if $\kappa=\rho_{2}$}, \\[1.5mm]
 0 & \text{otherwise}. 
 \end{cases}
\end{equation*}
Since $\kappa=\iota_{L}(\rho)=\rho$, we obtain 
$\XLRm{\rho_{1}}{\rho_{2}}{\rho}{n}=
 \XLRm{\emptyset}{\rho_{2}}{\rho}{n}=
 \XLRm{\emptyset}{\rho_{2}}{\kappa}{n+1}=
 \XLRm{\kappa_{1}}{\rho_{2}}{\kappa}{n}$, 
as desired. 

Assume, therefore, that $L > 0$. We set 
\begin{align*}
& \CQ(\rho_{2}):=
 \bigl\{(\omega_{2},\,\omega_{3}) \in \Par \times \Par \mid 
 \LR^{\rho_{2}}_{\omega_{2},\,\omega_{3}} \ne 0\bigr\}, \\
& 
\CR(\rho_{1},\,\omega_{2}):=
 \bigl\{\omega_{1} \in \Par \mid \omega_{1} \subset \rho_{1}, \ 
 |\rho_{1}|=|\omega_{1}|+|\omega_{2}|
 \bigr\} \quad \text{for $\omega_{2} \in \Par$}, \\
& 
\CR(\kappa_{1},\,\omega_{2}):=
 \bigl\{\omega_{1} \in \Par \mid \omega_{1} \subset \kappa_{1}, \ 
 |\kappa_{1}|=|\omega_{1}|+|\omega_{2}|
 \bigr\} \quad \text{for $\omega_{2} \in \Par$}.
\end{align*}
Then, from the definitions, we have 
%
%%%%%%%%%%%%%%%%%
%%% eq:step1a %%%
%%%%%%%%%%%%%%%%%
%
\begin{equation} \label{eq:step1a}
\XLRm{\rho_{1}}{\rho_{2}}{\rho}{n}=
\sum_{
 \begin{subarray}{c}
 J \subset \BZ_{\ge 1},\,\#J < \infty, \\[1mm]
 (\omega_{2},\,\omega_{3}) \in \CQ(\rho_{2}), \\[1mm]
 \omega_{1} \in \CR(\rho_{1},\,\omega_{2})
 \end{subarray}}
 \LR^{\rho_{1}}_{\omega_{1},\,\omega_{2}} \,
 \LR^{\rho_{2}}_{\omega_{2},\,\omega_{3}} \,
 \LR^{\Jn{\rho}}_{\omega_{3},\,\omega_{1}} \, 
 \Jn{\sgn}(\rho), 
\end{equation}
%
%%%%%%%%%%%%%%%%%
%%% eq:step1b %%%
%%%%%%%%%%%%%%%%%
%
\begin{equation} \label{eq:step1b}
\XLRm{\kappa_{1}}{\rho_{2}}{\kappa}{n+1}=
\sum_{
 \begin{subarray}{c}
 J \subset \BZ_{\ge 1},\,\#J < \infty, \\[1mm]
 (\omega_{2},\,\omega_{3}) \in \CQ(\rho_{2}), \\[1mm]
 \omega_{1} \in \CR(\kappa_{1},\,\omega_{2})
 \end{subarray}}
 \LR^{\kappa_{1}}_{\omega_{1},\,\omega_{2}} \,
 \LR^{\rho_{2}}_{\omega_{2},\,\omega_{3}} \,
 \LR^{\Jo{\kappa}}_{\omega_{3},\,\omega_{1}} \, 
 \Jo{\sgn}(\kappa).
\end{equation}
%
%%%%%%%%%%%
%%% c:J %%%
%%%%%%%%%%%
%
\begin{claim} \label{c:J}
{\rm (1)} Fix $(\omega_{2},\,\omega_{3}) \in \CQ(\rho_{2})$ and 
$\omega_{1} \in \CR(\rho_{1},\,\omega_{2})$. 
If $\LR^{\Jn{\rho}}_{\omega_{3},\,\omega_{1}} \ne 0$ 
for a finite subset $J$ of $\BZ_{\ge 1}$, 
then $J$ is contained in $[1,\,L]=\bigl\{1,\,2,\,\dots,\,L\bigr\}$. 
Also, if $J \subset [1,\,L]$, then the number of parts of $\Jn{\rho}$ 
that are equal to $L$ is greater than or equal to $y$. 

{\rm (2)} Fix $(\omega_{2},\,\omega_{3}) \in \CQ(\rho_{2})$ and 
$\omega_{1} \in \CR(\kappa_{1},\,\omega_{2})$. 
If $\LR^{\Jo{\kappa}}_{\omega_{3},\,\omega_{1}} \ne 0$ 
for a finite subset $J$ of $\BZ_{\ge 1}$, 
then $J$ is contained in $[1,\,L]=\bigl\{1,\,2,\,\dots,\,L\bigr\}$. 
\end{claim}

\noindent
{\it Proof of Claim~\ref{c:J}.}
We give a proof only for part (1), since the proof of part (2) is similar. 
Let $u \in \BZ_{\ge 1}$ be such that the $u$-th part of 
$\rho$ is equal to $L$, and such that 
the $(u+1)$-st part of $\rho$ is less than $L$; 
note that $u \ge y > |\rho_{2}|$ by assumption (iii). 
It follows that $u$ and $u-y$ are equal to the $L$-th part and 
the $(L+1)$-st part of the conjugate partition ${}^{t}\rho$ of $\rho$, 
respectively (see the figures below). 

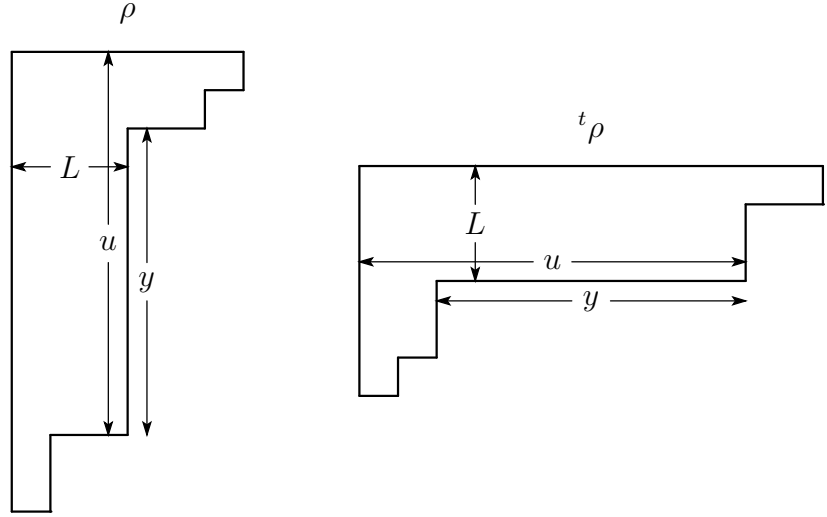
\vsp

\begin{center}
\hspace*{15mm}
%WinTpicVersion3.08
\unitlength 0.1in
\begin{picture}( 42.0000, 26.8500)(  8.0000,-30.0000)
% LINE 1 0 3 0
% 2 2000 800 1800 800
% 
\special{pn 13}%
\special{pa 2000 800}%
\special{pa 1800 800}%
\special{fp}%
% LINE 1 0 3 0
% 2 1800 800 1800 1000
% 
\special{pn 13}%
\special{pa 1800 800}%
\special{pa 1800 1000}%
\special{fp}%
% LINE 1 0 3 0
% 2 1800 1000 1400 1000
% 
\special{pn 13}%
\special{pa 1800 1000}%
\special{pa 1400 1000}%
\special{fp}%
% LINE 1 0 3 0
% 2 1400 1000 1400 2600
% 
\special{pn 13}%
\special{pa 1400 1000}%
\special{pa 1400 2600}%
\special{fp}%
% LINE 1 0 3 0
% 4 1400 2600 1000 2600 1000 3000 1000 3000
% 
\special{pn 13}%
\special{pa 1400 2600}%
\special{pa 1000 2600}%
\special{fp}%
\special{pa 1000 3000}%
\special{pa 1000 3000}%
\special{fp}%
% LINE 1 0 3 0
% 2 1000 3000 1000 2600
% 
\special{pn 13}%
\special{pa 1000 3000}%
\special{pa 1000 2600}%
\special{fp}%
% LINE 1 0 3 0
% 2 1000 3000 800 3000
% 
\special{pn 13}%
\special{pa 1000 3000}%
\special{pa 800 3000}%
\special{fp}%
% LINE 1 0 3 0
% 6 2000 800 2000 600 2000 600 800 600 800 600 800 3000
% 
\special{pn 13}%
\special{pa 2000 800}%
\special{pa 2000 600}%
\special{fp}%
\special{pa 2000 600}%
\special{pa 800 600}%
\special{fp}%
\special{pa 800 600}%
\special{pa 800 3000}%
\special{fp}%
% VECTOR 2 0 3 0
% 2 1200 1200 1400 1200
% 
\special{pn 8}%
\special{pa 1200 1200}%
\special{pa 1400 1200}%
\special{fp}%
\special{sh 1}%
\special{pa 1400 1200}%
\special{pa 1334 1180}%
\special{pa 1348 1200}%
\special{pa 1334 1220}%
\special{pa 1400 1200}%
\special{fp}%
% STR 2 0 3 0
% 3 1100 1100 1100 1200 5 0
% $L$
\put(11.0000,-12.0000){\makebox(0,0){$L$}}%
% VECTOR 2 0 3 0
% 2 1000 1200 800 1200
% 
\special{pn 8}%
\special{pa 1000 1200}%
\special{pa 800 1200}%
\special{fp}%
\special{sh 1}%
\special{pa 800 1200}%
\special{pa 868 1220}%
\special{pa 854 1200}%
\special{pa 868 1180}%
\special{pa 800 1200}%
\special{fp}%
% VECTOR 2 0 3 0
% 2 1300 1500 1300 600
% 
\special{pn 8}%
\special{pa 1300 1500}%
\special{pa 1300 600}%
\special{fp}%
\special{sh 1}%
\special{pa 1300 600}%
\special{pa 1280 668}%
\special{pa 1300 654}%
\special{pa 1320 668}%
\special{pa 1300 600}%
\special{fp}%
% VECTOR 2 0 3 0
% 2 1300 1700 1300 2600
% 
\special{pn 8}%
\special{pa 1300 1700}%
\special{pa 1300 2600}%
\special{fp}%
\special{sh 1}%
\special{pa 1300 2600}%
\special{pa 1320 2534}%
\special{pa 1300 2548}%
\special{pa 1280 2534}%
\special{pa 1300 2600}%
\special{fp}%
% STR 2 0 3 0
% 3 1300 1500 1300 1600 5 0
% $u$
\put(13.0000,-16.0000){\makebox(0,0){$u$}}%
% STR 2 0 3 0
% 3 1400 300 1400 400 5 0
% $\rho$
\put(14.0000,-4.0000){\makebox(0,0){$\rho$}}%
% LINE 1 0 3 0
% 2 2800 2395 2800 2195
% 
\special{pn 13}%
\special{pa 2800 2396}%
\special{pa 2800 2196}%
\special{fp}%
% LINE 1 0 3 0
% 2 2800 2195 3000 2195
% 
\special{pn 13}%
\special{pa 2800 2196}%
\special{pa 3000 2196}%
\special{fp}%
% LINE 1 0 3 0
% 2 3000 2195 3000 1795
% 
\special{pn 13}%
\special{pa 3000 2196}%
\special{pa 3000 1796}%
\special{fp}%
% LINE 1 0 3 0
% 2 3000 1795 4600 1795
% 
\special{pn 13}%
\special{pa 3000 1796}%
\special{pa 4600 1796}%
\special{fp}%
% LINE 1 0 3 0
% 4 4600 1795 4600 1395 5000 1395 5000 1395
% 
\special{pn 13}%
\special{pa 4600 1796}%
\special{pa 4600 1396}%
\special{fp}%
\special{pa 5000 1396}%
\special{pa 5000 1396}%
\special{fp}%
% LINE 1 0 3 0
% 2 5000 1395 4600 1395
% 
\special{pn 13}%
\special{pa 5000 1396}%
\special{pa 4600 1396}%
\special{fp}%
% LINE 1 0 3 0
% 2 5000 1395 5000 1195
% 
\special{pn 13}%
\special{pa 5000 1396}%
\special{pa 5000 1196}%
\special{fp}%
% LINE 1 0 3 0
% 6 2800 2395 2600 2395 2600 2395 2600 1195 2600 1195 5000 1195
% 
\special{pn 13}%
\special{pa 2800 2396}%
\special{pa 2600 2396}%
\special{fp}%
\special{pa 2600 2396}%
\special{pa 2600 1196}%
\special{fp}%
\special{pa 2600 1196}%
\special{pa 5000 1196}%
\special{fp}%
% VECTOR 2 0 3 0
% 2 3200 1595 3200 1795
% 
\special{pn 8}%
\special{pa 3200 1596}%
\special{pa 3200 1796}%
\special{fp}%
\special{sh 1}%
\special{pa 3200 1796}%
\special{pa 3220 1728}%
\special{pa 3200 1742}%
\special{pa 3180 1728}%
\special{pa 3200 1796}%
\special{fp}%
% VECTOR 2 0 3 0
% 2 3200 1395 3200 1195
% 
\special{pn 8}%
\special{pa 3200 1396}%
\special{pa 3200 1196}%
\special{fp}%
\special{sh 1}%
\special{pa 3200 1196}%
\special{pa 3180 1262}%
\special{pa 3200 1248}%
\special{pa 3220 1262}%
\special{pa 3200 1196}%
\special{fp}%
% VECTOR 2 0 3 0
% 2 3500 1695 2600 1695
% 
\special{pn 8}%
\special{pa 3500 1696}%
\special{pa 2600 1696}%
\special{fp}%
\special{sh 1}%
\special{pa 2600 1696}%
\special{pa 2668 1716}%
\special{pa 2654 1696}%
\special{pa 2668 1676}%
\special{pa 2600 1696}%
\special{fp}%
% VECTOR 2 0 3 0
% 2 3700 1695 4600 1695
% 
\special{pn 8}%
\special{pa 3700 1696}%
\special{pa 4600 1696}%
\special{fp}%
\special{sh 1}%
\special{pa 4600 1696}%
\special{pa 4534 1676}%
\special{pa 4548 1696}%
\special{pa 4534 1716}%
\special{pa 4600 1696}%
\special{fp}%
% STR 2 0 3 0
% 3 3200 1395 3200 1495 5 0
% $L$
\put(32.0000,-14.9500){\makebox(0,0){$L$}}%
% STR 2 0 3 0
% 3 3600 1595 3600 1695 5 0
% $u$
\put(36.0000,-16.9500){\makebox(0,0){$u$}}%
% STR 2 0 3 0
% 3 3800 895 3800 995 5 0
% ${}^{t}\rho$
\put(38.0000,-9.9500){\makebox(0,0){${}^{t}\rho$}}%
% VECTOR 2 0 3 0
% 2 1500 1700 1500 1000
% 
\special{pn 8}%
\special{pa 1500 1700}%
\special{pa 1500 1000}%
\special{fp}%
\special{sh 1}%
\special{pa 1500 1000}%
\special{pa 1480 1068}%
\special{pa 1500 1054}%
\special{pa 1520 1068}%
\special{pa 1500 1000}%
\special{fp}%
% VECTOR 2 0 3 0
% 2 1500 1900 1500 2600
% 
\special{pn 8}%
\special{pa 1500 1900}%
\special{pa 1500 2600}%
\special{fp}%
\special{sh 1}%
\special{pa 1500 2600}%
\special{pa 1520 2534}%
\special{pa 1500 2548}%
\special{pa 1480 2534}%
\special{pa 1500 2600}%
\special{fp}%
% STR 2 0 3 0
% 3 1500 1700 1500 1800 5 0
% $y$
\put(15.0000,-18.0000){\makebox(0,0){$y$}}%
% VECTOR 2 0 3 0
% 2 3700 1900 3000 1900
% 
\special{pn 8}%
\special{pa 3700 1900}%
\special{pa 3000 1900}%
\special{fp}%
\special{sh 1}%
\special{pa 3000 1900}%
\special{pa 3068 1920}%
\special{pa 3054 1900}%
\special{pa 3068 1880}%
\special{pa 3000 1900}%
\special{fp}%
% VECTOR 2 0 3 0
% 2 3900 1900 4600 1900
% 
\special{pn 8}%
\special{pa 3900 1900}%
\special{pa 4600 1900}%
\special{fp}%
\special{sh 1}%
\special{pa 4600 1900}%
\special{pa 4534 1880}%
\special{pa 4548 1900}%
\special{pa 4534 1920}%
\special{pa 4600 1900}%
\special{fp}%
% STR 2 0 3 0
% 3 3800 1800 3800 1900 5 0
% $y$
\put(38.0000,-19.0000){\makebox(0,0){$y$}}%
\end{picture}%
\end{center}

\noindent
Therefore, if we define $\bb({}^{t}\rho)=(b_{j})_{j \in \BZ_{\ge 1}}$ 
as in \eqref{eq:bb1}, then we have $b_{L}=u-(L-1)$, and hence 
\begin{equation*}
\#\bigl\{j \in \BZ_{\ge 1} \mid b_{j} \ge u-(L-1)\bigr\}=
\#[1,\,L]=L.
\end{equation*}
Also, if we define $\Jn{\bb({}^{t}\rho)}=
(\Jn{b_{j}})_{j \in \BZ_{\ge 1}}$ as in \eqref{eq:bb2}, 
then we have $\Jn{b_{j}} \ge n+1 \ge u-(L-1)$ for all $j \in J$ 
since $b_{L}=u-(L-1) \le b_{1} \le n+1$. Consequently, 
%
%%%%%%%%%%%%%
%%% eq:LJ %%%
%%%%%%%%%%%%%
%
\begin{equation} \label{eq:LJ}
\bigl\{j \in \BZ_{\ge 1} \mid \Jn{b_{j}} \ge u-(L-1)\bigr\}=
[1,\,L] \cup J.
\end{equation}
For the first assertion, 
suppose that $J$ is not contained in $[1,\,L]$. 
Then we get 
%
%%%%%%%%%%%%%%
%%% eq:cJ1 %%%
%%%%%%%%%%%%%%
%
\begin{equation} \label{eq:cJ1}
\# \bigl\{j \in \BZ_{\ge 1} \mid \Jn{b_{j}} \ge u-(L-1)\bigr\} \ge L+1.
\end{equation}
Let $\tau$ be the finite permutation of the set $\BZ_{\ge 1}$ 
such that 
\begin{equation*}
\Jn{b_{\tau(1)}} > \Jn{b_{\tau(2)}} > \cdots > 
\Jn{b_{\tau(j)}} > \Jn{b_{\tau(j+1)}} > \cdots. 
\end{equation*}
We deduce from \eqref{eq:cJ1} that 
$\Jn{b_{\tau(L+1)}} \ge u-(L-1)$, and hence that 
the $(L+1)$-st part $v$ of the conjugate partition
\begin{equation*}
(\Jn{b_{\tau(1)}} \ge \Jn{b_{\tau(2)}}+1 \ge \cdots \ge 
 \Jn{b_{\tau(j)}}+(j-1) \ge \Jn{b_{\tau(j+1)}}+j \ge \cdots)
\end{equation*}
of $\Jn{\rho}$ is greater than or 
equal to $u-(L-1)+L=u+1$. 
From this fact, we find that for each $1 \le j \le v$, 
the $j$-th row of the Young diagram of $\Jn{\rho}$ 
has more than $L+1$ boxes (see the figures below). 

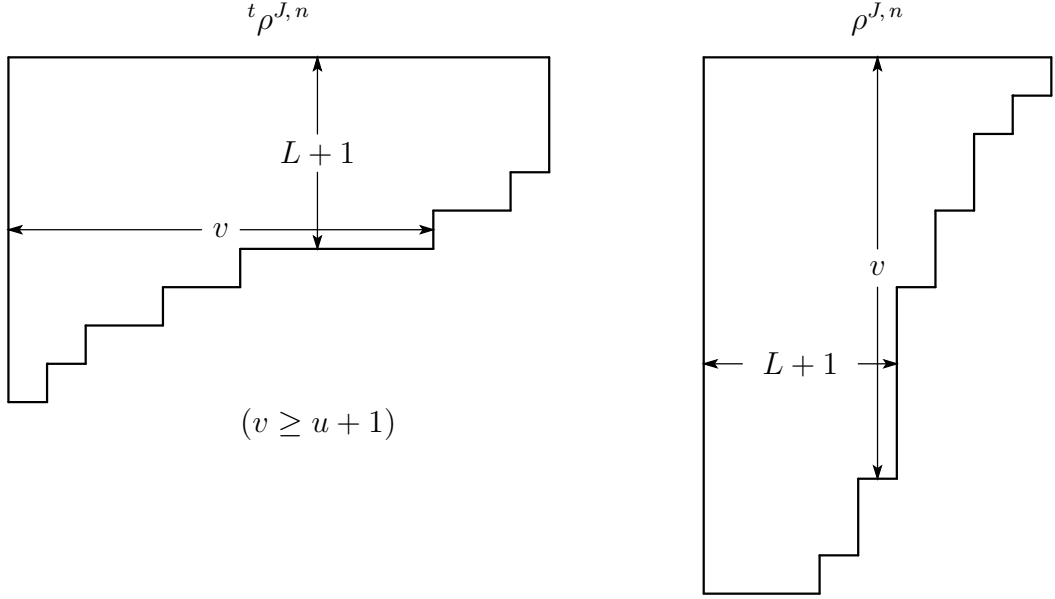
\begin{center}
%WinTpicVersion3.08
\unitlength 0.1in
\begin{picture}( 54.0000, 30.8500)(  4.0000,-34.0000)
% LINE 1 0 3 0
% 2 4800 3200 4800 2800
% 
\special{pn 13}%
\special{pa 4800 3200}%
\special{pa 4800 2800}%
\special{fp}%
% LINE 1 0 3 0
% 2 4800 3200 4600 3200
% 
\special{pn 13}%
\special{pa 4800 3200}%
\special{pa 4600 3200}%
\special{fp}%
% STR 2 0 3 0
% 3 1800 300 1800 400 5 0
% ${}^{t}\Jn{\rho}$
\put(18.0000,-4.0000){\makebox(0,0){${}^{t}\Jn{\rho}$}}%
% LINE 1 0 3 0
% 2 5800 800 5600 800
% 
\special{pn 13}%
\special{pa 5800 800}%
\special{pa 5600 800}%
\special{fp}%
% LINE 1 0 3 0
% 4 5600 800 5600 800 5600 1000 5600 800
% 
\special{pn 13}%
\special{pa 5600 800}%
\special{pa 5600 800}%
\special{fp}%
\special{pa 5600 1000}%
\special{pa 5600 800}%
\special{fp}%
% LINE 1 0 3 0
% 2 5600 1000 5400 1000
% 
\special{pn 13}%
\special{pa 5600 1000}%
\special{pa 5400 1000}%
\special{fp}%
% LINE 1 0 3 0
% 2 5800 800 5800 600
% 
\special{pn 13}%
\special{pa 5800 800}%
\special{pa 5800 600}%
\special{fp}%
% LINE 1 0 3 0
% 2 5000 1800 5000 2800
% 
\special{pn 13}%
\special{pa 5000 1800}%
\special{pa 5000 2800}%
\special{fp}%
% LINE 1 0 3 0
% 2 5000 1800 5200 1800
% 
\special{pn 13}%
\special{pa 5000 1800}%
\special{pa 5200 1800}%
\special{fp}%
% LINE 1 0 3 0
% 4 5200 1800 5200 1400 5400 1400 5200 1400
% 
\special{pn 13}%
\special{pa 5200 1800}%
\special{pa 5200 1400}%
\special{fp}%
\special{pa 5400 1400}%
\special{pa 5200 1400}%
\special{fp}%
% LINE 1 0 3 0
% 2 5400 1400 5400 1000
% 
\special{pn 13}%
\special{pa 5400 1400}%
\special{pa 5400 1000}%
\special{fp}%
% LINE 1 0 3 0
% 2 5000 2800 4800 2800
% 
\special{pn 13}%
\special{pa 5000 2800}%
\special{pa 4800 2800}%
\special{fp}%
% LINE 1 0 3 0
% 4 4600 3200 4600 3400 4600 3400 4000 3400
% 
\special{pn 13}%
\special{pa 4600 3200}%
\special{pa 4600 3400}%
\special{fp}%
\special{pa 4600 3400}%
\special{pa 4000 3400}%
\special{fp}%
% LINE 1 0 3 0
% 2 4000 3400 4000 600
% 
\special{pn 13}%
\special{pa 4000 3400}%
\special{pa 4000 600}%
\special{fp}%
% LINE 1 0 3 0
% 2 4000 600 5800 600
% 
\special{pn 13}%
\special{pa 4000 600}%
\special{pa 5800 600}%
\special{fp}%
% VECTOR 2 0 3 0
% 4 4900 1600 4900 600 4900 1800 4900 2800
% 
\special{pn 8}%
\special{pa 4900 1600}%
\special{pa 4900 600}%
\special{fp}%
\special{sh 1}%
\special{pa 4900 600}%
\special{pa 4880 668}%
\special{pa 4900 654}%
\special{pa 4920 668}%
\special{pa 4900 600}%
\special{fp}%
\special{pa 4900 1800}%
\special{pa 4900 2800}%
\special{fp}%
\special{sh 1}%
\special{pa 4900 2800}%
\special{pa 4920 2734}%
\special{pa 4900 2748}%
\special{pa 4880 2734}%
\special{pa 4900 2800}%
\special{fp}%
% LINE 1 0 3 0
% 2 3000 1400 2600 1400
% 
\special{pn 13}%
\special{pa 3000 1400}%
\special{pa 2600 1400}%
\special{fp}%
% LINE 1 0 3 0
% 2 3000 1400 3000 1200
% 
\special{pn 13}%
\special{pa 3000 1400}%
\special{pa 3000 1200}%
\special{fp}%
% LINE 1 0 3 0
% 2 600 2400 600 2200
% 
\special{pn 13}%
\special{pa 600 2400}%
\special{pa 600 2200}%
\special{fp}%
% LINE 1 0 3 0
% 4 600 2200 600 2200 800 2200 600 2200
% 
\special{pn 13}%
\special{pa 600 2200}%
\special{pa 600 2200}%
\special{fp}%
\special{pa 800 2200}%
\special{pa 600 2200}%
\special{fp}%
% LINE 1 0 3 0
% 2 800 2200 800 2000
% 
\special{pn 13}%
\special{pa 800 2200}%
\special{pa 800 2000}%
\special{fp}%
% LINE 1 0 3 0
% 2 600 2400 400 2400
% 
\special{pn 13}%
\special{pa 600 2400}%
\special{pa 400 2400}%
\special{fp}%
% LINE 1 0 3 0
% 2 1600 1600 2600 1600
% 
\special{pn 13}%
\special{pa 1600 1600}%
\special{pa 2600 1600}%
\special{fp}%
% LINE 1 0 3 0
% 2 1600 1600 1600 1800
% 
\special{pn 13}%
\special{pa 1600 1600}%
\special{pa 1600 1800}%
\special{fp}%
% LINE 1 0 3 0
% 4 1600 1800 1200 1800 1200 2000 1200 1800
% 
\special{pn 13}%
\special{pa 1600 1800}%
\special{pa 1200 1800}%
\special{fp}%
\special{pa 1200 2000}%
\special{pa 1200 1800}%
\special{fp}%
% LINE 1 0 3 0
% 2 1200 2000 800 2000
% 
\special{pn 13}%
\special{pa 1200 2000}%
\special{pa 800 2000}%
\special{fp}%
% LINE 1 0 3 0
% 2 2600 1600 2600 1400
% 
\special{pn 13}%
\special{pa 2600 1600}%
\special{pa 2600 1400}%
\special{fp}%
% LINE 1 0 3 0
% 4 3000 1200 3200 1200 3200 1200 3200 600
% 
\special{pn 13}%
\special{pa 3000 1200}%
\special{pa 3200 1200}%
\special{fp}%
\special{pa 3200 1200}%
\special{pa 3200 600}%
\special{fp}%
% LINE 1 0 3 0
% 2 3200 600 400 600
% 
\special{pn 13}%
\special{pa 3200 600}%
\special{pa 400 600}%
\special{fp}%
% LINE 1 0 3 0
% 2 400 600 400 2400
% 
\special{pn 13}%
\special{pa 400 600}%
\special{pa 400 2400}%
\special{fp}%
% VECTOR 2 0 3 0
% 4 2000 1000 2000 600 2000 1200 2000 1600
% 
\special{pn 8}%
\special{pa 2000 1000}%
\special{pa 2000 600}%
\special{fp}%
\special{sh 1}%
\special{pa 2000 600}%
\special{pa 1980 668}%
\special{pa 2000 654}%
\special{pa 2020 668}%
\special{pa 2000 600}%
\special{fp}%
\special{pa 2000 1200}%
\special{pa 2000 1600}%
\special{fp}%
\special{sh 1}%
\special{pa 2000 1600}%
\special{pa 2020 1534}%
\special{pa 2000 1548}%
\special{pa 1980 1534}%
\special{pa 2000 1600}%
\special{fp}%
% STR 2 0 3 0
% 3 2000 1000 2000 1100 5 0
% $L+1$
\put(20.0000,-11.0000){\makebox(0,0){$L+1$}}%
% VECTOR 2 0 3 0
% 4 1400 1500 400 1500 1600 1500 2600 1500
% 
\special{pn 8}%
\special{pa 1400 1500}%
\special{pa 400 1500}%
\special{fp}%
\special{sh 1}%
\special{pa 400 1500}%
\special{pa 468 1520}%
\special{pa 454 1500}%
\special{pa 468 1480}%
\special{pa 400 1500}%
\special{fp}%
\special{pa 1600 1500}%
\special{pa 2600 1500}%
\special{fp}%
\special{sh 1}%
\special{pa 2600 1500}%
\special{pa 2534 1480}%
\special{pa 2548 1500}%
\special{pa 2534 1520}%
\special{pa 2600 1500}%
\special{fp}%
% STR 2 0 3 0
% 3 1500 1400 1500 1500 5 0
% $v$
\put(15.0000,-15.0000){\makebox(0,0){$v$}}%
% STR 2 0 3 0
% 3 4500 2100 4500 2200 5 0
% $L+1$
\put(45.0000,-22.0000){\makebox(0,0){$L+1$}}%
% STR 2 0 3 0
% 3 4900 1600 4900 1700 5 0
% $v$
\put(49.0000,-17.0000){\makebox(0,0){$v$}}%
% STR 2 0 3 0
% 3 4900 300 4900 400 5 0
% $\Jn{\rho}$
\put(49.0000,-4.0000){\makebox(0,0){$\Jn{\rho}$}}%
% VECTOR 2 0 3 0
% 2 4200 2200 4000 2200
% 
\special{pn 8}%
\special{pa 4200 2200}%
\special{pa 4000 2200}%
\special{fp}%
\special{sh 1}%
\special{pa 4000 2200}%
\special{pa 4068 2220}%
\special{pa 4054 2200}%
\special{pa 4068 2180}%
\special{pa 4000 2200}%
\special{fp}%
% VECTOR 2 0 3 0
% 2 4800 2200 5000 2200
% 
\special{pn 8}%
\special{pa 4800 2200}%
\special{pa 5000 2200}%
\special{fp}%
\special{sh 1}%
\special{pa 5000 2200}%
\special{pa 4934 2180}%
\special{pa 4948 2200}%
\special{pa 4934 2220}%
\special{pa 5000 2200}%
\special{fp}%
% STR 2 0 3 0
% 3 1600 2500 1600 2600 2 0
% ($v \ge u+1$)
\put(16.0000,-26.0000){\makebox(0,0)[lb]{($v \ge u+1$)}}%
\end{picture}%
\end{center}

If $\omega_{1} \not\subset \Jn{\rho}$, then 
$\LR^{\Jn{\rho}}_{\omega_{3},\,\omega_{1}}=0$. 
Now, suppose that $\omega_{1} \subset \Jn{\rho}$. 
Since $\omega_{1} \in \CR(\rho_{1},\,\omega_{2})$, 
we have $\omega_{1} \subset \rho_{1}$, which implies that 
each row of the Young diagram of $\omega_{1}$ has 
at most $L$ boxes. Consequently, we deduce that 
$|\Jn{\rho}/\omega_{1}| \ge v \ge u+1 > 
|\rho_{2}| \ge |\omega_{3}|$, and hence 
$|\Jn{\rho}| > |\omega_{3}|+|\omega_{1}|$. 
Then, it follows that 
$\LR^{\Jn{\rho}}_{\omega_{3},\,\omega_{1}} = 0$, 
as desired. 

For the second assertion, 
suppose that $J \subset [1,\,L]$. 
As above, let $\tau$ be the finite permutation 
of the set $\BZ_{\ge 1}$ such that 
\begin{equation*}
\Jn{b_{\tau(1)}} > \Jn{b_{\tau(2)}} > \cdots > 
\Jn{b_{\tau(j)}} > \Jn{b_{\tau(j+1)}} > \cdots. 
\end{equation*}
Since 
\begin{equation*}
\bigl\{j \in \BZ_{\ge 1} \mid \Jn{b_{j}} \ge u-(L-1)\bigr\}=
[1,\,L] \cup J=[1,L]
\end{equation*}
by \eqref{eq:LJ}, we have $\Jn{b_{\tau(L)}} \ge u-(L-1)$. 
Also, since $J \subset [1,\,L]$, we see that 
$\tau(j)=j$ for all $j \ge L+1$, and hence 
$\Jn{b_{\tau(L+1)}}=\Jn{b_{L+1}}=b_{L+1}$. 
Recall that 
$b_{L+1}+L$ is equal to the $(L+1)$-st
part of ${}^{t}\rho$, which is equal to $u-y$. 
Therefore, we have 
\begin{align*}
\bigl\{\Jn{b_{\tau(L)}}+(L-1)\bigr\}-
\bigl\{\Jn{b_{\tau(L+1)}}+L\bigr\} & \ge 
\{u-(L-1)+(L-1)\}-(b_{L+1}+L) \\
& =u-(b_{L+1}+L)=y.
\end{align*}
Because $\Jn{b_{\tau(L)}}+(L-1)$ and 
$\Jn{b_{\tau(L+1)}}+L$ are the $L$-th part and 
the $(L+1)$-st part of ${}^{t}\Jn{\rho}$, respectively, 
we conclude that the Young diagram of $\Jn{\rho}$ 
has more than than $y$ rows having exactly $L$ boxes 
(see the figures below). 

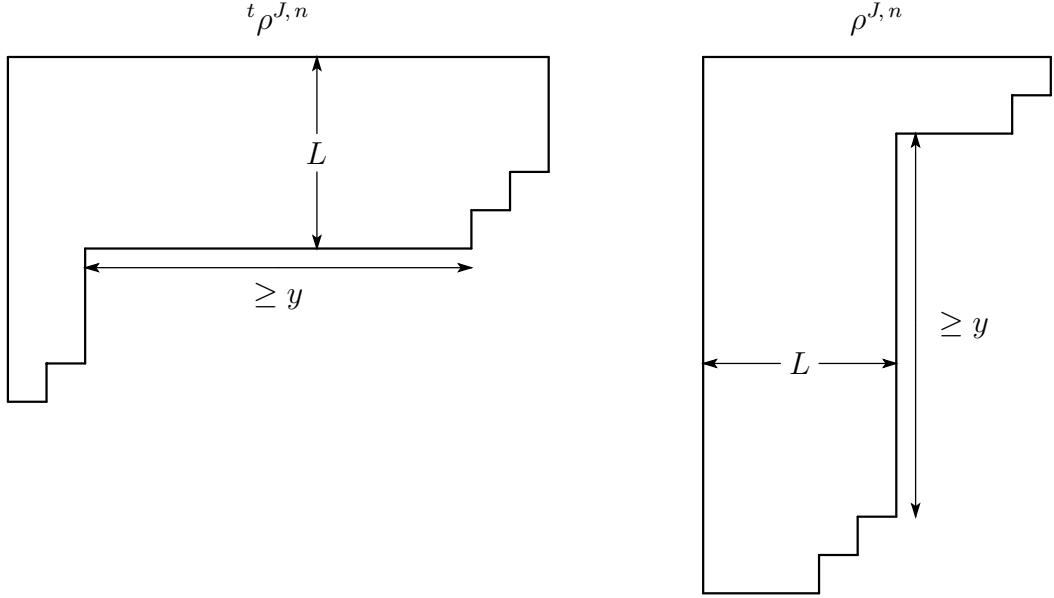
\vsp

\begin{center}
\hspace*{5mm}
%WinTpicVersion3.08
\unitlength 0.1in
\begin{picture}( 54.0000, 30.8500)(  4.0000,-34.0000)
% STR 2 0 3 0
% 3 1800 300 1800 400 5 0
% ${}^{t}\Jn{\rho}$
\put(18.0000,-4.0000){\makebox(0,0){${}^{t}\Jn{\rho}$}}%
% LINE 1 0 3 0
% 2 3000 1400 3000 1200
% 
\special{pn 13}%
\special{pa 3000 1400}%
\special{pa 3000 1200}%
\special{fp}%
% LINE 1 0 3 0
% 2 600 2400 600 2200
% 
\special{pn 13}%
\special{pa 600 2400}%
\special{pa 600 2200}%
\special{fp}%
% LINE 1 0 3 0
% 4 600 2200 600 2200 800 2200 600 2200
% 
\special{pn 13}%
\special{pa 600 2200}%
\special{pa 600 2200}%
\special{fp}%
\special{pa 800 2200}%
\special{pa 600 2200}%
\special{fp}%
% LINE 1 0 3 0
% 2 600 2400 400 2400
% 
\special{pn 13}%
\special{pa 600 2400}%
\special{pa 400 2400}%
\special{fp}%
% LINE 1 0 3 0
% 4 3000 1200 3200 1200 3200 1200 3200 600
% 
\special{pn 13}%
\special{pa 3000 1200}%
\special{pa 3200 1200}%
\special{fp}%
\special{pa 3200 1200}%
\special{pa 3200 600}%
\special{fp}%
% LINE 1 0 3 0
% 2 3200 600 400 600
% 
\special{pn 13}%
\special{pa 3200 600}%
\special{pa 400 600}%
\special{fp}%
% LINE 1 0 3 0
% 2 400 600 400 2400
% 
\special{pn 13}%
\special{pa 400 600}%
\special{pa 400 2400}%
\special{fp}%
% VECTOR 2 0 3 0
% 4 2000 1000 2000 600 2000 1200 2000 1600
% 
\special{pn 8}%
\special{pa 2000 1000}%
\special{pa 2000 600}%
\special{fp}%
\special{sh 1}%
\special{pa 2000 600}%
\special{pa 1980 668}%
\special{pa 2000 654}%
\special{pa 2020 668}%
\special{pa 2000 600}%
\special{fp}%
\special{pa 2000 1200}%
\special{pa 2000 1600}%
\special{fp}%
\special{sh 1}%
\special{pa 2000 1600}%
\special{pa 2020 1534}%
\special{pa 2000 1548}%
\special{pa 1980 1534}%
\special{pa 2000 1600}%
\special{fp}%
% STR 2 0 3 0
% 3 2000 1000 2000 1100 5 0
% $L$
\put(20.0000,-11.0000){\makebox(0,0){$L$}}%
% STR 2 0 3 0
% 3 4900 300 4900 400 5 0
% $\Jn{\rho}$
\put(49.0000,-4.0000){\makebox(0,0){$\Jn{\rho}$}}%
% LINE 1 0 3 0
% 4 3000 1400 2800 1400 2800 1400 2800 1600
% 
\special{pn 13}%
\special{pa 3000 1400}%
\special{pa 2800 1400}%
\special{fp}%
\special{pa 2800 1400}%
\special{pa 2800 1600}%
\special{fp}%
% LINE 1 0 3 0
% 2 2800 1600 800 1600
% 
\special{pn 13}%
\special{pa 2800 1600}%
\special{pa 800 1600}%
\special{fp}%
% LINE 1 0 3 0
% 2 800 1600 800 2200
% 
\special{pn 13}%
\special{pa 800 1600}%
\special{pa 800 2200}%
\special{fp}%
% VECTOR 2 0 3 0
% 4 1800 1700 800 1700 1800 1700 2800 1700
% 
\special{pn 8}%
\special{pa 1800 1700}%
\special{pa 800 1700}%
\special{fp}%
\special{sh 1}%
\special{pa 800 1700}%
\special{pa 868 1720}%
\special{pa 854 1700}%
\special{pa 868 1680}%
\special{pa 800 1700}%
\special{fp}%
\special{pa 1800 1700}%
\special{pa 2800 1700}%
\special{fp}%
\special{sh 1}%
\special{pa 2800 1700}%
\special{pa 2734 1680}%
\special{pa 2748 1700}%
\special{pa 2734 1720}%
\special{pa 2800 1700}%
\special{fp}%
% STR 2 0 3 0
% 3 1800 1750 1800 1850 5 0
% $\ge y$
\put(18.0000,-18.5000){\makebox(0,0){$\ge y$}}%
% LINE 1 0 3 0
% 2 4800 3200 4600 3200
% 
\special{pn 13}%
\special{pa 4800 3200}%
\special{pa 4600 3200}%
\special{fp}%
% LINE 1 0 3 0
% 2 5800 800 5600 800
% 
\special{pn 13}%
\special{pa 5800 800}%
\special{pa 5600 800}%
\special{fp}%
% LINE 1 0 3 0
% 4 5600 800 5600 800 5600 1000 5600 800
% 
\special{pn 13}%
\special{pa 5600 800}%
\special{pa 5600 800}%
\special{fp}%
\special{pa 5600 1000}%
\special{pa 5600 800}%
\special{fp}%
% LINE 1 0 3 0
% 2 5800 800 5800 600
% 
\special{pn 13}%
\special{pa 5800 800}%
\special{pa 5800 600}%
\special{fp}%
% LINE 1 0 3 0
% 4 4600 3200 4600 3400 4600 3400 4000 3400
% 
\special{pn 13}%
\special{pa 4600 3200}%
\special{pa 4600 3400}%
\special{fp}%
\special{pa 4600 3400}%
\special{pa 4000 3400}%
\special{fp}%
% LINE 1 0 3 0
% 2 4000 3400 4000 600
% 
\special{pn 13}%
\special{pa 4000 3400}%
\special{pa 4000 600}%
\special{fp}%
% LINE 1 0 3 0
% 2 4000 600 5800 600
% 
\special{pn 13}%
\special{pa 4000 600}%
\special{pa 5800 600}%
\special{fp}%
% VECTOR 2 0 3 0
% 4 4400 2200 4000 2200 4600 2200 5000 2200
% 
\special{pn 8}%
\special{pa 4400 2200}%
\special{pa 4000 2200}%
\special{fp}%
\special{sh 1}%
\special{pa 4000 2200}%
\special{pa 4068 2220}%
\special{pa 4054 2200}%
\special{pa 4068 2180}%
\special{pa 4000 2200}%
\special{fp}%
\special{pa 4600 2200}%
\special{pa 5000 2200}%
\special{fp}%
\special{sh 1}%
\special{pa 5000 2200}%
\special{pa 4934 2180}%
\special{pa 4948 2200}%
\special{pa 4934 2220}%
\special{pa 5000 2200}%
\special{fp}%
% LINE 1 0 3 0
% 4 4800 3200 4800 3000 4800 3000 5000 3000
% 
\special{pn 13}%
\special{pa 4800 3200}%
\special{pa 4800 3000}%
\special{fp}%
\special{pa 4800 3000}%
\special{pa 5000 3000}%
\special{fp}%
% LINE 1 0 3 0
% 2 5000 3000 5000 1000
% 
\special{pn 13}%
\special{pa 5000 3000}%
\special{pa 5000 1000}%
\special{fp}%
% LINE 1 0 3 0
% 2 5000 1000 5600 1000
% 
\special{pn 13}%
\special{pa 5000 1000}%
\special{pa 5600 1000}%
\special{fp}%
% VECTOR 2 0 3 0
% 4 5100 2000 5100 1000 5100 2000 5100 3000
% 
\special{pn 8}%
\special{pa 5100 2000}%
\special{pa 5100 1000}%
\special{fp}%
\special{sh 1}%
\special{pa 5100 1000}%
\special{pa 5080 1068}%
\special{pa 5100 1054}%
\special{pa 5120 1068}%
\special{pa 5100 1000}%
\special{fp}%
\special{pa 5100 2000}%
\special{pa 5100 3000}%
\special{fp}%
\special{sh 1}%
\special{pa 5100 3000}%
\special{pa 5120 2934}%
\special{pa 5100 2948}%
\special{pa 5080 2934}%
\special{pa 5100 3000}%
\special{fp}%
% STR 2 0 3 0
% 3 5350 1900 5350 2000 5 0
% $\ge y$
\put(53.5000,-20.0000){\makebox(0,0){$\ge y$}}%
% STR 2 0 3 0
% 3 4500 2100 4500 2200 5 0
% $L$
\put(45.0000,-22.0000){\makebox(0,0){$L$}}%
\end{picture}%
\end{center}

\noindent
This proves the claim. \bqed

%%%%
\vsp
%%%%

By using Claim~\ref{c:J}, we deduce 
from \eqref{eq:step1a} and \eqref{eq:step1b} that 
%
%%%%%%%%%%%%%%%%%
%%% eq:step2a %%%
%%%%%%%%%%%%%%%%%
%
\begin{equation} \label{eq:step2a}
\XLRm{\rho_{1}}{\rho_{2}}{\rho}{n}=
\sum_{
 \begin{subarray}{c}
 J \subset [1,\,L], \\[1mm]
 (\omega_{2},\,\omega_{3}) \in \CQ(\rho_{2}), \\[1mm]
 \omega_{1} \in \CR(\rho_{1},\,\omega_{2})
 \end{subarray}}
 \LR^{\rho_{1}}_{\omega_{1},\,\omega_{2}} \,
 \LR^{\rho_{2}}_{\omega_{2},\,\omega_{3}} \,
 \LR^{\Jn{\rho}}_{\omega_{3},\,\omega_{1}} \, 
 \Jn{\sgn}(\rho), 
\end{equation}
%
%%%%%%%%%%%%%%%%%
%%% eq:step2b %%%
%%%%%%%%%%%%%%%%%
%
\begin{equation} \label{eq:step2b}
\XLRm{\kappa_{1}}{\rho_{2}}{\kappa}{n+1}=
\sum_{
 \begin{subarray}{c}
 J \subset [1,\,L], \\[1mm]
 (\omega_{2},\,\omega_{3}) \in \CQ(\rho_{2}), \\[1mm]
 \omega_{1} \in \CR(\kappa_{1},\,\omega_{2})
 \end{subarray}}
 \LR^{\kappa_{1}}_{\omega_{1},\,\omega_{2}} \,
 \LR^{\rho_{2}}_{\omega_{2},\,\omega_{3}} \,
 \LR^{\Jo{\kappa}}_{\omega_{3},\,\omega_{1}} \, 
 \Jo{\sgn}(\kappa).
\end{equation}

%
%%%%%%%%%%%%%%
%%% c:iota %%%
%%%%%%%%%%%%%%
%
\begin{claim} \label{c:iota}
Fix $(\omega_{2},\,\omega_{3}) \in \CQ(\rho_{2})$. 
For every $\omega_{1} \in \CR(\rho_{1},\,\omega_{2})$, 
we have $\iota_{L}(\omega_{1}) \in \CR(\kappa_{1},\,\omega_{2})$. 
Thus, $\iota_{L}$ yields a map from $\CR(\rho_{1},\,\omega_{2})$ 
to $\CR(\kappa_{1},\,\omega_{2})$. Moreover, this map is bijective. 
\end{claim}

\noindent
{\it Proof of Claim~\ref{c:iota}.}
Recall that 
\begin{align*}
& \CR(\rho_{1},\,\omega_{2})=
 \bigl\{\omega_{1} \in \Par \mid \omega_{1} \subset \rho_{1}, \ 
 |\rho_{1}|=|\omega_{1}|+|\omega_{2}|
 \bigr\}, \\
& \CR(\kappa_{1},\,\omega_{2})=
 \bigl\{\omega_{1} \in \Par \mid \omega_{1} \subset \kappa_{1}, \ 
 |\kappa_{1}|=|\omega_{1}|+|\omega_{2}|
 \bigr\}.
\end{align*}
Let $\omega_{1} \in \CR(\rho_{1},\,\omega_{2})$. 
Since the first part of $\rho_{1}$ is equal to $L$ by assumption (ii), 
we see that the first part of $\omega_{1}$ is less than or equal to $L$. 
Therefore, the Young diagrams of $\kappa_{1}=\iota_{L}(\rho_{1})$ 
(resp., $\iota_{L}(\omega_{1})$) 
is obtained by adding one row having exactly $L$ boxes just above 
the top row of the Young diagram of $\rho_{1}$ (resp., $\omega_{1}$). 
Hence it is obvious that 
$\iota_{L}(\omega_{1}) \in \CR(\kappa_{1},\,\omega_{2})$. 

We need to show that the map 
$\iota_{L}:\CR(\rho_{1},\,\omega_{2}) \rightarrow 
\CR(\kappa_{1},\,\omega_{2})$ is bijective. 
Since the injectivity of the map $\iota_{L}$ is obvious, 
it remains to show the surjectivity of 
the map $\iota_{L}$. 
If $\delta_{1} \in \CR(\kappa_{1},\,\omega_{2})$, then 
there exists a part of $\delta_{1}$ that is equal to $L$.
Indeed, since $\delta_{1} \subset \kappa_{1}=\iota_{L}(\rho_{1})$, 
all parts of $\delta_{1}$ are less than or equal to $L$. 
Suppose, contrary to our assertion, that all parts of $\delta_{1}$ are
less than $L$. Then the skew Young diagram $\kappa_{1}/\delta_{1}$ 
must contain at least $y_{1}+1$ boxes, where $y_{1}+1$ is equal to 
the number of rows of the Young diagram of $\kappa_{1}$ having 
exactly $L$ boxes since $\kappa_{1}=\iota_{L}(\rho_{1})$. 
Because $y_{1}+1$ is greater than $|\rho_{2}|$ by assumption (iii), and 
because $(\omega_{2},\,\omega_{3}) \in \CQ(\rho_{2})$, 
we deduce that 
$|\kappa_{1}/\delta_{1}| \ge y_{1}+1 > |\rho_{2}| \ge |\omega_{2}|$, 
and hence $|\kappa_{1}| > |\delta_{1}|+|\omega_{2}|$, 
which contradicts the assumption that 
$\delta_{1} \in \CR(\kappa_{1},\,\omega_{2})$. 
Therefore, if we let $\omega_{1}$ be the partition 
whose Young diagram is obtained 
from the Young diagram of $\delta_{1}$ by removing one row 
having exactly $L$ boxes, then we have 
$\omega_{1} \in \CR(\rho_{1},\,\omega_{2})$ and 
$\iota_{L}(\omega_{1})=\delta_{1}$. 
This proves the surjectivity of the map $\iota_{L}$, 
as desired. \bqed

%%%%
\vsp
%%%%

%%%%%%%%%%%%
%%% c:sp %%%
%%%%%%%%%%%%
%
\begin{claim} \label{c:sp}
For each $J \subset [1,\,L]$, 
we have $\iota_{L}(\Jn{\rho})=\Jo{\kappa}$ and 
$\Jn{\sgn}(\rho)=\Jo{\sgn}(\kappa)$. 
\end{claim}

\noindent
{\it Proof of Claim~\ref{c:sp}.}
Since $\kappa=\iota_{L}(\rho)$, we deduce that 
\begin{equation*}
{}^{t}\kappa={}^{t}\rho+
(\underbrace{1,\,1,\,\dots,\,1}_{\text{$L$ times}},\,0,\,0,\,\dots).
\end{equation*}
Consequently, if we define the sequences 
$\bb({}^{t}\kappa)=(c_{j})_{j \in \BZ_{\ge 1}}$ and 
$\bb({}^{t}\rho)=(b_{j})_{j \in \BZ_{\ge 1}}$ 
as in \eqref{eq:bb1}, then
\begin{equation*}
\bb({}^{t}\kappa)=\bb({}^{t}\rho)+
(\underbrace{1,\,1,\,\dots,\,1}_{\text{$L$ times}},\,0,\,0,\,\dots). 
\end{equation*}
Furthermore, if we define the sequences 
$\Jo{\bb({}^{t}\kappa)}=(\Jo{c_{j}})_{j \in \BZ_{\ge 1}}$ and 
$\Jn{\bb({}^{t}\rho)}=(\Jn{b_{j}})_{j \in \BZ_{\ge 1}}$ 
as in \eqref{eq:bb2}, then
\begin{align*}
\Jo{c_{j}} & =c_{j}=b_{j}+1=\Jn{b_{j}}+1
\qquad \text{for $j \in [1,\,L] \setminus J$}; \\[1.5mm]
\Jo{c_{j}} & =
 R_{n+1}-c_{j}=R_{n}+2-(b_{j}+1) \\
& = R_{n}-b_{j}+1=\Jn{b_{j}}+1
\qquad \text{for $j \in J \subset [1,\,L]$}; \\[1.5mm]
\Jo{c_{j}} & =c_{j}=b_{j}=\Jn{b_{j}}
\qquad \text{for $j > L$}.
\end{align*}
These equations imply that 
%
%%%%%%%%%%%%%%%
%%% eq:kr+1 %%%
%%%%%%%%%%%%%%%
%
\begin{equation} \label{eq:kr+1}
\Jo{\bb({}^{t}\kappa)}=\Jn{\bb({}^{t}\rho)}+
(\underbrace{1,\,1,\,\dots,\,1}_{\text{$L$ times}},\,0,\,0,\,\dots).
\end{equation}
Let $\tau$ be the finite permutation of the set $\BZ_{\ge 1}$ such that 
\begin{equation*}
\Jo{c_{\tau(1)}} > \Jo{c_{\tau(2)}} > \cdots > 
\Jo{c_{\tau(j)}} > \Jo{c_{\tau(j+1)}} > \cdots; 
\end{equation*}
note that $\tau(j)=j$ for all $j \ge L+1$ 
since $J \subset [1,\,L]$. 
It follows from \eqref{eq:kr+1} that 
\begin{equation*}
\Jn{b_{\tau(1)}} > \Jn{b_{\tau(2)}} > \cdots > 
\Jn{b_{\tau(j)}} > \Jn{b_{\tau(j+1)}} > \cdots, 
\end{equation*}
and hence that $\Jo{\sgn}(\kappa)=\Jn{\sgn}(\rho)$. 
Also, it follows from the definitions that 
\begin{equation*}
\Jo{\kappa}=
\Jn{\rho}+
(\underbrace{1,\,1,\,\dots,\,1}_{\text{$L$ times}},\,0,\,0,\,\dots),
\end{equation*}
which means that 
$\iota_{L}(\Jn{\rho})=\Jo{\kappa}$. 
This proves the claim. \bqed

%%%%
\vsp
%%%%

Using Claims~\ref{c:iota} and \ref{c:sp}, 
we deduce from \eqref{eq:step2b} that 
%
%%%%%%%%%%%%%%%%
%%% eq:step3 %%%
%%%%%%%%%%%%%%%%
%
\begin{align}
\XLRm{\kappa_{1}}{\rho_{2}}{\kappa}{n+1} & =
\sum_{
 \begin{subarray}{c}
 J \subset [1,\,L], \\[1mm]
 (\omega_{2},\,\omega_{3}) \in \CQ(\rho_{2}), \\[1mm]
 \omega_{1} \in \CR(\kappa_{1},\,\omega_{2})
 \end{subarray}}
 \LR^{\kappa_{1}}_{\omega_{1},\,\omega_{2}} \,
 \LR^{\rho_{2}}_{\omega_{2},\,\omega_{3}} \,
 \LR^{\Jo{\kappa}}_{\omega_{3},\,\omega_{1}} \, 
 \Jo{\sgn}(\kappa) \nonumber  \\[5mm]
& = 
\sum_{
 \begin{subarray}{c}
 J \subset [1,\,L], \\[1mm]
 (\omega_{2},\,\omega_{3}) \in \CQ(\rho_{2}), \\[1mm]
 \omega_{1} \in \CR(\rho_{1},\,\omega_{2})
 \end{subarray}}
 \LR^{\iota_{L}(\rho_{1})}_{\iota_{L}(\omega_{1}),\,\omega_{2}} \,
 \LR^{\rho_{2}}_{\omega_{2},\,\omega_{3}} \,
 \LR^{\iota_{L}(\Jn{\rho})}_{\omega_{3},\,\iota_{L}(\omega_{1})} \, 
 \Jn{\sgn}(\rho). \label{eq:step3}
\end{align}
First, observe that 
for each $J \subset [1,\,L]$, 
$(\omega_{2},\,\omega_{3}) \in \CQ(\rho_{2})$, and 
$\omega_{1} \in \CR(\rho_{1},\,\omega_{2})$, there holds 
%
%%%%%%%%%%%%%%
%%% eq:LR1 %%%
%%%%%%%%%%%%%%
%
\begin{equation} \label{eq:LR1}
\LR^{\iota_{L}(\rho_{1})}_{\iota_{L}(\omega_{1}),\,\omega_{2}}=
\LR^{\rho_{1}}_{\omega_{1},\,\omega_{2}}.
\end{equation}
Indeed, since $\omega_{1} \subset \rho_{1}$ by the definition of 
$\CR(\rho_{1},\,\omega_{2})$, it follows from assumption (ii) that 
the Young diagram of $\iota_{L}(\rho_{1})$ (resp., $\iota_{L}(\omega_{1})$)
is obtained by adding one row having exactly $L$ boxes 
just above the top row of the Young diagram of 
$\rho_{1}$ (resp., $\omega_{1}$). 
This implies that the skew Young diagram 
$\iota_{L}(\rho_{1})/\iota_{L}(\omega_{1})$ is identical to 
the skew Young diagram $\rho_{1}/\omega_{1}$. 
Therefore, by the Littlewood-Richardson rule 
(see, for example, \cite[Chapter~5, Section~2, Proposition~3]{Fulton} 
or \cite[\S2]{Ko98}), we obtain \eqref{eq:LR1}. 
Next, we show that for each $J \subset [1,\,L]$, 
$(\omega_{2},\,\omega_{3}) \in \CQ(\rho_{2})$,  and 
$\omega_{1} \in \CR(\rho_{1},\,\omega_{2})$, there holds 
%
%%%%%%%%%%%%%%
%%% eq:LR2 %%%
%%%%%%%%%%%%%%
%
\begin{equation} \label{eq:LR2}
\LR^{\iota_{L}(\Jn{\rho})}_{\omega_{3},\,\iota_{L}(\omega_{1})}=
\LR^{\Jn{\rho}}_{\omega_{3},\,\omega_{1}}.
\end{equation}
Suppose that $\omega_{1} \not\subset \Jn{\rho}$.
Let $j \in \BZ_{\ge 1}$ be such that 
the $j$-th part of $\omega_{1}$ is 
greater than the $j$-th path of $\Jn{\rho}$. 
Since $\omega_{1} \subset \rho_{1}$ by the definition of 
$\CR(\rho_{1},\,\omega_{2})$, it follows from assumption (ii) that 
\begin{equation*}
L \ge \text{the $j$-th part of $\omega_{1}$} 
  > \text{the $j$-th part of $\Jn{\rho}$}, 
\end{equation*}
and hence that 
\begin{align*}
& \text{the $(j+1)$-st part of $\iota_{L}(\omega_{1})$}=
  \text{the $j$-th part of $\omega_{1}$} \\
& \qquad > 
  \text{the $j$-th part of $\Jn{\rho}$}=
  \text{the $(j+1)$-st part of $\iota_{L}(\Jn{\rho})$}. 
\end{align*}
This implies that 
$\iota_{L}(\omega_{1}) \not\subset \iota_{L}(\Jn{\rho})$.
Thus, in this case, we obtain 
$\LR^{\iota_{L}(\Jn{\rho})}_{\omega_{3},\,\iota_{L}(\omega_{1})}=0=
\LR^{\Jn{\rho}}_{\omega_{3},\,\omega_{1}}$. 
Also, if $|\Jn{\rho}| \ne |\omega_{1}|+|\omega_{3}|$, then 
it is obvious that 
$|\iota_{L}(\Jn{\rho})| \ne |\iota_{L}(\omega_{1})|+|\omega_{3}|$, 
and hence 
$\LR^{\iota_{L}(\Jn{\rho})}_{\omega_{3},\,\iota_{L}(\omega_{1})}=0=
\LR^{\Jn{\rho}}_{\omega_{3},\,\omega_{1}}$. 
We may, therefore, assume that 
\begin{equation*}
\text{
$\omega_{1} \subset \Jn{\rho}$ and 
$|\Jn{\rho}|=|\omega_{1}|+|\omega_{3}|$}.
\end{equation*}
In this case, there exists 
$j_{0} \in \BZ_{\ge 1}$ such that 
the $j_{0}$-th parts of $\Jn{\rho}$ and 
$\omega_{1}$ are both equal to $L$. 
Indeed, it follows from Claim~\ref{c:J}\,(1) that 
the number of parts of $\Jn{\rho}$ that are equal to $L$ is 
greater than or equal to $y$. 
Let $j_{1},\,j_{2} \in \BZ_{\ge 1}$ be such that 
the $j$-th part of $\Jn{\rho}$ is equal to $L$ 
for all $j_{1}+1 \le j \le j_{2}$, and $j_{2}-j_{1} \ge y$. 
Suppose, contrary to our assertion, 
that the $j$-th part of $\omega_{1}$ is less than 
$L$ for all $j_{1}+1 \le j \le j_{2}$; note that all parts of 
$\omega_{1} \in \CR(\rho_{1},\,\omega_{2})$ are at most $L$. 
Then we must have 
\begin{equation*}
|\Jn{\rho}/\omega_{1}| \ge j_{2}-j_{1} \ge y > 
|\rho_{2}|=|\omega_{2}|+|\omega_{3}| \ge |\omega_{3}|, 
\end{equation*}
and hence $|\Jn{\rho}| > |\omega_{1}|+|\omega_{3}|$, 
which contradicts our assumption above. 
Consequently, the Young diagram of 
$\iota_{L}(\Jn{\rho})$ (resp., $\iota_{L}(\omega_{1})$) is 
obtained by inserting one row having exactly $L$ boxes 
between the $j_{0}$-th row and the $(j_{0}+1)$-st row 
of the Young diagram of $\Jn{\rho}$ (resp., $\omega_{1}$). 
This implies that the skew Young diagram 
$\iota_{L}(\Jn{\rho})/\iota_{L}(\omega_{1})$ is identical to 
the skew Young diagram $\Jn{\rho}/\omega_{1}$. 
Therefore, by the Littlewood-Richardson rule, 
we obtain \eqref{eq:LR2}.

Now, substituting \eqref{eq:LR1} and \eqref{eq:LR2}
into \eqref{eq:step3}, we finally obtain
\begin{align*}
\XLRm{\kappa_{1}}{\rho_{2}}{\kappa}{n+1} & =
\sum_{
 \begin{subarray}{c}
 J \subset [1,\,L], \\[1mm]
 (\omega_{2},\,\omega_{3}) \in \CQ(\rho_{2}), \\[1mm]
 \omega_{1} \in \CR(\rho_{1},\,\omega_{2})
 \end{subarray}}
 \LR^{\iota_{L}(\rho_{1})}_{\iota_{L}(\omega_{1}),\,\omega_{2}} \,
 \LR^{\rho_{2}}_{\omega_{2},\,\omega_{3}} \,
 \LR^{\iota_{L}(\Jn{\rho})}_{\omega_{3},\,\iota_{L}(\omega_{1})} \, 
 \Jn{\sgn}(\rho) \\[3mm]
& = 
\sum_{
 \begin{subarray}{c}
 J \subset [1,\,L], \\[1mm]
 (\omega_{2},\,\omega_{3}) \in \CQ(\rho_{2}), \\[1mm]
 \omega_{1} \in \CR(\rho_{1},\,\omega_{2})
 \end{subarray}}
 \LR^{\rho_{1}}_{\omega_{1},\,\omega_{2}} \,
 \LR^{\rho_{2}}_{\omega_{2},\,\omega_{3}} \,
 \LR^{\Jn{\rho}}_{\omega_{3},\,\omega_{1}} \, 
 \Jn{\sgn}(\rho) \\[3mm]
& = \XLRm{\rho_{1}}{\rho_{2}}{\rho}{n}
\qquad \text{by \eqref{eq:step2a}}.
\end{align*}
This completes the proof of Proposition~\ref{prop:stab}.

\paragraph{Acknowledgments.}
We would like to express our sincere thanks to Jae-Hoon Kwon 
for several helpful comments concerning this work, and also 
for his kindness. 

%======================%
%     BIBLIOGRAPHY     %
%======================%

{\small
\setlength{\baselineskip}{13pt}
\renewcommand{\refname}{References}

}

\end{document}